\documentclass[a4paper,12pt, reqno]{amsart}

\usepackage[margin=18mm]{geometry}
\usepackage{amssymb, amsmath, amsthm, amscd}

\usepackage[T1]{fontenc}
\usepackage{eucal,mathrsfs,dsfont}
\usepackage{color}
\usepackage{mathtools}
\usepackage{hyperref}
\usepackage{mathrsfs}
\usepackage{amsthm}
\usepackage{calc}
\usepackage{xcolor}

\usepackage{hyperref}

\newcommand{\bk}{\color{black}}

\renewcommand{\leq}{\leqslant}

\renewcommand{\le}{\leqslant}
\renewcommand{\ge}{\geqslant}

\def\wt{\widetilde}

\definecolor{mno}{rgb}{0.5,0.1,0.5}

\newcommand{\R}{\mathds R}

\newcommand{\Pp}{\mathds P}
\newcommand{\Ee}{\mathds E}

\newcommand{\I}{\mathds 1}

\newtheorem{theorem}{Theorem}[section]
\newtheorem{lemma}[theorem]{Lemma}
\newtheorem{proposition}[theorem]{Proposition}

\theoremstyle{definition}
\newtheorem{definition}[theorem]{Definition}
\newtheorem{example}[theorem]{Example}
\newtheorem{remark}[theorem]{Remark}

\newcommand{\ac}[1]{\thelemma.{#1}}
\numberwithin{equation}{section}

\begin{document}
\allowdisplaybreaks
\title[symmetric stable processes on horn-shaped regions] {\bfseries
Two-sided Dirichlet heat estimates of symmetric stable processes on horn-shaped regions}
\author{Xin Chen\qquad Panki Kim \qquad Jian Wang}
\thanks{\emph{X.\ Chen:}
   Department of Mathematics, Shanghai Jiao Tong University, 200240 Shanghai, P.R. China. \texttt{chenxin217@sjtu.edu.cn}}
\thanks{\emph{P.\ Kim:}
 Department of Mathematical Sciences and Research Institute of Mathematics, Seoul National University,
Seoul 08826, South Korea.
\texttt{pkim@snu.ac.kr}}
  \thanks{\emph{J.\ Wang:}
   School of Mathematics and Statistics  \&  Fujian Key Laboratory of Mathematical
Analysis and Applications (FJKLMAA) \&  Center for Applied Mathematics of Fujian Province (FJNU), Fujian Normal University, 350007 Fuzhou, P.R. China. \texttt{jianwang@fjnu.edu.cn}}

\date{}

\maketitle

\begin{abstract} In this paper,
we consider  symmetric
$\alpha$-stable processes on (unbounded) horn-shaped regions which are
non-uniformly $C^{1,1}$ near infinity.
By using probabilistic approaches extensively,
we establish two-sided Dirichlet heat estimates
of
such processes for all time.
  The estimates are very sensitive with respect to the reference function
   corresponding to each horn-shaped region. Our results also cover the case that the associated Dirichlet semigroup is not intrinsically ultracontractive.
A striking observation  from our estimates is that, even when
the associated Dirichlet semigroup is  intrinsically ultracontractive,
 the so-called Varopoulos-type estimates do not hold
  for symmetric stable processes on horn-shaped regions.

  \medskip

\medskip

\end{abstract}
\allowdisplaybreaks

\section{Background and main results}\label{section1}

Dirichlet heat kernel is the fundamental solution of the heat
equation with zero exterior conditions, which plays an important
role in the study of Cauchy or Poisson problems with Dirichlet
conditions. While the research on estimates and properties for the Dirichlet
heat kernel of the Laplacian has a long history and fruitful results
(see \cite{GS} and the references therein),
the corresponding work for the fractional Laplacian or more general non-local operators  was powerfully attracted  and extendedly
developed in recent few years.

\smallskip

Let $\Delta^{\alpha/2}:=-(-\Delta)^{\alpha/2}$ be the fractional Laplacian on $\R^d$ with $\alpha\in (0,2)$, which is the infinitesimal generator of the
 (rotationally) symmetric $\alpha$-stable process $X:=\{X_t,t\ge0; \Pp^x,x\in \R^d\}$. The fractional Laplacian $\Delta^{\alpha/2}$
 is a non-local operator and
can be written in the form
\begin{equation}\label{e1-0}
\Delta^{\alpha/2}f(x)=\lim_{\varepsilon \to0}\int_{\{|y-x|\ge \varepsilon\}}(f(y)-f(x))\frac{c_{d,\alpha}}{|y-x|^{d+\alpha}}\,dy,\quad f\in C_c^\infty(\R^d),
\end{equation}
where
$c_{d,\alpha}$ is a positive constant depending
only on $d$ and $\alpha$, and $C_c^\infty(\R^d)$ is the space of smooth functions with compact support in $\R^d$.
Throughout this paper, we denote by $p(t,x,y)$ the heat kernel of the fractional Laplacian $\Delta^{\alpha/2}$ (or equivalently the transition density function of the symmetric $\alpha$-stable
process $X$) on $\R^d$. It is well known (e.g. see \cite{BG, CK}) that $$p(t,x,y)\simeq
t^{-d/\alpha}\wedge \frac{t}{|x-y|^{d+\alpha}} \quad \text{ for all } (t,x,y) \in (0, \infty) \times \R^d \times \R^d.$$
Here and below, we denote $a \wedge b:=\min\{a, b\}$ and $f\simeq g$ if the quotient $f/g$ remains bounded between two positive constants.

\smallskip

For every open subset $D\subset \R^d$, we denote by $X^D$ the subprocess of $X$ killed
upon leaving $D$. The infinitesimal generator of $X^D$ is the Dirichlet fractional Laplacian
$\Delta^{\alpha/2}|_D$ (the fractional Laplacian with zero exterior condition). It is known (see \cite{CS98}) that $X^D$ has
the transition density $p_D(t,x,y)$ with respect to the Lebesgue measure (which is called the Dirichlet heat kernel) that is
jointly continuous on $(0,\infty)\times D\times D$. The first breakthrough on two-sided estimates of the transition density for the Dirichlet fractional Laplacian (which we will call Dirichlet heat kernel estimates later) was done by the second named author jointly with Zhen-Qing Chen and Renming Song in \cite{CKS1}.

\smallskip

To state the main results in \cite{CKS1} explicitly, we first recall the definition of
uniform
$C^{1,1}$ open set.
An open
set $D$ in $\R^d$ with $d\ge 2$ is said to be
$C^{1,1}$ at $z \in \partial D$, if there are  a localization radius $R>0$ and
a constant $\Lambda>0$ (both of them may depend on $z\in D$) such that there
exist a $C^{1,1}$-function $\psi:=\psi_z: \R^{d-1}\to \R$
satisfying
$\psi (0, \dots, 0)= 0$, $\nabla \psi (0, \dots, 0)=(0, \dots, 0)$,
$\| \nabla \psi  \|_\infty \leq \Lambda$ and
$| \nabla \psi (x)-\nabla \psi (y)| \leq \Lambda |x-y|$ for all $x,y\in \R^{d-1}$,
and an orthonormal coordinate system ${\rm CS}_z$ with its origin
at $z$ such that
$$
B(z, R)\cap D=\{ y:=
(y_1,\wt y)
 \mbox{ in } {\rm CS}_z: |y|< R,
y_1>\psi (\wt y)\}.
$$
The pair $(R, \Lambda)$ is called the $C^{1,1}$ characteristics of $D$
at $z$. An open
set $D$ in $\R^d$ with $d\ge 2$ is said to be a (uniform)
$C^{1,1}$ open set, if there exist $R, \Lambda>0$ such that $D$ is $C^{1,1}$ at every $z \in \partial D$ with the same $C^{1,1}$ characteristics $(R, \Lambda)$ of $D$.
The pair $(R, \Lambda)$ is called the characteristics of the
$C^{1,1}$ open set $D$. It is known that any $C^{1,1}$ open set $D$ with the characteristics $(R, \Lambda)$ satisfies the (uniform) interior ball
condition; that is, there exists $r<R$ such that for every $x\in D$ with $\delta_D(x)\le r$,
it holds that
$B(\xi_{x,r}^*,r)\subset D$, where $\delta_D(x)$ is the  Euclidean
distance between $x$ and $D^c$, and
$\xi_{x,r}^*:=z_x+r(x-z_x)/|x-z_x|$ with $z_x\in \partial D$ such that $|x-z_x|=\delta_D(x)$.

Let $D$ be a $C^{1,1}$ open subset of $\R^d$.
It was shown in \cite[Theorem 1.1]{CKS1} that
\begin{itemize}
\item[(i)] For every $T>0$, on $(0,T]\times D\times D$,
\begin{equation}\label{hk-s}p_D(t,x,y)\simeq p(t,x,y) \left(\frac{\delta_D(x)^{\alpha/2}}{\sqrt{t}}\wedge1\right)\left(\frac{\delta_D(y)^{\alpha/2}}{\sqrt{t}}\wedge1\right).\end{equation}
\item[(ii)] Suppose in addition that $D $ is bounded.
Then, for  every $T>0$, on $(T,\infty]\times D\times D$,
 \begin{equation}\label{hk-s1}p_D(t,x,y)\simeq \delta_D(x)^{\alpha/2} \delta_D(y)^{\alpha/2} e^{-\lambda_D t},\end{equation}
where $\lambda_D> 0$ is the smallest eigenvalue of the Dirichlet fractional Laplacian
$(-\Delta)^{\alpha/2}|_D$.
\end{itemize}

 (i) says that, until any finite time, the Dirichlet heat
kernel $p_D(t,x,y)$ is comparable with the global heat kernel $p(t,x,y)$ multiplied by
some weighted functions $\frac{\delta_D(x)^{\alpha/2}}{\sqrt{t}}\wedge1$ and $\frac{\delta_D(y)^{\alpha/2}}{\sqrt{t}}\wedge1$, which are determined by the
dependency
 between time and position of the points $x,y\in D$. The uniform $C^{1,1}$-property of the open set $D$ plays a key role in the proof of (i).
On the other hand,
the estimate of $p_D(t,x,y)$ for large time
given in (ii) is based on the result
(i) and  the so-called intrinsic ultracontractivity  of $p_D(t,x,y)$, i.e.,
$p_D(t,x,y)\le c
\phi_1(x)\phi_1(y)e^{-\lambda_D t}$, where
$\phi_1$ is the ground state (i.e., the positive
eigenfunction corresponding to the first eigenvalue $\lambda_D$)
and satisfies that $\phi_1(x)\simeq \delta_D(x)^{\alpha/2}$.
 The notion of intrinsic ultracontractivity was
first introduced by Davies and Simon in \cite{DS}.

The idea and the approach in \cite{CKS1} later were extensively  adopted to study Dirichlet heat kernel estimates for  censored stable-like processes in \cite{CKS-1}, for relativistic stable processes in \cite{CKS0}, for $\Delta^{\alpha/2}+\Delta^{\beta/2}$ in \cite{CKS-2}, for $\Delta+\Delta^{\alpha/2}$ in  \cite{CKS-3}, for subordinate Brownian motions with Gaussian components in \cite{CKS3}, for unimodal L\'evy processes in \cite{BK}, for a large class of symmetric pure jump Markov processes dominated by isotropic unimodal L\'evy processes with weak scaling conditions in \cite{GKK, KK}, and so on.

As mentioned above, the uniform $C^{1,1}$-property of $D$ is crucial for the estimate \eqref{hk-s}. When $D$ has lower regularity, \eqref{hk-s} may not be available but
Dirichlet heat kernel estimates can be established in terms of the survival probability $\Pp^x(\tau_D>t)$ instead of $\frac{\delta_D(x)^{\alpha/2}}{\sqrt{t}}\wedge1$, where $\tau_D$ is the first exit time  from $D$ of the process $X$, i.e., $\tau_D=\inf\{t>0: X_t\notin D\}$.
That is, in these cases one would expect that for any $T>0$, on $(0,T]\times D\times D$,
\begin{equation}\label{e:va}p_D(t,x,y)
\asymp p(t,x,y)\Pp^x(\tau_D>t)\Pp^y(\tau_D>t), \quad x,y\in D,\,\,0<t\le
T.\end{equation}\eqref{e:va} are called the Varopoulos-type estimates in the literature, and they can be traced back to the paper \cite{Var} by Varopoulos, where \eqref{e:va} are proved to be satisfied for Dirichlet heat kernels of a divergence and
nondivergence form elliptic operator (even with time-dependent coefficients)  on bounded Lipschitz domains.
Nowadays, \eqref{e:va} have been obtained for a quite large class of discontinuous processes.
See
\cite[Theorem 1]{BGR0} for Dirichlet heat kernel estimates of symmetric $\alpha$-stable process when $D$ is $\kappa$-fat (including domain
above the graph of a Lipschitz function), and see
\cite[Theorem 1.3 and Corollary 1.4]{CKS} and  \cite[Theorems 2.22 and 2.23]{CKSV}  for the corresponding results for  rotationally symmetric L\'evy processes and more general jump processes with critical killings, respectively.
On the other hand, as indicated above,
 the estimate \eqref{hk-s1} for large time
  is a direct consequence of the intrinsic ultracontractivity of the associated Dirichlet semigroup, which is satisfied when $C^{1,1}$ open set $D$ is bounded. Indeed,
  the intrinsic ultracontractivity holds for symmetric $\alpha$-stable process
  on any bounded open set $D$; see \cite{K,CWi14}.

\smallskip

 When $D$ is unbounded, \eqref{hk-s1} would fail. For example, it was proved in \cite[Theorem 1.2]{CT} that when $D$ is a half-space-like $C^{1,1}$ open set of $\R^d$, \eqref{hk-s} holds for all $(t,x,y)\in (0,\infty)\times D\times D$. See \cite{CT} for more details and \cite{CKS-u1, CKS-u2, CKS-u3, CKS-u4, Ki} for related developments on other (general) symmetric jump processes.

\ \

\noindent {\bf Notation}\,\, We will use the symbol ``$:=$'' to denote a definition,
which is read as ``is defined to be''.
In this paper,
for $a,b\in \R$ we denote $a\wedge b:=\min\{a,b\}$ and $a\vee b:=\max\{a,b\}$.  We also use the convention $0^{-1}=+\infty$.
We write
$h(s)\simeq f(s)$,
if there exist constants $c_{1},c_{2}>0$ such that
$
c_{1}f(s)\leq h(s)\leq c_{2}f(s)
$
for the specified range of the argument $s$.
Similarly, we write $h(s)\asymp f(s)g(s)$,
if there exist constants $c_{1},c_{2},c_{3},c_{4}>0$ such that
$
f(c_{1}s)g(c_{2}s)\leq h(s)\leq f(c_{3}s)g(c_{4}s)
$
for the specified range of $s$.
Upper case letters
 with subscripts $C_i$, $i=0,1,2,  \dots$, denote constants
that will be fixed throughout the paper.
Letters
 $C_{i,j,\cdot}$,
 $C_{i,j}$, $c_{i,j,\cdot}$, $i,j=0,1,2,  \dots$ with subscripts denote constants from Lemma $i.j$ or Proposition $i.j$ or the equation $(i,j)$, which are also fixed throughout the paper.
Lower case letters $c$'s without subscripts denote strictly positive
constants  whose values
are unimportant and which  may change even within a line, while values of lower case letters with subscripts
$c_i, i=0,1,2,  \dots$, are fixed in each proof,
and the labeling of these constants starts anew in each proof.
$c_i=c_i(a,b,c,\ldots)$, $i=0,1,2,  \dots$, denote  constants depending on $a, b, c, \ldots$.
The dependence on the dimension $d \ge 2$
and the index $\alpha\in (0,2)$
may not be mentioned explicitly.
Without any mention, the constants $C, C_{\cdot}, C_i, C_{i,j}, C_{i,j,\cdot}, c, c_{\cdot}, c_i,  c_{i,j,\cdot}$ are independent of $x,y\in D$ and $t>0$.
For $x \in D$ we use $z_x$ to denote a point $z_x$ in $\partial D$ such that $|x-z_x|=\delta_D(x)$.
For a Borel subset $V$ in $\R^d$, $|V|$ denotes  the Lebesgue measure of $V$.
We use the convention that $\inf\emptyset=\infty$ and $\sup \emptyset =0$.
\subsection{Setting and main result}\label{subsection1}
The aim of this paper is to study two-sided  Dirichlet heat kernel estimates of symmetric
$\alpha$-stable processes on horn-shaped regions (see below for the definition). We emphasis that horn-shaped regions are
non-uniformly $C^{1,1}$ near infinity
and usually unbounded, so the corresponding Dirichlet heat kernel estimates go beyond the scope of all the papers quoted above.

 In fact,  due to
 the non-uniform $C^{1,1}$-property
 of horn-shaped regions, new ideas and much more efforts are required to achieve  the sharp Dirichlet heat kernel estimates. Furthermore,
  on the one hand, our two-sided Dirichlet heat kernel estimates are
 for full time.
On the other hand,
our results cover the case that the associated Dirichlet semigroup is not intrinsically ultracontractive. To the best of our knowledge, this is the first result on explicit estimates for Dirichlet heat kernel on non-uniformly $C^{1,1}$ and unbounded domains. Even we did not find the corresponding results for Brownian motions in the literature.

\ \

 Throughout our paper,
  we always let  $f:\R\to (0,\infty)$ be a continuous function satisfying  the following conditions:
\begin{align}
&\text{$f(-t)\equiv f(0)$ for $t>0$ and $f\in C^{1,1} ((0,\infty))$}; \label{e:f1}\\
&\text{$f$ is non-increasing on $(0,\infty)$ with $\lim_{r\to\infty}f(r)=0$};\label{e:f2}\\
&\text{for any $c\ge1$, $f(cs)\simeq f(s)$ on $\R$.
\label{e:f3}}
\end{align}
Note that the above
properties imply that $f(s-2)  \le c f(s) $ for all $s$.
The function $f$ is served as the reference function for the horn-shaped region, which will be defined explicitly below.

\smallskip

Let $d\ge 2$, and write $x=(x_1,\tilde x)\in \R^d$, where $\tilde x=(x_2,x_3,\cdots, x_d)$.
For any $a>0$, denote
$D^a_f:=\{x\in \R^d:
x_1>a, |\tilde x|< f(x_1)\}.$

\begin{definition}
For any $d\ge2$, let $D$ be an open set of $\R^d$.

\noindent (1) We say that $D$ is a horn-shaped region with the
reference function $f$, if there exists $M  \ge 2 f(0)$ such that
\begin{itemize}
\item[(i)]
$D \cap  \{x\in \R^d:
x_1<M\}$ is bounded;
\item[(ii)] $\{x\in D:
x_1>M\}=D^{M}_f$;
\item[(iii)]
there exist $c_* \in (0, 1]$ and  $\Lambda >0$ such that
{\it for all } $x \in D_f^{M}$,
$D$ is $C^{1,1}$ at $z_x\in \partial D_f^{M}$ with the  characteristics
$(c_*f(x_1), \bk \Lambda)$.
\end{itemize}

\noindent
(2) We say that $D$ is a horn-shaped $C^{1,1}$ region with the
reference function $f$, if  $D$ is a horn-shaped region with the
reference function $f$ and there exist $c_* \in (0, 1]$ and  $\Lambda >0$ such that for all $x \in D$,
$D$ is $C^{1,1}$ at $z_x\in \partial D$ with the  characteristics
$(c_*f(x_1), \bk \Lambda)$.
\end{definition}
See Figure \ref{pic:1} for a horn-shaped $C^{1,1}$ region
$D$ when $d=2$.
\begin{figure}[!h]
	\centering
	\includegraphics[width=0.7\columnwidth]{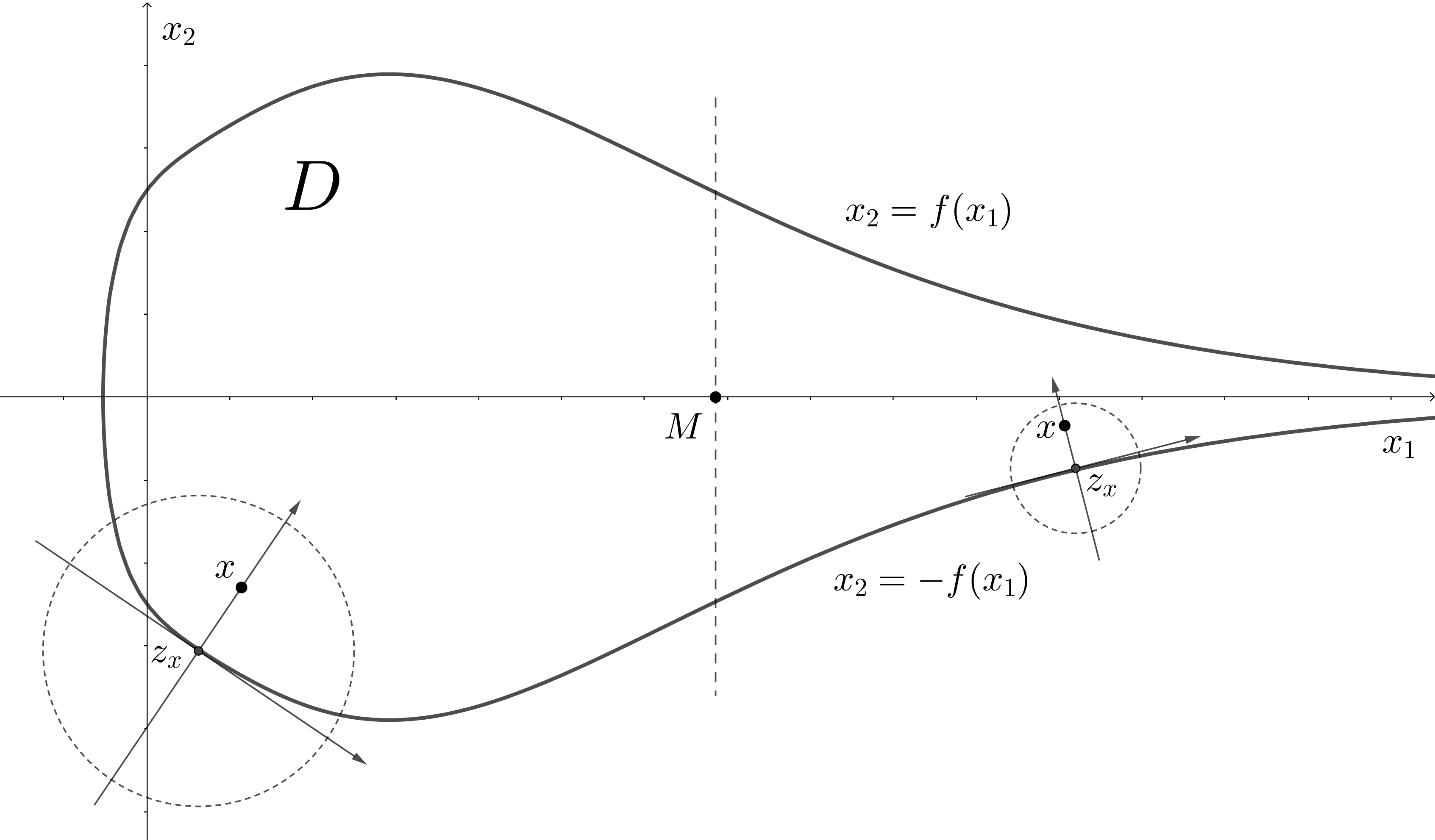}
	\caption{A horn-shaped $C^{1,1}$ region in $\R^2$.}
	 \label{pic:1}
\end{figure}
\begin{remark}\label{r:hc11}
It is easy to see that, for every horn-shaped region $D$ with the
reference function $f$, there exist horn-shaped $C^{1,1}$ regions $U_1$ and  $U_2$ with the
same reference function $f$ such that $U_1 \subset D  \subset U_2$, and $\delta_{U_1}(x)= \delta_{U_2}(x)=  \delta_{D}(x)$ for $x \in D^{M}_f$ with some constant $M>0$.
\end{remark}

For both mathematical and physical backgrounds
on the study of analytic properties related to horn-shaped regions,  readers
are referred to
\cite{B, BB, BD, Be, CKW, CL, DS, LPZ}.
We note that the properties \eqref{e:f1}--\eqref{e:f3} of the reference function $f$ essentially are also imposed in \cite{B,BB,BD, CL,DS, LPZ}, when explicit two-sided estimates for Dirichlet eigenfunctions for horn-shaped regions are concerned.

\smallskip

In the following, we fix a $C^{1,1}$ horn-shaped region $D$ with the
reference function $f$, and set \begin{align}\label{e:DPsi}
\Psi(t,x):=\frac{\delta_D(x)^{\alpha/2}\left(f(x_1)^{\alpha/2}\wedge t^{1/2}\right)}
 {t\wedge1}\wedge 1,\quad x\in D, t>0.
\end{align}
The function $\Psi(t,x)$ will be used to describe the behavior of Dirichlet heat kernels near the boundary of $D$.
Note that, by the definition of $D$, there exists  a  constant $c_0>0$ such that $\delta_D(x)\le c_0f(x_1)$ for all $x\in
D$.
Thus,  there exist $c_1,c_2>0$ such that for all $x\in D$ and $t>0$,
\begin{align*}
\Psi(t,x)\simeq
& \begin{cases}
\displaystyle 1 &\text{ if } \delta_D(x)\ge c_1t^{1/\alpha}; \\
\displaystyle
\frac{\delta_D(x)^{\alpha/2}}{\sqrt{t}}   & \text{ if } \delta_D(x)\le
c_1t^{1/\alpha}\le c_2f(x_1); \\
\displaystyle
  \frac{\delta_D(x)^{\alpha/2}
f(x_1)^{\alpha/2}}{t\wedge1} & \text{ if }
c_1 t^{1/\alpha}\ge c_2 f(x_1).\end{cases}
\end{align*}
We also set
\begin{equation}\label{e1-1}
\phi(x):=\frac{\delta_D(x)^{\alpha/2}f(x_1)^{\alpha/2} }{(1+|x|)^{d+\alpha}},\quad x\in D,
\end{equation}
which is comparable to  the ground state
 of Dirichlet fractional Laplacian $(-\Delta)^{\alpha/2}|_D$
 for the horn-shaped region $D$;
see \cite[Theorem 1 and Proposition 1]{Kw} or \cite[Theorem 6.1]{CKW} for more details.

For
any fixed
constant $c>0$, let
$t_0(x):=t_0(c, x)\in (0,\infty)$, which is defined for all $x\in D$, such that
\begin{align}
\label{e:t0}
e^{-ct_0(x) f(x_1)^{-\alpha}}= t_0(x)(1+|x|)^{-(d+\alpha-1)}, \quad x \in D.
\end{align}
Since the function $t\mapsto e^{-c_0f(x_1)^{-\alpha}t}$ is continuous and strictly
decreasing on $(0,\infty)$ with values on $(0,1)$
 and the function $t\mapsto
t(1+|x|)^{-(d+\alpha-1)}$ is continuous and strictly increasing on $(0,\infty)$ with values on
$(0,\infty)$,
$t_0(x)$ exists and is unique for all $x\in D$.
The functions $e^{-c_0f(x_1)^{-\alpha}t}$ and $t(1+|x|)^{-(d+\alpha-1)}$ come from estimates of
the survival probability $\Pp^x(\tau_D> t)$; see Lemma \ref{e:lm2} below.
One can see that there is a constant $c_1>0$ such that for all $x\in D$,
$f(x_1)^{\alpha}\le c_1 t_0(x).$
Usually it is not easy to obtain the explicit value of $t_0(x)$;
however, we possibly can get explicit estimates of $t_0(x)$
for all $x \in D$
under some mild assumption on the reference function $f$. For example, if $f(r)\ge
c(1+r)^{- p}$ for some constants $c$ and $p>0$, then
$t_0(x)\simeq f(x_1)^\alpha
\log(2+|x|)$ for all $x\in D$.

\ \

The main result of this paper is as follows.

\begin{theorem}\label{Main} Suppose that $d\ge2$ and  $D$ is a $C^{1,1}$ horn-shaped region of $\R^d$
associated  with
the reference function $f$ satisfying \eqref{e:f1}, \eqref{e:f2} and \eqref{e:f3}.
Let
$p_D(t, x, y)$  be the transition density of killed symmetric $\alpha$-stable process $X^D$ with $\alpha\in (0,2)$.
Then,
there exist constants $c_{1.3.0}, c_{1.3.1}>0$
such that
the following two statements hold with
$t_0( \cdot ):=t_0(c_{1.3.0}, \cdot)$.
\smallskip

\noindent
{\rm(1)}
 For any $x,y\in D$ and any $0<t\le c_{1.3.1}(f(x_1) \vee f(y_1))^\alpha\le 1$,
\begin{equation}\label{e:hk1} p_D(t,x,y)\simeq  p(t,x,y)\Psi(t,x)\Psi(t,y).  \end{equation}

\smallskip

\noindent
{\rm(2)}
Suppose in addition that $f(s)\ge c(1+s)^{-p}$ on $(0,\infty)$ for some $c,p>0$, and the function $s \mapsto f(s)^{\alpha}\log(2+s)$ is comparable to some
monotone function $g$ on $(0,\infty)$
$($i.e. $g(s)\simeq f(s)^\alpha \log(2+s)$$)$.
\begin{itemize}\item[(i)] If  $g$ is
non-increasing on $(0,\infty)$ so that $\lim\limits_{s\to\infty} g(s)=0$, then
there exist positive constants
$c_{1.3.i}$ $(2\le i \le 10)$
such that
 for any $x,y\in D$ and any $ c_{1.3.1} (f(x_1)\vee f(y_1))^\alpha\le t\le c_{1.3.2} (t_0(x)\vee t_0(y))
( \le c_{1.3.2} \|t_0\|_\infty<\infty )$,
\begin{equation}\label{e:hk2}\begin{split}\qquad p_D(t,x,y)
 \asymp p(t,x,y) \Psi(t,x)\Psi(t,y) \exp\left\{-t (f(x_1)\vee f(y_1))^{-\alpha} \right\};\end{split}\end{equation}
and for any $x,y\in D$ and any $t\ge c_{1.3.2} (t_0(x)\vee t_0(y))$,
\begin{equation}\label{e:hk3}\begin{split}
\qquad&c_{1.3.3} \phi(x)\phi(y) \max\Big\{\int_{0}^{c_{1.3.4} s_1(c_{1.3.5}t)}
f(s)^{d-1}e^{-c_{1.3.6}tf(s)^{-\alpha}}\,ds, e^{-c_{1.3.6} t}\Big\}  \\
\qquad&\le p_D(t,x,y) \\
\qquad&\le
c_{1.3.7}\phi(x)\phi(y) \max\Big\{\int_{0}^{c_{1.3.8}s_1(c_{1.3.9}t)}
f(s)^{d-1}e^{-c_{1.3.10}tf(s)^{-\alpha}}\,ds, e^{-c_{1.3.10} t}\Big\},
\end{split}\end{equation}
where $s_1(t)=g^{-1}(t)\vee 2$
and $g^{-1}(t)=\inf\{s\ge0: g(s)\le t\}$ for $t>0$.
\item[(ii)] If  $g$ is
non-decreasing on $(0,\infty)$ so that $\lim\limits_{s\to\infty} g(s)>0$, then
there exists  a constant $c_{1.3.2}>0$ such that
 for any $x,y\in D$ and any $ c_{1.3.1}((f(x_1) \vee f(y_1))^\alpha\le t\le c_{1.3.2} (t_0(x)\wedge t_0(y))$,
\begin{equation}\label{e:hk4}\qquad p_D(t,x,y) \asymp p(t,x,y) \Psi(t,x)\Psi(t,y) \exp\{-t (f(x_1)\vee f(y_1))^{-\alpha} \};\end{equation} and for any $x,y\in D$
and any $t\ge c_{1.3.2} (t_0(x)\wedge t_0(y))
(\ge c_{1.3.2} \inf_{z\in D} t_0(z) >0)$,
\begin{equation}\label{e:hk5}\qquad p_D(t,x,y)\asymp \phi(x)\phi(y) e^{-t}.\end{equation}
 \end{itemize}
 \end{theorem}

\begin{remark}\label{r1-4}
Let us give some remarks on Theorem \ref{Main}.

\noindent
(i)
 It is clear from Remark \ref{r:hc11} and
the proof of Theorem \ref{Main} that,
for horn-shaped region $D$
(not necessarily $C^{1,1}$ near the origin),
   the conclusions of Theorem \ref{Main} still hold true for all $x,y\in D$ with $|x|\vee |y|$ large enough.

\noindent
(ii)
When
$0<t\le  c_1(f(x_1)^\alpha \wedge f(y_1)^\alpha)$,
$p_D(t,x,y)$ satisfies \eqref{e:hk1}, which is of the same form as \eqref{hk-s}; that is, $p_D(t,x,y)$ is comparable with the global heat kernel $p(t,x,y)$ multiplied by
 weighted functions $\Psi(t,x)$ and $\Psi(t,y)$, which are comparable to
 $\frac{\delta_D(x)^{\alpha/2}}{\sqrt{t}}\wedge1$ and $\frac{\delta_D(y)^{\alpha/2}}{\sqrt{t}}\wedge1$ respectively.
  This assertion is reasonable since $C^{1,1}$ horn-shaped region $D$ enjoys the \lq\lq
  semi-uniform\rq\rq\, interior ball condition in the sense that for any $x\in D$ and $r\in (0, c_*f(x_1))$ (with possibly small $c_*$), $B(\xi_{x,r}^*,r)\subset D$ with $\xi_{x,r}^*=z_x+r(x-z_x)/|x-z_x|$.

\noindent
(iii) For  $t\ge  c_1(f(x_1)^\alpha \vee f(y_1)^\alpha)$, estimates for $p_D(t,x,y)$ heavily rely on the asymptotic
 property of the reference function $f$. According to \cite[Theorem 5]{Kw}, under assumptions of case (i) in (2) the associated Dirichlet
 semigroup $(P_t^D)_{t\ge0}$ is intrinsically ultracontractive. Note that $s_1(t)=2$ for large $t$ under assumptions of case (i) in (2). Hence, similar to \eqref{hk-s1}, the estimate indicated in \eqref{e:hk3} for $t\ge 1$ essentially is a direct consequence of the intrinsic ultracontractivity of $(P_t^D)_{t\ge0}$. However, when $c_1(f(x_1)^\alpha \vee f(y_1)^\alpha)\le
 t\le 1$, estimates for $p_D(t,x,y)$ are much more delicate.

\noindent
(iv)
 It will be shown in Lemma \ref{e:lm2} that the
following upper bound on survival probability
 holds true: for any $x\in D$ and $t>0$,
 \begin{equation}\label{e:eeffcc}
 \Pp^x(\tau_D>t)\le c_1\Psi(t,x)\min\Big\{e^{-c_2f(x_1)^{-\alpha}t}+t(1+|x|)^{-(d+\alpha-1)},
e^{-c_2t}\Big\}.
 \end{equation}
In particular, when
$t = T_0:=c_{1.3.2} \|t_0\|_\infty<\infty$
and $10T_0<|x| \le 2|y|$,
\eqref{e:eeffcc} implies that
$$
p(T_0,x,y)\Pp^x(\tau_D>T_0)\Pp^y(\tau_D>T_0)\le c_3(T_0)
\frac{\phi(x)\phi(y) |x|  }{|y|^{d+\alpha-1}}.
$$
On the other hand,
\eqref{e:hk3} implies that
$
p_D(T_0,x,y) \asymp
\phi(x)\phi(y) $.
Therefore, the so-called Varopoulos-type estimates  \eqref{e:va}
 do not hold true
under assumptions of case (i) in (2), which is different from \cite[Theorem 1.1]{CKS1} and
 \cite[Theorem 1]{BGR0}.

\noindent
(v)
 Under assumptions of case (ii) in (2), the associated Dirichlet
 semigroup $(P^D_t)_{t\ge0}$ is not intrinsically ultracontractive, see
 also \cite[Theorem 5]{Kw}. Though \eqref{e:hk4} is of the same form as that for \eqref{e:hk2},
 the ranges of time variable are different; that is, $c_2 (t_0(x)\wedge t_0(y))\ge 1$ in \eqref{e:hk4}, while $c_2 (t_0(x)\vee t_0(y))\le 1$ in \eqref{e:hk2}.
 Also
 by this reason, the estimates \eqref{e:hk3} and \eqref{e:hk5} are different too, even both of them enjoy the same form (by neglecting constants in the exponential term) when $t\to\infty$.
\end{remark}

The proof of Theorem \ref{Main} is
completely different from
those in \cite{CKS1} and  \cite{CT},  where two-sided Dirichlet heat kernel estimates for fractional Laplacians in
 uniformly $C^{1,1}$ open sets and half-space-like open sets
 were established respectively.
For example,
because of the non-uniformity on  $C^{1,1}$ characteristics,  the boundary Harnack principle
can not be
applied
to $C^{1,1}$ horn-shaped regions,
and so
the approach of \cite[Theorem 1.1 (i)]{CKS1} does not work in the present setting.
In order to obtain
Dirichlet heat kernel estimates of horn-shaped
 regions, we need to take into accounts
 carefully the interaction between jumping kernel of symmetric $\alpha$-stable processes and the characterization (heavily depending on the reference function $f$) of the horn-shaped region.
 Roughly speaking, the proof of Theorem \ref{Main} is split into three cases according to different ranges of time and space.
 (1) When $0<t\le  c_1  (f(x_1)\vee f(y_1))^\alpha$,
we make use of the Chapman-Kolmogorov equation and a general formula for upper bounds of Dirichlet heat kernels
(see \cite[Lemma 1.10]{BK}, \cite[Lemma 5.1]{GKK} and \cite[Lemma 3.1]{CKS3}).
 Note that, in this case the estimates for exit probability (see Lemma \ref{l5-4-2})  are different from those implied by
 \eqref{hk-s} when  $c_1  (f(x_1)\wedge f(y_1))^\alpha \le t\le  c_1  (f(x_1)\vee f(y_1))^\alpha$.
(2) When
 $c_1(f(x_1)\vee f(y_1))^\alpha\le t\le c_2(t_0(x)\vee t_0(y))$ or $c_1(f(x_1)\vee f(y_1))^\alpha\le t\le c_2(t_0(x)\wedge t_0(y))$,
 we will adopt the chain argument to
derive lower bounds and
 apply the split technique
 combined with the survival probability
 \eqref{l3-1-0}
 to obtain upper bounds.
 In particular, in arguments for both cases above,
 instead of the boundary Harnack principle,
we make use of the L\'evy system. (3) When $t\ge c_2(t_0(x)\vee t_0(y))$ or $t\ge c_2(t_0(x)\wedge t_0(y))$,
the dominant behaviour (with the largest probability)
of the killed process taking time $t$ from $x$ to $y$ is that, the process jumps form $x$ to the origin, and then jumps to $y$ after spending more than $t/2$
at a
neighborhood of origin or
at another neighborhood inside $D$ with the largest survival probability.
This gives us the intuitive meanings of \eqref{e:hk3} and \eqref{e:hk5}.
In this case, lower bounds are derived by using assertions in cases (1) and (2); however, the proofs of upper bounds are much more involved. In particular, we will use the iteration arguments based on the survival probability.

\subsection{ Relation with intrinsic ultracontractivity}\label{section1.2}
Recall that
in the present setting the Dirichlet semigroup  $(P^D_t)_{t\ge0}$ is
intrinsically ultracontractive,
if for every $t>0$ there is a constant $C_{D, t}>0$ such that
\begin{equation}\label{e:iu}p_D(t,x,y)\le C_{D,t}\phi(x)\phi(y),\quad x,y\in D, \end{equation}
where $\phi$ is defined by \eqref{e1-1}
that is comparable with the ground state of $(P^D_t)_{t\ge0}$.

The intrinsic ultracontractivity of Markov semigroups (including Dirichlet semigroups and Feyman-Kac semigroups) has been intensively established for various L\'evy type processes. For more details, see \cite{CWi14, CW16} and the references therein.
The intrinsic ultracontractivity and two-sided estimates of ground state for
symmetric $\alpha$-stable processes and more general symmetric jump processes on unbounded open sets were
investigated in \cite{Kw} and \cite{CKW}, respectively.
We note that the two-sided Dirichlet heat kernel estimates are much more complex than estimates of ground state.
Informally, to obtain
Dirichlet heat kernel estimates
we need to consider the relationship between time and space carefully;
for ground state estimates we only just take time $t=1$ and make use of estimates for $p(1,x,y)$;
see \cite[Sections 5 and 6]{CKW}.

In the following, we deduce
 explicit estimates for the intrinsic ultracontractivity under assumptions in (i) of (2) in Theorem \ref{Main}, by directly applying two-sided Dirichlet heat kernel estimates.  Recall that $g(s)\simeq
f(s)^{\alpha}\log(2+s)$.

\begin{proposition}\label{p:IU}
Under assumptions in {\rm(i)} of {\rm(2)} in Theorem $\ref{Main}$, \eqref{e:iu} holds with
$$C_{D,t}=
c_{\ac{2}}
\begin{cases}
\displaystyle t^{-2-d/\alpha}(1+g^{-1}(c_{\ac{3}} t))^{2d+2\alpha},&\,\,0<t\le c_{\ac{1}}(t_0(x)\vee t_0(y));\\
\displaystyle \max\Big\{\int_{0}^{c_{\ac{4}}s_1(c_{\ac{5}}t)}
f(s)^{d-1}e^{-c_{\ac{6}}tf(s)^{-\alpha}}\,ds, e^{-c_{\ac{7}} t}\Big\},&\,\,t> c_{\ac{1}}(t_0(x)\vee t_0(y)).
\end{cases}
$$
\end{proposition}

\begin{proof}
According to \eqref{e:hk1} and \eqref{e:hk2}, there are constants $c_0, c_1>0$ such that
$p_D(t,x,y)\le c_1p(t,x,y)\Psi(t,x)\Psi(t,y)$ for any $x,y\in D$ and $0<t\le c_0(t_0(x)\vee t_0(y))\le 1$.

In the following, without loss of generality, we
may assume that $x,y\in D$ with $x_1\ge y_1$.
According to the
non-increasing property of the function $g$ and
$\lim_{s\to\infty}g(s)=0$ as well as $t_0(y)\simeq g(|y|)$,  $t\le c_2t_0(y)$ for some $c_2>0$ implies that
$|y|\le c_3g^{-1}(c_4 t)$. In particular,
\begin{equation}\label{e:key}
(1+|y|)^{d+\alpha}\le
\left(1+c_3g^{-1}(c_4 t)\right)^{d+\alpha}.
\end{equation}
Thus, if
$|y|/2\le |x|\le 2|y|$, then, for $0<t\le c_0(t_0(x)\vee t_0(y))\le 1$,
\begin{align*}p_D(t,x,y)\le &c_5 t^{-d/\alpha}
\frac{\delta_D(x)^{\alpha/2}(f(x_1)^{\alpha/2}\wedge t^{1/2})}{t}
\frac{\delta_D(y)^{\alpha/2}
(f(y_1)^{\alpha/2}\wedge t^{1/2})}{t}\\
\le& c_{6} t^{-2-d/\alpha}(1+g^{-1}(c_7 t))^{2d+2\alpha} \frac{\delta_D(x)^{\alpha/2}
f(x_1)^{\alpha/2}}{(1+|x|)^{d+\alpha}}\frac{\delta_D(y)^{\alpha/2}f(y_1)^{\alpha/2}}{(1+|y|)^{d+\alpha}}\\
=& c_{6} t^{-2-d/\alpha}(1+g^{-1}(c_7 t))^{2d+2\alpha}\phi(x)\phi(y),\end{align*} where in the second inequality we used the fact that
$|y|/2\le |x|\le 2|y|$ and \eqref{e:key};
if $|x|\ge 2|y|$, then, for $0<t\le c_0(t_0(x)\vee t_0(y))$, we can argue as follows
\begin{align*}p_D(t,x,y)\le &   \frac{c_{8} t}{(1+|x|)^{d+\alpha}}
\frac{\delta_D(x)^{\alpha/2}(f(x_1)^{\alpha/2}\wedge t^{1/2})}{t}
\frac{\delta_D(y)^{\alpha/2}
(f(y_1)^{\alpha/2}\wedge t^{1/2})}{t}\\
\le& c_{9} t^{-1}(1+ g^{-1}(c_{10}t))^{d+\alpha} \frac{\delta_D(x)^{\alpha/2}
f(x_1)^{\alpha/2}}{(1+|x|)^{d+\alpha}}\frac{\delta_D(y)^{\alpha/2}f(y_1)^{\alpha/2}}{(1+|y|)^{d+\alpha}}\\
=& c_{9} t^{-1}(1+g^{-1}(c_{10}t))^{d+\alpha}\phi(x)\phi(y),\end{align*} where the first inequality follows from
the fact that $|x|\ge2|y|$, and
the second inequality is due to \eqref{e:key}.
Similarly, we can prove that if $|x|\le |y|/2$, then, for $0<t\le c_0(t_0(x)\vee t_0(y))$,
$$p_D(t,x,y)\le c_{11} t^{-1}(1+g^{-1}(c_{12}t))^{d+\alpha}\phi(x)\phi(y).$$

Combining all the
estimates above with \eqref{e:hk3}, we can obtain that \eqref{e:iu} holds for all $x,y\in D$ and $t>0$ with the desired estimates for $C_{D,t}$.
 \end{proof}
 We would like to mention that the arguments above (in particular, \eqref{e:key}) fail, under assumptions in (ii) of (2) in Theorem \ref{Main}, i.e., when the function $g(s)$ is
 non-decreasing on $(0,\infty)$.

\subsection{A toy example}

In this part,
we present the following example to illustrate how powerful Theorem \ref{Main} is.

\begin{example}\label{exam} Let $f(s)=\log^{-\theta}(2+s)$ with $\theta>0$ for all $s\in[0,\infty)$. For any $x,y\in D$, set $t_1(x,y)=\log^{-\theta \alpha}(e+(|x|\wedge |y|))$ and $t_2(x,y)= \log^{-(\theta \alpha-1)}(e+(|x|\wedge |y|)).$  Then, we have the following two statements.

\noindent
(i) Assume that $\theta>1/\alpha$. Then,
there exist positive constants $c_{1.6.1}$, $c_{1.6.2}$ and $c_{1.6.3}$ such that
for all $x,y\in D$,
\begin{align*}
&p_D(t,x,y)\asymp  \\
 &\begin{cases}\displaystyle p(t,x,y)
\bigg(\frac{\delta_D(x)^{\alpha/2}\left(\log^{-\theta\alpha/2}(e+|x|)\wedge t^{1/2}\right)}
 {t}\wedge 1\bigg)\bigg(\frac{\delta_D(y)^{\alpha/2}\left(\log^{-\theta\alpha/2}(e+|y|)\wedge t^{1/2}\right)}
 {t}\wedge 1\bigg)\\
\qquad\qquad \qquad\qquad \qquad\qquad\qquad\qquad\qquad\qquad\qquad\qquad \text{for all }0<t\le c_{1.6.1}t_1(x,y); \\[2pt]
\displaystyle p(t,x,y) \frac{\delta_D(x)^{\alpha/2}\log^{-\theta\alpha/2}(e+|x|)}{t}\frac{\delta_D(y)^{\alpha/2}\log^{-\theta\alpha/2}(e+|y|)}{t} \exp(-t\log^{\theta\alpha}(e+(|x|\wedge|y|)))&\\
\qquad\qquad \qquad\qquad \qquad\qquad\qquad\qquad\qquad\qquad\qquad\qquad \text{for all } c_{1.6.1}t_1(x,y)< t\le
c_{1.6.2} t_2(x,y);\\[2pt]
\displaystyle  \frac{\delta_D(x)^{\alpha/2}\log^{-\theta\alpha/2}(e+|x|)}{(1+|x|)^{d+\alpha}}\frac{\delta_D(y)^{\alpha/2}
\log^{-\theta\alpha/2}(e+|y|)}{(1+|y|)^{d+\alpha}}
\exp(t^{-1/(\theta\alpha-1)}),& \\
\qquad\qquad \qquad\qquad \qquad\qquad\qquad\qquad\qquad\qquad\qquad\qquad \text{for all } c_{1.6.2}t_2(x,y)< t\le c_{1.6.3};\\[2pt]
\displaystyle
\frac{\delta_D(x)^{\alpha/2}\log^{-\theta\alpha/2}(e+|x|)}{(1+|x|)^{d+\alpha}}
\frac{\delta_D(y)^{\alpha/2}\log^{-\theta\alpha/2}(e+|y|)}{(1+|y|)^{d+\alpha}}
\exp(-t), \\
\qquad\qquad  \qquad\qquad\qquad\qquad\qquad\qquad\qquad\qquad\qquad\qquad \text{for all } t> c_{1.6.3}.
\end{cases}\end{align*}

\noindent
(ii) Assume that $\theta\le 1/\alpha$. Then,
there exist positive constants $c_{1.6.4}$ and $c_{1.6.5}$ such that
for all $x,y\in D$,
\begin{align*}&p_D(t,x,y)\asymp\\
 &\begin{cases} \displaystyle
\bigg(\frac{\delta_D(x)^{\alpha/2}\left(\log^{-\theta\alpha/2}(e+|x|)\wedge t^{1/2}\right)}
 {t}\wedge 1\bigg)\bigg(\frac{\delta_D(y)^{\alpha/2}\left(\log^{-\theta\alpha/2}(e+|y|)\wedge t^{1/2}\right)}
 {t}\wedge 1\bigg)\\
\qquad\qquad  \qquad\qquad\qquad\qquad\qquad\qquad\qquad\qquad\qquad\qquad \text{for all } 0<t\le c_{1.6.4}t_1(x,y);\\[2pt]
\displaystyle  p(t,x,y)\frac{\delta_D(x)^{\alpha/2}\log^{-\theta\alpha/2}(e+|x|)}{t}\frac{  \delta_D(y)^{\alpha/2}\log^{-\theta\alpha/2}(e+|y|)}{t} \exp\big(-t\log^{\theta\alpha}(e+(|x|\wedge|y|))\big)\\
 \qquad\qquad \qquad\qquad\qquad\qquad\qquad\qquad\qquad\qquad\qquad\qquad \text{for all } c_{1.6.4}t_1(x,y)< t\le
 c_{1.6.5}t_2(x,y);\\[2pt]
 \displaystyle
\frac{  \delta_D(x)^{\alpha/2}\log^{-\theta\alpha/2}(e+|x|)}{(1+|x|)^{d+\alpha}}
\frac{ \delta_D(y)^{\alpha/2}\log^{-\theta\alpha/2}(e+|y|)}{(1+|y|)^{d+\alpha}}
\exp(-t) \\
 \qquad\qquad \qquad\qquad\qquad\qquad\qquad\qquad\qquad\qquad\qquad\qquad \text{for all } t> c_{1.6.5}t_2(x,y).
\end{cases}\end{align*}
 \end{example}
\begin{proof} This  directly follows from Theorem \ref{Main}. Here we give some details on the case that $\theta>1/\alpha$ and
$c_{1.6.2}t_2(x,y)\le t\le c_{1.6.3}$.
For any $t>0$, define
 \begin{align*}
s_1(t)=\inf\{s>0: f(s)^\alpha\log(2+s)\le t\}\vee 2.
\end{align*}
 Then, for $0<t\le 1$,
 $$s_1(t)\asymp \exp\big({t^{-1/(\theta\alpha-1)}}\big).$$
 Hence, for any $c_i>0$ $(1\le i\le 3)$
 and $t\in (0,1]$,
 $$
  \int_0^{c_1s_1(c_2t)}f(s)^{d-1}e^{-c_3tf(s)^{-\alpha}}\,ds
  \asymp \exp\big({t^{-1/(\theta\alpha-1)}}\big).
$$

Indeed, it is clear that for all $t\in (0,1]$,
$$\int_0
 ^{c_1s_1(c_2t)}
 f(s)^{d-1}e^{-c_3tf(s)^{-\alpha}}\,ds\le   (\log 2)^{-\theta(d-1)}\int_{0}^{c_1s_1(c_2t)}\,ds \le c_{4}\exp({c_{5}t^{-1/(\theta\alpha-1)}}).$$
On the other hand, noting that $c_1s_1(c_2t)\ge c_6\exp({c_7 t^{-1/(\theta\alpha-1)}})$ for all $t\in (0,1]$ with some $c_6,c_7>0$ that satisfies $2c_3c_7^{\theta\alpha-1}\le 1$, and also that the function
$s\mapsto f(s)^{d-1}e^{-c_3tf(s)^{-\alpha}}$ is decreasing on $(0,\infty)$, we have
\begin{align*}&\int_{0}^{c_1s_1(c_2t)}
f(s)^{d-1}e^{-c_3tf(s)^{-\alpha}}\,ds\\
&\ge   \int_{0}^{c_6\exp({c_7 t^{-1/(\theta\alpha-1)}})}
f(s)^{d-1}e^{-c_3tf(s)^{-\alpha}}\,ds\\
 &\ge \log^{-\theta(d-1)}\left(2+ c_6\exp({c_7 t^{-1/(\theta\alpha-1)}})\right) \cdot \exp\left(-c_3t \log^{\alpha\theta}(2+ c_6e^{c_7 t^{-1/(\theta\alpha-1)}})\right) \cdot c_6\exp({c_7 t^{-1/(\theta\alpha-1)}})\\
& \ge c_{8}t^{\theta(d-1)/(\theta\alpha-1)} \exp( -c_9t-c_3 c_7^{\theta \alpha}t^{-1/(\theta\alpha-1)})\cdot \exp({c_7 t^{-1/(\theta\alpha-1)}})\\
&\ge c_{10}t^{\theta(d-1)/(\theta\alpha-1)}  \exp\left({\frac{c_7}{2} t^{-1/(\theta\alpha-1)}}\right)\ge  \exp({c_{11} t^{-1/(\theta\alpha-1)}})
\end{align*} for all $t\in (0,1]$.

  With these at hand, we can get the required assertions in Example $\ref{exam}$. \end{proof}

  Note that,
  for this example,
  the associated Dirichlet semigroup  $(P^D_t)_{t\ge0}$ is
intrinsically ultracontractive, if and only if $\theta>1/\alpha$; see \cite[Example 2]{Kw} or \cite[Theorem 1.1(1)]{CKW}.
On the other hand, it is easy to see that $\limsup_{|x|,|y|\to  \infty}t_2(x,y)=0$, if and only if
$\theta>1/\alpha$.
This explains why there is a threshold at $\theta=1/\alpha$ for two-sided estimates of $p_D(t,x,y)$.

\bigskip

The rest of this paper is arranged as follows. The next section serves as  preparations for main proofs. Results in
Sections 2 will be frequently used in the proof of Theorem \ref{Main}. In particular,
upper bound estimates of survival probabilities for full time are presented here. Sections 3, 4 and 5 are devoted to the proof of Theorem \ref{Main}, according to different ranges of time. Proof of
Theorem \ref{Main} and further remarks are briefly given in Section 6.

\section{Preparations}

\subsection{Preliminary estimates}
In this part, we collect some (mostly known) results which
will be frequently used in proofs of  our paper.
Throughout this paper, let $X:=\{X_t,t\ge0; \Pp^x,x\in \R^d\}$ be a (rotationally) symmetric $\alpha$-stable process in $\R^d$ with $d \ge 2$, whose transition density is denoted by  $p(t,x,y)$. For any open subset $U$, let
$X^U$ be
 the subprocess of $X$ killed
upon leaving $U$, whose transition density is denoted by $p_U(t,x,y)$.
Let $\tau_U:=\inf\{t\ge 0: X_t\notin U\}$ be the first exit time from $U$ for the process $X$. It
is well known (cf. see \cite[Lemma 3.2]{CKS1}) that, for any
$\kappa_1,\kappa_2>0$, there exists a constant $c_1:=c_1(\kappa_1,\kappa_2)>0$ such that for
all $x\in \R^d$ and $t>0$,
\begin{equation}\label{l5-1-2}
\Pp^x\left(\tau_{B(x,\kappa_1t^{1/\alpha})}>\kappa_2t\right)\ge c_1.
\end{equation}
Recall that the L\'evy system of $X$ describes the behaviors of jumps for the process $X$.
In particular, given a non-negative function $f:\R_+\times \R^d\times \R^d \to \R_+$ with
$f(s,x,x)=0$ for all $s>0$ and $x\in \R^d$, it holds for any stopping time $\tau$ that
\begin{equation}\label{e2-6}
\Ee^x\left[\sum_{s\le \tau}f\left(s,X_{s-},X_s\right)\right]=
\Ee^x\int_0^\tau\int_{\R^d}f\left(s,X_s,y\right)\frac{c_{d,\alpha}}{|X_s-y|^{d+\alpha}}\,dy\,ds.
\end{equation}
We refer the reader to \cite[Lemma 4.7]{CK} for more details about the property of L\'evy system.
On the other hand, according to \cite[Lemma 2]{Kw}, we have
\begin{lemma}\label{l5-3}
There exists a constant $c_{\ac{1}}\!>0$ such that for any open set $U\subset
\R^d$, $x\in U$ and $t>0$,
\begin{equation}\label{l5-3-1}
\Pp^x\left(\tau_U>
t\right)\le \exp\left(-c_{\ac{1}}\eta_Ut\right),
\end{equation}
where $\eta_U:=\inf_{x\in
U}\displaystyle\int_{U^c}{|x-z|^{-d-\alpha}\,}dz.$
\end{lemma}

Throughout the remainder of this paper, let $f: \R\to (0,\infty)$ satisfy \eqref{e:f1}, \eqref{e:f2} and \eqref{e:f3}. For fixed
constants $c_* \in (0,
1/5]$ and  $\Lambda >0$, let
$D$ be a horn-shaped $C^{1,1}$ region with the
reference function $f$
so that for all $x \in D$,  $D$ is $C^{1,1}$ at $z_x\in \partial D$ with the  characteristics
$(
5c_*f(x_1), \Lambda)$.

To save notations in the proofs, without loss of generality, we
may assume that the following conditions are satisfied:
\medskip

\noindent
(i) $f(0)\le 2^{-2}$, and for all $x\in D$, $\delta_D(x)\le 2^{-1}$;

\noindent
(ii) $D \cap  \{x\in \R^d:
x_1<2\}  \subset B(0, 2)$, and
$\{x\in D:
x_1>2\}=D^{2}_f$;

\noindent
(iii)(non-uniform) Interior ball condition:\,
for every $x\in D$ and $0<r\le 5c_*f(x_1)$,
$B(\xi^*_{x,r},r)\subset D$, where
$\xi^*_{x,r}:=z_x+r({x-z_x})/{|x-z_x|}$.
\medskip

We remark here that, clearly the arguments below work for general $C^{1,1}$ horn-shape regions without the additional assumptions (i)--(iii).

\begin{lemma}\label{L:2.2} There exists a constant $c_{\ac{1}}>0$ such that for all $x\in D$, $0<t\le c_{\ac{1}} f(x_1)^\alpha$ and $\lambda_i>0$ $(i=1,2,3)$, there is a constant $c_{\ac{2}}:=c_{\ac{2}}(c_{\ac{1}}, \lambda_1, \lambda_2,\lambda_3)$ so that when $\delta_D(x)\ge \lambda_1 t^{1/\alpha}$,
$p_D(t,x,y)\ge c_{\ac{2}} t^{-d/\alpha}$
holds for all $y\in D$ with $\delta_D(y)\ge  \lambda_2 t^{1/\alpha}$ and $|x-y|\le \lambda_3 t^{1/\alpha}$.
\end{lemma}
\begin{proof}Since for any $x\in D$,
$D$ is $C^{1,1}$ at $z_x\in \partial D$ with
the characteristics $(5c_*f(x_1),\Lambda)$, the desired assertion
can be proven by the arguments for the proof of
\cite[Proposition 3.3]{CKS1} or \cite[Proposition 3.6]{CKS}. \end{proof}

The next lemma is partially  motivated by \cite[Lemma 5.4]{KK} and \cite[Lemma 7.4]{GKK}.
\begin{lemma}\label{l5-1}
For every   $\lambda\in (0,1]$, there exists a constant $c_{\ac{1}}:=c_{\ac{1}}(\lambda)>0$ such that for all  $t>0$ and $x\in
D$ with $0<t^{1/\alpha}\le c_*f(x_1)$, there is $\xi^t_x\in D$ so that $B(\xi^t_x, 4\lambda t^{1/\alpha})\subset D$ and
\begin{equation}\label{l5-1-1}
\int_{B(\xi^t_{x}, 2\lambda t^{1/\alpha})}p_D(t,x,z)\,dz\ge c_{\ac{1}}
\left(\frac{\delta_D(x)^{\alpha/2}}{\sqrt{t}}\wedge 1\right),
\end{equation}
where
$$\xi_x^t:=
\begin{cases}
\xi^*_{x,4\lambda t^{1/\alpha}}=z_x+4\lambda
t^{1/\alpha}({x-z_x})/{|x-z_x|} &\text{ when } \delta_D(x)\le 4\lambda t^{1/\alpha};\\
x &\text{ when } \delta_D(x)> 4\lambda t^{1/\alpha}.
\end{cases}
$$
\end{lemma}
\begin{proof}
Fix $\lambda\in (0,1]$.
If
$\delta_D(x)> 4\lambda t^{1/\alpha}$, then $B(x,4\lambda
t^{1/\alpha})\subset D$, and so
\begin{align*}
\int_{B(x, 2\lambda t^{1/\alpha})}p_D(t,x,z)\,dz&
\ge \int_{B(x, 2\lambda t^{1/\alpha})}p_{B(x, 2\lambda t^{1/\alpha})}(t,x,z)\,dz=\Pp^x\big(\tau_{B(x, 2\lambda t^{1/\alpha})}>t\big)\ge c_4,
\end{align*}
where the last inequality follows from \eqref{l5-1-2} with
$\kappa_1=2\lambda$ and $\kappa_2=1$. Thus \eqref{l5-1-1} holds for this case.

Now, we turn to the case that $\delta_D(x)\le 4\lambda t^{1/\alpha}$.
Since $5\lambda t^{1/\alpha}\le 5c_*f(x_1)$, according to the (non-uniform) interior ball condition of $D$,
 $B(\xi^t_{x}, 4\lambda t^{1/\alpha})\subset D$ and $B(\tilde\xi^t_{x}, 5\lambda
t^{1/\alpha})\subset D$ with
$\tilde \xi^t_{x}:=\xi^*_{x,5\lambda t^{1/\alpha}}=z_x+5\lambda
t^{1/\alpha}{x-z_x}/{|x-z_x|}.$
In particular,  $B(\xi^t_{x},2\lambda
t^{1/\alpha})\subset B(\tilde \xi^t_{x},
5\lambda
t^{1/\alpha})\subset D$.
Since $\delta_D(x)\le 4\lambda t^{1/\alpha}$, we have  $x\in B(\tilde \xi^t_{x},
5\lambda t^{1/\alpha})$ with $\delta_{B(\tilde \xi^t_{x},
5\lambda t^{1/\alpha})}(x)= \delta_D(x)$, and, for any $z\in B(\xi^t_{x},2\lambda
t^{1/\alpha})$, $|x-z|\le|x-\xi^t_{x}|+|\xi^t_{x}-z|\le c_1 t^{1/\alpha}$ and $\delta_{B(\tilde \xi^t_{x},
5\lambda t^{1/\alpha})}(z)\ge 2\lambda t^{1/\alpha}$.
Thus, according to \eqref{hk-s},
\begin{align*}
&\int_{B(\xi^t_{x}, 2\lambda t^{1/\alpha})}p_D(t,x,z)\,dz \ge
\int_{B(\xi^t_{x}, 2\lambda t^{1/\alpha})}p_{B(\tilde \xi^t_{x},
5\lambda t^{1/\alpha})}(t,x,z)\,dz\\
&\ge
c_2\frac{\delta_{B(\tilde \xi^t_{x},
5\lambda t^{1/\alpha})}(x)^{\alpha/2}}{\sqrt{t}}
t^{-d/\alpha}
\int_{B(\xi^t_{x}, 2\lambda t^{1/\alpha})}\frac{\delta_{B(\tilde \xi^t_{x},
5\lambda t^{1/\alpha})}(z)^{\alpha/2}}{\sqrt{t}}\,dz\ge c_3\frac{\delta_D(x)^{\alpha/2}}{\sqrt{t}}.
\end{align*}
\end{proof}

\begin{lemma}\label{l5-2}
There exist constants $c_{\ac{1}}\in (0,1)$ and $ c_{\ac{2}}>0$ such that for all  $t>0$  and $x\in D$ with $0<t \le
c_{\ac{1}}f(x_1)^\alpha$,
\begin{equation}\label{l5-2-1}
\Pp^x\big(\tau_D>t\big)\le
c_{\ac{2}}\left(\frac{\delta_D(x)^{\alpha/2}}{\sqrt{t}}\wedge1\right).
\end{equation}
\end{lemma}
\begin{proof}
It suffices to prove
\eqref{l5-2-1} for the case $\delta_D(x)\le c_1t^{1/\alpha}$ with
arbitrary fixed $c_1>0$.

On the one hand, note that for any $x\in D$, $D$ is $C^{1,1}$ at $z_x\in
\partial D$ with the characteristics
$(5c_* f(x_1), \Lambda)$. We can follow the proof of \cite[(2.11) in Theorem
2.6]{KK} to
find
constants
$c_2,c_3 \in (0, 1)$
such that for every $x\in D$ and $0<t\le
c_2
f(x_1)^\alpha$
with $\delta_D(x)\le c_3t^{1/\alpha}$,
\begin{equation}\label{l5-2-3}
\Ee^x[\tau_{V_t}]\le c_4t^{1/2}\delta_D(x)^{\alpha/2},
\end{equation}
where  $V_t:=B(z_x,2c_3 t^{1/\alpha})\cap D$.

On the other hand, according to
\cite[Lemma 2.4]{CKS-u4}, it holds that for every
$t>0$ and $x\in D$ with with $\delta_D(x)\le c_3t^{1/\alpha}$,
\begin{equation}\label{l5-2-4}
\Pp^x(X_{\tau_{V_t}}\in D) \le \Pp^x(X_{\tau_{V_t}}\in B(z_x,2c_3 t^{1/\alpha})^c)\le c_5t^{-1}\Ee^x[\tau_{V_t}].
\end{equation}

Combining both estimates above together
yields that, for any $x\in D$ and $0<t\le c_2 f(x_1)^\alpha$ with $\delta_D(x)\le  c_3 t^{1/\alpha}$, (by noting that $f\le 2^{-2}$),
\begin{align*}
\Pp^x(\tau_D>t)&=\Pp^x(\tau_{V_t}\ge t)+\Pp^x(\tau_D>t>\tau_{V_t})\le \Pp^x(\tau_{V_t}\ge t)+\Pp^x(X_{\tau_{V_t}}\in D)\\
&\le
c_6t^{-1}\Ee^x[\tau_{V_t}]\le
c_7\frac{\delta_D(x)^{\alpha/2}}{\sqrt{t}},
\end{align*}proving the desired assertion.
\end{proof}

\begin{lemma}\label{l5-2-0}
For all $\lambda\in(0,1]$, there exist constants $c_{\ac{1}}:=c_{\ac{1}}(\lambda)$ and $c_{\ac{2}}:=c_{\ac{2}}(\lambda)\in (0,1)$ such that for any $t>0$ and $x\in
D$ with $0<t \le c_{\ac{1}}f(x_1)^\alpha$ and
$\delta_D(x)\le \lambda t^{1/\alpha}$,
\begin{equation}\label{l5-2-2}
\Pp^x\big(\tau_{B(z_x,10\lambda t^{1/\alpha})\cap D}>t\big)\ge
c_{\ac{2}}\frac{\delta_D(x)^{\alpha/2}}{\sqrt{t}}.
\end{equation}
\end{lemma}
\begin{proof}
This follows from
the proof of \cite[Lemma 5.2]{KK},
thanks to the fact that for any $x\in D$,
$D$ is $C^{1,1}$ at $z_x\in \partial D$ with
the characteristics $(5c_*f(x_1),\Lambda)$.
\end{proof}

\begin{lemma}\label{l5-4-2}
There exist constants $c_{\ac{1}} \in (0, 1)$ and $c_{\ac{2}}>0$ such that for all $t>0$
and
$x\in D$  with $\delta_D(x)\le
c_{\ac{1}}t^{1/\alpha}$,
\begin{equation}\label{e:l5-4-2}
\Ee^x\big[\tau_{B(z_x,c_{2.6.1}(t^{1/\alpha}\wedge 1))\cap D}\big]
\le
c_{\ac{2}}\delta_D(x)^{\alpha/2}\big(f(x_1)^{\alpha/2}\wedge t^{1/2}
\big).
\end{equation}
\end{lemma}
\begin{proof}
Let $c_2, c_3 \in (0,1]$ be the constants in \eqref{l5-2-3}, and set $c_{\ac{1}}=c_3$.
When $0<t\le c_2f(x_1)^{\alpha}$,  \eqref{e:l5-4-2} follows from
\eqref{l5-2-3}.
 If $t>c_2f(x_1)^{\alpha}$, then,
according to
\cite[Lemma 6.2]{CKW}
 (by choosing $c_{\ac{1}}$ small if necessary),
\begin{align*}
\Ee^x\big[\tau_{B(z_x,c_{\ac{1}}(t^{1/\alpha}\wedge 1))\cap D}\big]&\le
\Ee^x\big[\tau_{B(z_x,c_{\ac{1}})\cap D}\big]\le
c_4\delta_D(x)^{\alpha/2}f(x_1-2)^{\alpha}.
\end{align*}
Combining both estimates above with the fact that $f(x_1-2)\le
c_5f(x_1)$ for $x\in D$ immediately yields
\eqref{e:l5-4-2}.
\end{proof}

Recall that $\Psi(t,x)$ is  defined in \eqref{e:DPsi}.
\begin{lemma}\label{l:new}
There exists a constant $c_{\ac{1}}>0$ such that for
every $x, y \in D$ and $t>0$ with $t^{1/\alpha} \le 2 |x-y|$,
$$
p_D(t,x,y) \le c_{\ac{1}} \frac{t}
{|x-y|^{d+\alpha}}\Psi(t,x).$$
\end{lemma}
\begin{proof}
(i) Case 1: $\delta_D(x)\ge 2^{-4}t^{1/\alpha}$.
For any $x,y\in D$ and $t>0$,
$$
p_D(t,x,y) \le p(t,x,y)  \le \frac{c_1t}{|x-y|^{d+\alpha}} \le   \frac{c_2t}{|x-y|^{d+\alpha}}\left(
\frac{\delta_D(x)^{\alpha/2}\left(f(x_1)^{\alpha/2}\wedge t^{1/2}\right)}{t}\wedge1\right),$$
where in the last inequality we used the fact that $2^{-4}t^{1/\alpha}\le \delta_D(x)\le c_3f(x_1)$ for all $x\in D$.

(ii) Case 2: $\delta_D(x)\le 2^{-4}t^{1/\alpha}$. Without loss of generality, we
may assume that the constant $c_{2.6.1}$ in Lemma \ref{l5-4-2} is smaller than $2^{-4}$.
For fixed $x,y \in D$ such that $t^{1/\alpha} \le 2 |x-y|$,  let $V_1=B(z_x,
c_4 c_{2.6.1} (t^{1/\alpha} \wedge 1)) \cap D$ with $c_4\in (0,1)$ small enough, $V_3=\{z\in D: |z-x|\ge |x-y|/2\}$
and $V_2=D\backslash (V_1\cup V_3)$. Since $|z-x|\ge |x-y|/2\ge t^{1/\alpha}/4$ for all $z\in V_3$ and $c_{2.6.1}\le 2^{-4}$, we have  $dist(V_1,V_3)>0$. Then, by
\cite[Lemma 5.1]{GKK} (see \cite[Lemma 1.10]{BK} and \cite[Lemma 3.1]{CKS3}
for the proof)
we find that \begin{align*}
p_D(t,x,y) &\le \Pp^x(\tau_{V_1}\in V_2)\sup_{0\le s\le t, z\in
V_2}p_D(s,z,y) +c_5(t\wedge \Ee^x [\tau_{V_1}])\sup_{v\in V_1,z\in V_3}
\frac{1}{|v-z|^{d+\alpha}}\\
&\le c_6\left(\frac{\Ee^x[\tau_{V_1}]}{t
\wedge1}\wedge1\right)\sup_{0\le s\le t, z\in V_2}p_D(s,z,y) + c_5\left(\frac{\Ee^x[\tau_{V_1}]}{t}\wedge1\right)\sup_{v\in V_1,z\in V_3}\frac{t}{|v-z|^{d+\alpha}}\\
&\le \frac{c_7t}{|x-y|^{d+\alpha}}\left(\frac{\Ee^x[\tau_{V_1}]}{t\wedge1}\wedge1\right)\le \frac{c_8t}{|x-y|^{d+\alpha}}\left(
\frac{\delta_D(x)^{\alpha/2}\left(f(x_1)^{\alpha/2}\wedge t^{1/2}\right)}{t\wedge1}\wedge1\right),
\end{align*}
where the second inequality is due to \eqref{l5-2-4},
in the third inequality we
used the facts that
$$p_D(s,z,y)\le p(s,z,y)\le \frac{c_9s}{|z-y|^{d+\alpha}}\le
\frac{c_{10}t}{|x-y|^{d+\alpha}},\quad z\in V_2, 0<s\le t$$ (thanks to $|z-x|\le |x-y|/2$ for all
$z\in V_2$) and
\begin{align*}
\sup_{v\in V_1, z\in V_3}\frac{1}{|v-z|^{d+\alpha}}\le
\sup_{v\in V_1, z\in V_3}\frac{1}{(|z-x|-|v-x|)^{d+\alpha}}\le\frac{1}{{(|x-y|}/{2}-2^{-4}t^{1/\alpha})^{d+\alpha}}\le
\frac{c_{11}}{|x-y|^{d+\alpha}},
\end{align*} (thanks to the fact that $|x-y|\ge t^{1/\alpha}/2$), and  the fourth inequality follows from
 \eqref{e:l5-4-2}.
 The proof is complete.
\end{proof}

\subsection{Estimate of the survival probability}
In this part, we will present the following
estimate for the survival probability, which extends Lemma \ref{l5-2} for all
$t>0$.

\begin{lemma}\label{e:lm2} There are positive constants $c_{\ac{1}}$ and $c_{\ac{2}}$ such that for any $t>0$ and $x\in D$,
\begin{equation}\label{l3-1-0}
\begin{split}
\Pp^x(\tau_D >
t) \le c_{\ac{1}} \Psi(t,x)\min\Big\{e^{- c_{\ac{2}}f(x_1)^{-\alpha}t}+t(1+|x|)^{-(d+\alpha-1)},
e^{-c_{\ac{2}}t}\Big\}.
\end{split}
\end{equation}
\end{lemma}

\begin{proof}
(i) We will first show that for all $t>0$ and $x\in D$,
\begin{equation}\label{l3-1-1}
\Pp^x(\tau_D> t)\le
c_1\min\Big\{e^{-c_2f(x_1)^{-\alpha}t}+t(1+|x|)^{-(d+\alpha-1)},
e^{-c_2t}\Big\}.
\end{equation}
By \cite[(2.10) in Proposition 2.8]{CKW} and the fact that
$\delta_D(x)\le c_3f(x_1)$ for all $x\in D$, we know that for any $U\subset D$ and $z\in U$,
\begin{equation}\label{l5-3-1--}
 \int_{U^c}\frac{1}{|z-y|^{d+\alpha}}\,dy\ge
\int_{D^c} \frac{1}{|z-y|^{d+\alpha}}\,dy\ge
c_4\delta_D(z)^{-\alpha}\ge c_5f(z_1)^{-\alpha}.
\end{equation}
In particular, by \eqref{l5-3-1}, for all $t>0$ and $x\in D$,
$
\Pp^x(\tau_D> t) \le  e^{-c_6 t}.
$
Thus,
in order to verify \eqref{l3-1-1},
we only need to prove that for all $t>0$ and $x
\in D$ with $|x|$ large enough,
\begin{equation}\label{l3-1-1-a}
\Pp^x(\tau_D> t)\le c_1\Big(
e^{-c_2tf(x_1)^{-\alpha}}+t(1+|x|)^{-(d+\alpha-1)}\Big).\end{equation}

For any $x\in D$ with $|x|$ large enough, let $U=B(x,|x|/2)\cap D$. Then,
for $t>0$,
\begin{align*}&\Pp^x\left(\tau_D> t\right)=\Pp^x\left(\tau_U> t\right)+\Pp^x\left(\tau_{D}>
t\ge
\tau_{U}\right)\\
&\le \Pp^x\left(\tau_U> t\right)+\Pp^x\left(X_{\tau_U}\in D, \tau_U\le
t, X_t\in B(x,|x|/3)\cap D\right)  +\Pp^x\left(X_t\in B(x,|x|/3)^c\cap D\right)\\
&=:I_1+I_2+I_3.\end{align*}

First, by \eqref{l5-3-1} and \eqref{l5-3-1--},
\begin{align*}I_1\le  \Pp^x(\tau_U> t)\le  \exp\Big(-c_7t\inf_{z\in D:|z|>|x|/2}f(z_1)^{-\alpha}\Big)
 \le \exp\left(-c_8 f(x_1)^{-\alpha}t\right),\end{align*}where the last
inequality above is due to \eqref{e:f3}.

Second, due to the strong Markov property and \eqref{e:f3},
\begin{align*}I_2&\le
\Ee^x\left[\Pp^{X_{\tau_U}}\left(X_{t-\tau_U}\in B(x,|x|/3)\cap D\right):
\tau_U\le t,X_{\tau_U}\in D
\right]\\
&\le \sup_{0<s\le t,z\in U^c\cap D}\Pp^z(X_s\in B(x,|x|/3)\cap D)\le \sup_{0<s\le t, |z-x|\ge|x|/2}\int_{B(x,|x|/3)\cap D}p(s,z,y)\,dy\\
&\le
\frac{c_{9}t}{|x|^{d+\alpha}}|B(x,|x|/3)\cap D|\le \frac{c_{10}f(x_1)^{d-1}t}{(1+|x|)^{d+\alpha-1}}\le
\frac{c_{11}t}{(1+|x|)^{d+\alpha-1}},\end{align*} where in the
fourth inequality we used the fact that $|z-y|\ge |x|/6$ for any
$x,y,z\in \R^d$ with $|z-x|\ge |x|/2$ and $|y-x|\le |x|/3$ (and
so $p(s,z,y)\le c_{12}s|x|^{-d-\alpha}$ for all $s>0$).

Third,
it holds that
\begin{align*}
I_3&\le \int_{B(x,|x|/3)^c\cap D}p(t,x,z)\,dz \le
\int_{B(x,|x|/3)^c\cap D}\frac{c_{13}t}{|x-z|^{d+\alpha}}\,dz\\
&\le c_{14}
t
\int_{|x|/3}^\infty
\frac{1+f(s)^{d-1}}
{s^{d+\alpha}}\,ds
\le
\frac{c_{15}t}{(1+|x|)^{d+\alpha-1}}.
\end{align*}

Combining all the estimates above, we prove \eqref{l3-1-1-a}, and so
\eqref{l3-1-1} holds true.

(ii) In the following, we set
$$L(x,t)= \min\Big\{e^{-c_2tf(x_1)^{-\alpha}}+t(1+|x|)^{-(d+\alpha-1)}, e^{-c_2t}\Big\}.$$
 We first consider the case $\delta_D(x)\le
c_{2.6.1}  t^{1/\alpha} $ (where $c_{2.6.1}>0$ is the constant in Lemma \ref{l5-4-2}).
Letting  $V_1=B(z_x, c_{2.6.1} (t^{1/\alpha}\wedge
4^{-1}))\cap D$,  we have
\begin{align*}
\Pp^x(\tau_D>
t)\le&\Pp^x(\tau_{V_1}>t/2, \tau_D> t)+\Pp^x(0<\tau_{V_1}\le t/2, X_{\tau_{V_1}}\in D, \tau_D> t)=:J_1+J_2. \end{align*} By the strong Markov property,
\eqref{e:l5-4-2}
and \eqref{l3-1-1}, we get
\begin{align*}
J_1&=\Ee^x[
\I_{\{\tau_{V_1}>t/2\}}\Pp^{X_{t/2}}(\tau_D>t/2)]\le\Pp^x(\tau_{V_1}>t/2)\sup_{z\in V_1}\Pp^z(\tau_D>t/2)\\
&\le c_{16}\left(\frac{\Ee^x[\tau_{V_1}]}{t} \wedge1\right)\sup_{z\in V_1}\Pp^z(\tau_D>t/2)
\le c_{17}\Psi(t,x)\sup_{z\in V_1}L(z,t/2).
\end{align*}

Let $V_3=\{z\in D:|z-x|\ge 1+ |x|/2 \}$ and
$V_2=D\setminus (V_1\cup V_3)$.
If $z \in B(z_x,2^{-1})$, then
$|z-x| \le |z-z_x| +\delta_D(x) \le 2^{-1}+2^{-1} =1$, which implies that $dist(V_1,V_3)>0$.
(Here we note that $\delta_D(x)\le 1/2$ for all $x\in D$
by our assumption).
Using $V_1$, $V_2$ and $V_3$, we bound $J_2$ as
\begin{align*}
J_2
=
& \Ee^x\left[ \Pp^{X_{\tau_{V_1}}}(\tau_D>t/2): 0<\tau_{V_1}\le t/2,X_{\tau_{V_1}}\in V_2\right] \\
&+\Ee^x\left[ \Pp^{X_{\tau_{V_1}}}(\tau_D>t/2):0<\tau_{V_1} \le t/2,X_{\tau_{V_1}}\in V_3\right]=:J_{2,1}+J_{2,2}.\end{align*}
We find that
\begin{align*}
J_{2,1}&\le \Pp^x(X_{\tau_{V_1}}\in V_2) \sup_{z\in V_2}\Pp^z(\tau_D>t/2)\le c_{18}\left(\frac{\Ee^x[\tau_{V_1}]}{t\wedge1} \wedge1\right)\sup_{z\in V_2}L(z,t/2)\\
&\le
c_{19} \Psi(t,x)\sup_{z\in V_2} L(z,t/2),
\end{align*}
where the second inequality above follows from
\eqref{l5-2-4}
and \eqref{l3-1-1}, and the last inequality is
due to \eqref{e:l5-4-2}.

For $J_{2,2}$,  we use the L\'evy system
\eqref{e2-6},
\eqref{l3-1-1} and \eqref{e:l5-4-2} again and obtain that
\begin{align*}
J_{2,2}&\le
c_{20} e^{-c_2t} \Ee^x\Big[\int_0^{\tau_{V_1}\wedge (t/2)}\int_{V_3}\frac{1}{|X_s^{V_1}-z|^{d+\alpha}}\,dz\,ds\Big]\\
&\le c_{21}e^{-c_2t}(\Ee^x[\tau_{V_1}]\wedge
t)\int_{V_3}\frac{dz}{|x-z|^{d+\alpha}}
\le c_{22} e^{-c_2t}\left({\delta_D(x)^{\alpha/2}
 (f(x_1)^{\alpha/2}\wedge t^{1/2} )}\wedge
t\right) \int_{1+|x|/2}^\infty\frac{ds}
{s^{d+\alpha}} \\
&\le c_{23} \Big(\frac{\delta_D(x)^{\alpha/2}
\left(f(x_1)^{\alpha/2}\wedge t^{1/2}\right)}{t}\wedge
1\Big)\frac{t e^{-c_2t}}{(1+|x|)^{d+\alpha-1}}\le c_{24}\Psi(t,x)\min\Big\{\frac{t}{(1+|x|)^{d+\alpha-1}},
e^{-c_2t/2}\Big\}, \end{align*} where in the second inequality we used the fact that for any $y\in V_1$ and $z\in V_3$,
$$|y-z|\ge |z-x|-|x-y|\ge |z-x|-|x-z_x|-|y-z_x|\ge |x-z|-1/2-1/4\ge |x-z|/4.$$

Note that $|z|\le {3|x|}/{2}+1$ for every $z\in V_1\cup V_2$.
Then, by the fact that $f(s-2)\le c_{25}f(s)$ for all $s>0$ and
\eqref{e:f3},
$$\sup_{z\in V_1\cup V_2} L(z,t/2)\le c_{26}\min\Big\{e^{-c_{27}tf(x_1)^{-\alpha}}+t(1+|x|)^{-(d+\alpha-1)}, e^{-c_{27}t}\Big\}.$$
Therefore, the desired assertion \eqref{l3-1-0} for the case
$\delta_D(x)\le c_{2.6.1}
t^{1/\alpha}$
follows from all the estimates
above.

Next, we turn to the case that $\delta_D(x)\ge c_{2.6.1} t^{1/\alpha}$
(which is possible only when $t^{1/\alpha}\le c_{27}f(x_1)$, thanks to the fact that $\delta_D(x)\le c_{28} f(x_1)$ for all $x\in D$).
Then, according to \eqref{l3-1-1}, we have
\begin{align*}
\Pp^x(\tau_D>t)&\le c_{1}\min\left\{e^{-c_{2}f(x_1)^{-\alpha}t}+
t(1+|x|)^{-(d+\alpha-1)},e^{-c_{2}t}\right\}\\
&\le c_{29}\Psi(t,x)\min\left\{e^{-c_{2}f(x_1)^{-\alpha}t}+
t(1+|x|)^{-(d+\alpha-1)},e^{-c_{2}t}\right\},
\end{align*}where in the second inequality we used the fact that
$t^{1/\alpha}\le c_{27}f(x_1)$. Thus, we establish \eqref{l3-1-0}
for all $x\in D$. The proof is complete.
\end{proof}

\ \

From the next section to Section \ref{section5}, we will prove
Theorem \ref{Main}, which is exactly split into three cases according to different ranges of time $t$. By the
symmetry of $p_D(t,x,y)$ with respect to $(x,y)$, without loss of generality,
{\it we will assume
that $x_1\ge y_1$ throughout Sections $3$--$5$.}

\section{Case I: $t\le C_0f(y_1)^\alpha$ for some small constant $C_0>0$}
In this section, we will consider the case that $0<t\le
C_0f(y_1)^{\alpha}$, where $C_0\in (0,1)$ is a small positive constant to be fixed later.

\subsection{Near diagonal estimates, i.e.,
$|x-y|\le t^{1/\alpha}$.
}
\begin{lemma}\label{l2-2}{\bf (Lower bound)}\,
There exist constants $c_{3.1.1},\,c_{3.1,2}\in (0,1)$ such that for all $t>0$ and $x,y\in
D$ with $0<t\le c_{3.1.1} f(y_1)^\alpha$ and $|x-y|\le
 t^{1/\alpha}$,
 $$ p_D(t,x,y)\ge
c_{3.1.2}t^{-d/\alpha}\left(\frac{\delta_D(x)^{\alpha/2}}{\sqrt{t}}\wedge
1\right) \left(\frac{\delta_D(y)^{\alpha/2}}{\sqrt{t}}\wedge
1\right).
$$
\end{lemma}
\begin{proof}
(i)
Case 1:
$\delta_D(y)> 3 t^{1/\alpha}$.
Since
$B(x,2t^{1/\alpha})\subset D$ in this case, by \eqref{hk-s}
$$
p_D(t,x,y)
\ge p_{B(x,2t^{1/\alpha})}(t,x,y)\ge c_1
t^{-d/\alpha}.
$$

(ii) Case 2: $\delta_D(y)\le 3t^{1/\alpha}$.
It is obvious that $\delta_D(x)\le 4t^{1/\alpha}$. Recall that we have assumed that $f\le 1/4$. Then, $|x_1-y_1| \le |x-y|\le
t^{1/\alpha}\le c_{3.1.1}^{1/\alpha} f(y_1) \le 1/4$, and so
$f(x_1)\simeq f(y_1)$.
In particular, $t\le c_2c_{3.1.1} f(x_1)^\alpha$. Hence, by choosing
$c_{3.1.1} \in (0,1)$
small if necessary, we
get from Lemma \ref{l5-1} that $B(\xi^t_{x},2(t/3)^{1/\alpha})\subset D$,
$B(\xi^t_{y},2(t/3)^{1/\alpha})\subset D$, and
\begin{equation}\label{l2-2-2}
\begin{split}
& \int_{B(\xi^t_{x}, (t/3)^{1/\alpha})}p_D(t/3,x,z)\,dz\ge c_3\frac{\delta_D(x)^{\alpha/2}}{\sqrt{t}},\quad \int_{B(\xi^t_{y}, (t/3)^{1/\alpha})}p_D(t/3,y,z)\,dz\ge
c_3\frac{\delta_D(y)^{\alpha/2}}{\sqrt{t}},
\end{split}
\end{equation}where $\xi^t_{x}:=z_x+2(t/3)^{1/\alpha}({x-z_x})/{|x-z_x|}$ and
$\xi^t_{y}:=z_y+2({t/3})^{1/\alpha}({y-z_y})/{|y-z_y|}$.

On the other hand, for every $z_1\in B(\xi^t_{x}, (t/3)^{1/\alpha})$
and $z_2\in B(\xi^t_{y}, (t/3)^{1/\alpha})$, we have
$\delta_D(z_1)\ge (t/3)^{1/\alpha}$, $\delta_D(z_2)\ge
(t/3)^{1/\alpha}$ and
\begin{align*}
|z_1-z_2|&\le |z_1-\xi^t_{x}|+|\xi^t_{x}-x|+|x-y|+ |\xi^{t}_y-y|
+|z_2-\xi^t_{y}|\le c_4t^{1/\alpha}.
\end{align*}
Thus, by Lemma \ref{L:2.2},
$$p_D(t/3,z_1,z_2)\ge  c_5t^{-d/\alpha},\quad (z_1, z_2) \in B(\xi^t_{x},
(t/3)^{1/\alpha}) \times B(\xi^{t}_y, (t/3)^{1/\alpha}).$$ Combining
this with \eqref{l2-2-2} in turn gives us
\begin{align*} p_D(t,x,y)
&\ge \int_{B(\xi^t_{x},(t/3)^{1/\alpha})}\int_{B(\xi^t_{y},(t/3)^{1/\alpha})}
p_D(t/3,x,z_1)p_D(t/3,z_1,z_2)p_D(t/3,z_2,y)\,dz_1\,dz_2\\
&\ge c_5t^{-d/\alpha}\left(\int_{B(\xi^t_{x},
(t/3)^{1/\alpha})}p_D(t/3,x,z)\,dz\right)
\left(\int_{B(\xi^t_{y},  (t/3)^{1/\alpha})}p_D(t/3,y,z)\,dz\right)\\
&\ge c_6t^{-d/\alpha}\frac{\delta_D(x)^{\alpha/2}}{\sqrt{t}}\frac{\delta_D(y)^{\alpha/2}}{\sqrt{t}}.
\end{align*}

Using both estimates
in (i) and (ii),
we obtain the desired assertion.
\end{proof}

\begin{lemma}\label{l2-3}{\bf (Upper bound)}\,
There exist constants $c_{3.2.1}\in (0,1)$ and $c_{3.2.2}>0$ such that for all $t>0$ and $x,y\in D$ with $0<t\le c_{3.2.1}f(y_1)^\alpha$ and $|x-y|\le t^{1/\alpha}$,
$$
p_D(t,x,y)\le
c_{3.2.2}t^{-d/\alpha}\left(\frac{\delta_D(x)^{\alpha/2}}{\sqrt{t}}\wedge1\right)
\left(\frac{\delta_D(y)^{\alpha/2}}{\sqrt{t}}\wedge1\right).$$
 \end{lemma}
\begin{proof}
As
explained in the beginning of part (ii) of the proof for Lemma
\ref{l2-2} above, we can apply Lemma \ref{l5-2}  and obtain that for  all $t>0$ and $x,y\in D$ with $0<t\le c_{3.2.1}f(y_1)^\alpha$ and $|x-y|\le t^{1/\alpha}$ (by choosing $c_{3.2.1}$ small enough if necessary),
it holds that $t\le c_1c_{3.2.1}f(x_1)^\alpha$, and
$$\Pp^x(\tau_D>t/3)\le c_2\left(\frac{\delta_D(x)^{\alpha/2}}{\sqrt{t}}\wedge1\right),\quad \Pp^y(\tau_D>t/3)\le c_2\left(\frac{\delta_D(y)^{\alpha/2}}{\sqrt{t}}\wedge1\right).$$
Hence,
\begin{align*}p_D(t,x,y)&=\int_D\int_Dp_D(t/3,x,z_1)p_D(t/3,z_1,z_2)p_D(t/3,z_2,y)\,dz_1\,dz_2\\
&\le c_3t^{-d/\alpha}\left(\int_Dp_D(t/3,x,z_1)\,dz_1\right)\left(\int_Dp_D(t/3,z_2,y)\,dz_2\right)\\
&=c_3t^{-d/\alpha}\Pp^x(\tau_D>t/3) \Pp^y(\tau_D>t/3)
 \le c_4
t^{-d/\alpha}\left(\frac{\delta_D(x)^{\alpha/2}}{\sqrt{t}}\wedge1\right)\left(\frac{\delta_D(y)^{\alpha/2}}{\sqrt{t}}\wedge1\right).\end{align*}
 The proof is
complete.
\end{proof}
\subsection{Off-diagonal estimates, i.e.,
$|x-y|\ge  t^{1/\alpha}$}
\begin{lemma}\label{l2-4}{\bf (Lower bound when $0<t\le C_0f(x_1)^{\alpha}$.)}\,
There exist constants $c_{3.3.1},\,c_{3.3.2} \in (0,1)$
such that for all $t>0$ and $x,y\in
D$ with $0<t\le c_{3.3.1}f(x_1)^{\alpha}$ and
$|x-y|\ge t^{1/\alpha}$,
$$
p_D(t,x,y)\ge
c_{3.3.2}\frac{t}
{|x-y|^{d+\alpha}}\left(\frac{\delta_D(x)^{\alpha/2}}{\sqrt{t}}\wedge
1\right) \left(\frac{\delta_D(y)^{\alpha/2}}{\sqrt{t}}\wedge
1\right). $$
\end{lemma}
\begin{proof}
(i) Case 1: $\delta_D(x)\le 40^{-1}t^{1/\alpha}$ and
$\delta_D(y)\ge 2^{-2}t^{1/\alpha}$.
Note that $B(y,2^{-2}t^{1/\alpha})\subset D$ and
$0<t\le
c_{3.3.1}f(x_1)^{\alpha}\le c_{3.3.1}f(y_1)^{\alpha}$.
Then, by choosing $c_{3.3.1}>0$ small enough, we can obtain
from Lemma \ref{l5-2-0} that
\begin{equation}\label{l2-4-2}
\Pp^x\big(\tau_{B(z_x,4^{-1}t^{1/\alpha})\cap D}>t\big)\ge
c_1\frac{\delta_D(x)^{\alpha/2}}{\sqrt{t}}.
\end{equation}
In the following, we set $V_1=B(z_x,2^{-2}t^{1/\alpha})\cap
D$, $V_2=B(y,2^{-3}t^{1/\alpha})$ and
$V_2'=B(y,2^{-4}t^{1/\alpha})$. Since $|x-y|\ge
t^{1/\alpha}$ and $\delta_D(x)\le 40^{-1}t^{1/\alpha}$, we
have $V_1\cap V_2=\emptyset$, and for $(v, z) \in
V_1 \times V_2$,
\begin{equation}\label{l2-4-3}
\begin{split}
|v-z|&\le |v-z_x|+|z_x-x|+|x-y|+|y-z| \le |x-y|+ {3} t^{1/\alpha}/4\le 2|x-y|.
\end{split}
\end{equation} On the other hand, since $\delta_D(y)\ge 2^{-2}t^{1/\alpha}$ and
$t\le c_{3.3.1}f(y_1)^\alpha$,
by choosing $c_{3.3.1} \le c_{2.2.1}$, it follows from  Lemma \ref{L:2.2}
 that
\begin{equation}\label{l2-4-4}
p_D(t/2,z,y)\ge c_3t^{-d/\alpha},\quad z\in V_2,
\end{equation} where we used the fact that $\delta_D(z)\ge 2^{-3}t^{1/\alpha}\ge |z-y|/2$ for all $z\in V_2$.

Therefore,  by the strong Markov property and \eqref{l2-4-4},
\begin{align*}
&p_D(t,x,y)=\Ee^x \left[ p_D(t/2,X_{t/2}^D,y)\right]=\Ee^x\left[p_D(t/2,X_{t/2},y):t/2<\tau_D\right]\\
&\ge \Ee^x\left [p_D(t/2,X_{t/2},y) : 0\le \tau_{V_1}\le
t/4,\,X_{\tau_{V_1}}\in V_2',\,X_s\in V_2 \text{ for all }  s\in
[\tau_{V_1}, \tau_{V_1}+t/2]\right]\\
&\ge c_3 t^{-d/\alpha} \Pp^x\left(0\le \tau_{V_1}\le t/4,X_{\tau_{V_1}}\in V_2',X_s\in V_2\text{ for all } s\in [\tau_{V_1}, \tau_{V_1}+t/2]\right)\\
&\ge c_3 t^{-d/\alpha}\Ee^x\left[\Pp^{X_{\tau_{V_1}}}(\tau_{B(X_{\tau_{V_1}},2^{-4}t^{1/\alpha})}>t/2):
0\le \tau_{V_1}\le t/4,X_{\tau_{V_1}}\in V_2'
\right]\\
&\ge c_3t^{-d/\alpha}\Pp^x\left(0\le \tau_{V_1}\le t/4,X_{\tau_{V_1}}\in V_2' \right)\inf_{z\in V_2'}\Pp^{z}\left(\tau_{B(z,2^{-4}t^{1/\alpha})}>t/2\right)\\
&\ge c_4 t^{-d/\alpha} \int_0^{t/4}\int_{V_1}p_{V_1}(s,x,z)\int_{V_2'}\frac{1}{|z-u|^{d+\alpha}}\,du\,dz\,ds\\
&\ge c_5 \frac{ t^{-d/\alpha}|V_2'|}{|x-y|^{d+\alpha}}\int_0^{t/4}\Pp^x(\tau_{V_1}>s)\,ds \ge \frac{c_6t}{|x-y|^{d+\alpha}}\Pp^x(\tau_{V_1}>t/4)
\ge
\frac{c_7t}{|x-y|^{d+\alpha}}
\cdot\frac{\delta_D(x)^{\alpha/2}}{\sqrt{t}},\end{align*} where in
the fifth inequality we used the L\'evy system
\eqref{e2-6}
and \eqref{l5-1-2}
(since $V_1\cap V_2=\emptyset$), the sixth inequality is due to
\eqref{l2-4-3}, and the last inequality follows from \eqref{l2-4-2}.

(ii) Case 2:
$\delta_D(x)\ge 40^{-1}t^{1/\alpha}$ and
$\delta_D(y)\ge 2^{-2}t^{1/\alpha}$. Following the argument
of part (i) with $V_1$ replaced by
$V_1=B(x,40^{-1}t^{1/\alpha})$ and noting that
$\Pp^x\big(\tau_{V_1}>t/4\big)\ge c_8$, we can  prove that
$$
p_D(t,x,y)\ge \frac{c_{9}t}{|x-y|^{d+\alpha}}.
$$

(iii) Case 3:
$\delta_D(y)\le 2^{-2}t^{1/\alpha}$.
According to Lemma \ref{l5-1} (by choosing $c_{3.3.1}$ small if necessary),
 for every $0<t\le c_{3.3.1}f(y_1)^{\alpha}$, there is
$\xi_y^{t}:=z_y+2^{-1}{(t/2)^{1/\alpha}}({y-z_y})/{|y-z_y|}\in D$
such that $B(\xi_y^{t},2^{-1}{(t/2)^{1/\alpha}})\subset D$,
$|\xi_y^{t}-y|\le |\xi^t_y-z_y|+|z_y-y|\le \frac{3t^{1/\alpha}}{4}$,
and
\begin{equation}\label{l2-4-5}
\int_{B(\xi_y^{t},2^{-2}(t/2)^{1/\alpha})}p_D(t/2,z,y)\,d z\ge c_{10}\frac{\delta_D(y)^{\alpha/2}}{\sqrt{t}}.
\end{equation}
On the other hand, for all $z\in
B(\xi_y^{t},2^{-2}(t/2)^{1/\alpha})$, we have $t \le c_{11}c_{3.3.1}f(z_1)^\alpha$,
$\delta_D(z) \ge \delta_D(\xi_y^{t})-2^{-2}(t/2)^{1/\alpha}\ge 2^{-2}(t/2)^{1/\alpha}$,
\begin{equation}\label{l2-4-6i}|x-z| \ge |x-y|-
|y-\xi_y^{t}|-|\xi_y^{t}-z|
\ge
\frac{1}{4}(1-2^{-\alpha}
)t^{1/\alpha}=
\frac{1}{4}(2^{\alpha}-1)(t/2)^{1/\alpha}\end{equation}
and
\begin{equation}\label{l2-4-6}|x-z| \le |x-y|+
|y-\xi_y^{t}|+|\xi_y^{t}-z|
\le |x-y|+t^{1/\alpha}\le 2|x-y|.
\end{equation}
Hence,
by conclusions in parts (i) and (ii) (after adjusting constants),
we can obtain that for any $ z\in B(\xi_y^{t},2^{-2}(t/2)^{1/\alpha})$,
\begin{align*}
p_D(t/2,x,z)&\ge
\frac{c_{12}t}{|x-z|^{d+\alpha}}\left(\frac{\delta_D(x)^{\alpha/2}}{\sqrt{t}}\wedge
1\right) \ge
\frac{c_{13}t}{|x-y|^{d+\alpha}}\left(\frac{\delta_D(x)^{\alpha/2}}{\sqrt{t}}\wedge
1\right).
\end{align*}
Therefore, putting both estimates
together, we arrive at
\begin{align*}p_D(t,x,y)=&\int_D p_D(t/2,x,z)p_D(t/2,z,y)\,dz
\ge \int_{B(\xi_y^{t},
2^{-2}(t/2)^{1/\alpha})}
p_D(t/2,x,z)p_D(t/2,z,y)\,dz\\
\ge& \frac{c_{13} t}{|x-y|^{d+\alpha}}\left(\frac{\delta_D(x)^{\alpha/2}}{\sqrt{t}}\wedge 1\right)
\int_{B(\xi_y^{t},
2^{-2}(t/2)^{1/\alpha})} p_D(t/2,z,y)\,dz\\
\ge& \frac{c_{14} t}{|x-y|^{d+\alpha}}\left(\frac{\delta_D(x)^{\alpha/2}}{\sqrt{t}}\wedge1\right)
\frac{\delta_D(y)^{\alpha/2}}{\sqrt{t}}.\end{align*}

By all the conclusions above, we can obtain the desired assertion.
\end{proof}

\begin{lemma}\label{l2-5} {\bf (Lower bound when $C_0f(x_1)^\alpha\le
t\le C_0f(y_1)^{\alpha}.)$}\,\,
There exists $c_{3.4.0}\in (0,1)$ such that,
for all  $c_{3.4.1}, c_{3.4.2} \in (0, c_{3.4.0}]$ and for all  $x,y\in D$ and $t>0$ satisfying
$c_{3.4.1}f(x_1)^{\alpha}\le t \le
c_{3.4.2}f(y_1)^{\alpha}$ and $|x-y|\ge t^{1/\alpha}$,
there is  $c_{3.4.3}\in (0,1)$ so that
$$
p_D(t,x,y)\ge
\frac{c_{3.4.3}
t}{|x-y|^{d+\alpha}}\left(\frac{\delta_D(x)^{\alpha/2}f(x_1)^{\alpha/2}}{t}\wedge
1\right) \left(\frac{\delta_D(y)^{\alpha/2}}{\sqrt{t}}\wedge
1\right).
$$
\end{lemma}
\begin{proof}
 We
 may assume that $c_{3.4.1}, c_{3.4.2} \in (0, c_{3.4.0}]$, where $c_{3.4.0}$ is a small positive   constant less than $c_{2.5.1}=c_{2.5.1}(1/10)$ which will be chosen later.

 (i) Case 1:
 $\delta_D(x)\le
{10^{-1}(c_{3.4.1}/4)^{1/\alpha}f(x_1)}$
 and $\delta_D(y)\ge
2^{-2}t^{1/\alpha}$ (and that $c_{3.4.1}f(x_1)^{\alpha}\le t \le
c_{3.4.2}f(y_1)^{\alpha}$ and $|x-y|\ge t^{1/\alpha}$). Set
$V_1=B(z_x,
{(c_{3.4.1}/4)
^{1/\alpha}f(x_1)})\cap D$,
$V_2=B(y,2^{-3}t^{1/\alpha})\subset D$ and
$V_2'=B(y,2^{-4}t^{1/\alpha})$. Since $|x-y|\ge
t^{1/\alpha}$, it is easy to verify that $V_1\cap V_2=\emptyset$.
Then, by the Markov property,
\begin{align*}
&p_D(t,x,y)
=\Ee^x \left[ p_D(t/2,X_{t/2}^D,y)\right]=\Ee^x\left[p_D(t/2,X_{t/2},y): \tau_D> t/2 \right]\\
&\ge \Ee^x\left[p_D(t/2,X_{t/2},y):0\le \tau_{V_1}\le
2^{-2}c_{3.4.1}f(x_1)^\alpha,\,X_{\tau_{V_1}}\in V_2',\,X_s\in
V_2\text{ for all } s\in [\tau_{V_1},
\tau_{V_1}+t/2] \right].\end{align*} According to arguments in part (i)
of the proof of Lemma \ref{l2-4}, \eqref{l2-4-3} and \eqref{l2-4-4}
still hold by choosing $c_{3.4.0}$ less than $c_{2.2.1}$.
Therefore, by the strong Markov property again,
\begin{align*}  p_D(t,x,y)
&\ge c_1 t^{-d/\alpha}\Pp^x\left(0\le \tau_{V_1}\le {c_{3.4.1}f(x_1)^{\alpha}}/{4},
X_{\tau_{V_1}}\in V_2',X_s\in V_2 {\hbox{ for all }} \in [\tau_{V_1}, \tau_{V_1}+t/2]\right)\\
& \ge c_1 t^{-d/\alpha}\Ee^x\left[\Pp^{X_{\tau_{V_1}}}(\tau_{B(X_{\tau_{V_1}}, 2^{-4}t^{1/\alpha})}
>t): 0\le \tau_{V_1}\le
{c_{3.4.1}f(x_1)^{\alpha}}/{4},X_{\tau_{V_1}}\in
V_2'\right]\\
&\ge c_1t^{-d/\alpha}\Pp^x\left(0\le \tau_{V_1}\le {c_{3.4.1}f(x_1)^{\alpha}}/{4}, X_{\tau_{V_1}}\in V_2'\right)\inf_{z\in V_2'}\Pp^{z}
\left(\tau_{B(z,2^{-4}t^{1/\alpha})}>t\right)\\
& \ge c_2 t^{-d/\alpha} \int_0^{2^{-2}c_{3.4.1}f(x_1)^\alpha}\int_{V_1}p_{V_1}(s,x,z)\int_{V_2'}\frac{1}{|z-u|^{d+\alpha}}\,du\,dz\,ds\\
&\ge c_3 t^{-d/\alpha}|V_2'| \frac{1}{|x-y|^{d+\alpha}}\int_0^{2^{-2}c_{3.4.1}f(x_1)^\alpha}\Pp^x(\tau_{V_1}>s)\,ds\\
&\ge \frac{c_{4}f(x_1)^{\alpha}}{|x-y|^{d+\alpha}}\Pp^x\left(\tau_{V_1}>2^{-2}c_{3.4.1}f(x_1)^\alpha\right)\\
&\ge \frac{c_{5}f(x_1)^{\alpha}}{|x-y|^{d+\alpha}}
\cdot\frac{\delta_D(x)^{\alpha/2}}{{f(x_1)^{\alpha/2}}}\ge
c_{6}\left( \frac{\delta_D(x)^{\alpha/2}
f(x_1)^{\alpha/2}}{t}\wedge1\right)
\frac{t}{|x-y|^{d+\alpha}},\end{align*} where in the first
inequality we used \eqref{l2-4-4}, in the fourth inequality we used
the L\'evy system
\eqref{e2-6}
and \eqref{l5-1-2}, the fifth step follows from \eqref{l2-4-3}, and
the seventh inequality is due to \eqref{l5-2-2} with
$t=2^{-2}c_{3.4.1}f(x_1)^\alpha$ and
$\lambda=1/10$.

(ii)
Case 2: $\delta_D(x)\ge
10^{-1}{(c_{3.4.1}/4)^{1/\alpha}f(x_1)}$ and $\delta_D(y)\ge
2^{-2}t^{1/\alpha}$ (and that $c_{3.4.1}f(x_1)^{\alpha}\le t \le
c_{3.4.2}f(y_1)^{\alpha}$ and $|x-y|\ge t^{1/\alpha}$). Following the arguments as in part (i)
with $V_1$ replaced by
$V_1=B(x,10^{-1}{(c_{3.4.1}/4)^{1/\alpha}f(x_1)})$, and using the fact
that $\Pp^x\big(\tau_{V_1}>2^{-2}c_{3.4.1}f(x_1)^\alpha\big)\ge
c_7$, we can prove
$$
p_D(t,x,y)\ge
\frac{c_8f(x_1)^\alpha}{|x-y|^{d+\alpha}}\ge \frac{c_9 t}{|x-y|^{d+\alpha}}\frac{\delta_D(x)^{\alpha/2}f(x_1)^{\alpha/2}}{t},
$$
where in the last inequality above we used
the fact that  $\delta_D(x) \le c_{10}f(x_1) $ for all $x\in D$.

(iii)
Case 3: $\delta_D(y)\le 2^{-2}t^{1/\alpha}$ (and that $c_{3.4.1}f(x_1)^{\alpha}\le t \le
c_{3.4.2}f(y_1)^{\alpha}$ and $|x-y|\ge t^{1/\alpha}$).
Following arguments in part (iii) in the proof of Lemma \ref{l2-4},
we can verify that there exists $\xi_y^{t}\in D$ such that $B(\xi_y^t, 2^{-1}(t/2)^{1/\alpha})\subset D$,
\eqref{l2-4-5}, \eqref{l2-4-6i} and \eqref{l2-4-6} are satisfied, and
$(c_{3.4.1}/2)f(x_1)^{\alpha}\le t/2 \le c_{11}c_{3.4.2}f(z_1)^{\alpha}$ holds for
all $z\in B(\xi_y^{t},2^{-2}(t/2)^{1/\alpha})$. Therefore, according to
conclusions in parts (i) and (ii)
(by adjusting the constant $c_{3.4.0}$ smaller properly if necessary),
we can obtain for all $z\in
B(\xi_y^{t},2^{-2}(t/2)^{1/\alpha})$,
\begin{align*}
p_D(t/2,x,z)&\ge \frac{c_{12}t}{|x-z|^{d+\alpha}}\left(\frac{\delta_D(x)^{\alpha/2}f(x_1)^{\alpha/2}}{t}\wedge 1\right)\ge \frac{c_{13}t}{|x-y|^{d+\alpha}}\left(\frac{\delta_D(x)^{\alpha/2}f(x_1)^{\alpha/2}}{t}\wedge 1\right),
\end{align*} where we used \eqref{l2-4-6} in the last inequality.
Combining this with \eqref{l2-4-5}, we have
\begin{align*}p_D(t,x,y)=&\int_D p_D(t/2,x,z)p_D(t/2,z,y)\,dz
\ge \int_{B(\xi_y^{t},2^{-2}(t/2)^{1/\alpha})} p_D(t/2,x,z)p_D(t/2,z,y)\,dz\\
\ge&   c_{13}\left(\frac{\delta_D(x)^{\alpha/2}f(x_1)^{\alpha/2}}{t}\wedge 1\right)\frac{t}{|x-y|^{d+\alpha}}\int_{B(\xi_y^{t},2^{-2}(t/2)^{1/\alpha})}p_D(t/2,z,y)\,dz\\
\ge& c_{14} \left(\frac{\delta_D(x)^{\alpha/2}f(x_1)^{\alpha/2}}{t}\wedge 1\right)\left(\frac{\delta_D(y)^{\alpha/2}}{\sqrt{t}}\wedge1\right)\frac{t}{|x-y|^{d+\alpha}}.
\end{align*}

Therefore, the desired assertion follows from all the conclusions
above.
\end{proof}

\begin{lemma}\label{l2-6} {\bf(Upper bound)}\,\,
There exist $c_{3.5.1} \in (0,1]$ and $c_{3.5.2}>0$
such that for all $t>0$ and $x,y\in
D$ with $0<t\le c_{3.5.1}f(y_1)^{\alpha}$ and $|x-y|\ge t^{1/\alpha}$,
$$
p_D(t,x,y)\le
c_{3.5.2}\frac{
t}{|x-y|^{d+\alpha}}\left(\frac{\delta_D(x)^{\alpha/2}\left(f(x_1)^{\alpha/2}\wedge
t^{1/2}\right)}{t}\wedge 1\right)
\left(\frac{\delta_D(y)^{\alpha/2}}{\sqrt{t}}\wedge 1\right).$$
\end{lemma}
\begin{proof}
Since $|x-y|\ge
t^{1/\alpha}$,
for every $u\in D$ such that $|u-x|\le |y-x|/2$, we have
$ |y-u|\ge  |x-y|-|u-x|\ge
|x-y|/2\ge t^{1/\alpha}/2$.
Thus, by Lemma \ref{l:new},
we have that for every $s \le t$ and $y,u\in D$ such that $|u-x|\le |y-x|/2$,
\begin{equation}\label{e:u0}
p_D(s,y,u) \le c_1\frac{s}{|u-y|^{d+\alpha}}\left(
\frac{\delta_D(y)^{\alpha/2}}{{s}^{1/2}}\wedge1\right) = c_1 \frac{{s}^{1/2}\delta_D(y)^{\alpha/2}\wedge s}
{|u-y|^{d+\alpha}}.
\end{equation}

We first
consider the case that $\delta_D(x)\le 2^{-4}c_{2.6.1}t^{1/\alpha}$.
Let
$V_1=B(z_x, 2^{-4}c_{2.6.1}
t^{1/\alpha})\cap D$, $V_3=\{z\in D:
|z-x|\ge |x-y|/2\}$ and $V_2=D\backslash(V_1\cup V_3)$. It is easy to check that $dist (V_1,V_3)>0$.
Then,
applying
\cite[Lemma 5.1]{GKK} and \eqref{l5-2-4}, we can get
\begin{align*}
p_D(t,x,y)
&\le c_2\left(\frac{\Ee^x[\tau_{V_1}]}{t}\wedge1\right)\sup_{0\le s\le t, z\in V_2}p_D(s,z,y)\\
&\quad +c_{3} \left(\int_0^t\Pp^x(\tau_{V_1}>s)\Pp^y(\tau_D>t-s)\,ds\right)\sup_{u\in V_1,z\in V_3}\frac{1}{|u-z|^{d+\alpha}}=:I_1+I_2.
\end{align*}
On the one hand, note that $|z-y|\ge
|x-y|-|z-x|\ge |x-y|/2$ for all $z\in V_2.$ Combining \eqref{e:u0}
with
\eqref{e:l5-4-2}
yields that
\begin{align*}
I_1\le &c_{4}\left( \frac{\delta_D(x)^{\alpha/2} (
f(x_1)^{\alpha/2}\wedge t^{1/2})}{t}\wedge1\right)
\sup_{0\le s\le t, z\in V_2}
\frac{({s}^{1/2}\delta_D(y)^{\alpha/2}\wedge s)}
{|z-y|^{d+\alpha}}\\
\le &c_{5}\left( \frac{\delta_D(x)^{\alpha/2} (
f(x_1)^{\alpha/2}\wedge t^{1/2})}{t}\wedge1\right)
\left(
\frac{\delta_D(y)^{\alpha/2}}
{\sqrt{t}}\wedge1\right)\frac{t}{|x-y|^{d+\alpha}}.
\end{align*}
On the other hand, we write
\begin{align*}&\int_0^t\Pp^x(\tau_{V_1}>s)\Pp^y(\tau_D>t-s)\,ds\\
 &\le\int_0^{t/2} \Pp^x(\tau_{V_1}>s)\,ds \,\Pp^y(\tau_D>t/2) +\int_{t/2}^t \Pp^y
(\tau_D>t-s)\,ds\,
\Pp^x(\tau_{V_1}>t/2)=:I_{2,1}+I_{2,2}. \end{align*}  Note that $t\le c_{3.5.1}f(y_1)^{\alpha} \le c_{3.5.1}$, and let $c_{3.5.1}$ small enough if necessary.  By \eqref{l5-2-1} and
\eqref{e:l5-4-2},
\begin{align*}I_{2,1}&\le  c_{6}\left(\frac{\delta_D(y)^{\alpha/2}}{\sqrt{t}}\wedge 1\right)\int_0^{t/2}\Pp^x(\tau_{V_1}>s)\,ds\le c_{7}\left(\frac{\delta_D(y)^{\alpha/2}}{\sqrt{t}}\wedge 1\right) \left(\Ee^x[\tau_{V_1}]\wedge t\right)\\
&\le c_{8}\left(\frac{\delta_D(y)^{\alpha/2}}{\sqrt{t}}\wedge
1\right)\big((\delta_D(x)^{\alpha/2} (f(x_1)^{\alpha}\wedge
t^{1/2}))\wedge t\big). \end{align*} Similarly, also by
\eqref{l5-2-1} and
\eqref{e:l5-4-2},
\begin{align*}I_{2,2}\le&c_{9}\left(\frac{\Ee^x [\tau_{V_1}]}{t}\wedge1\right)\int_0^{t/2}\Pp^y(\tau_D>s)\,ds \le c_{10}\left(\frac{\Ee^x [\tau_{V_1}]}{t}\wedge1\right)\int_0^{t/2}
\left(\frac{\delta_D(y)^{\alpha/2}}{\sqrt{s}}\wedge 1\right)\,ds\\
\le& c_{11} \left( \frac{\delta_D(x)^{\alpha/2} (
f(x_1)^{\alpha/2}\wedge
t^{1/2})}{t}\wedge1\right)\big((t^{1/2}\delta_D(y)^{\alpha/2})\wedge
t\big).\end{align*} Note that for all $u\in V_1$ and $z\in V_3$,
$$|u-z|\ge |x-y|-|x-u|-|z-y|\ge |x-y|-t^{1/\alpha}/16-|x-y|/2\ge c_{12}
|x-y|.$$ Combining with all the estimates above, we have
$$I_2\le c_{13}\left( \frac{\delta_D(x)^{\alpha/2} ( f(x_1)^{\alpha/2}\wedge t^{1/2})}{t}\wedge1\right)\left(\frac{\delta_D(y)^{\alpha/2}}{\sqrt{t}} \wedge 1\right) \frac{t}{|x-y|^{d+\alpha}}$$ and so
$$p_D(t,x,y)\le  c_{14}\left( \frac{\delta_D(x)^{\alpha/2} ( f(x_1)^{\alpha/2}\wedge t^{1/2})}{t}\wedge1\right)\left(\frac{\delta_D(y)^{\alpha/2}}{\sqrt{t}}
\wedge 1\right) \frac{t}{|x-y|^{d+\alpha}}.$$

  When $\delta_D(x)\ge 2^{-4}
 c_{2.6.1}t^{1/\alpha}$,
we can follow the arguments for the case  $\delta_D(x)\le 2^{-4}c_{2.6.1}t^{1/\alpha}$
above with $V_1$ replaced by $V_1=B(x,
2^{-4}c_{2.6.1}t^{1/\alpha})$
(by noticing that
$\Ee^x[\tau_{V_1}]\le c_{15}t$)
 to prove that
\begin{align*}
p_D(t,x,y)&\le   c_{16}\left(\frac{\delta_D(y)^{\alpha/2}}{\sqrt{t}}
\wedge 1\right) \frac{t}{|x-y|^{d+\alpha}}\\
&\le c_{17}\left( \frac{\delta_D(x)^{\alpha/2} (
f(x_1)^{\alpha/2}\wedge
t^{1/2})}{t}\wedge1\right)\left(\frac{\delta_D(y)^{\alpha/2}}{\sqrt{t}}
\wedge 1\right) \frac{t}{|x-y|^{d+\alpha}},
\end{align*}where in the second inequality we used the fact that
$t^{1/\alpha}
\le c_{18}f(x_1)$, thanks to the property that
$\delta_D(x)\le c_{19} f(x_1)$ for all $x\in D$. The proof is complete.
\end{proof}

Notice that, if $t \le
C_0f(y_1)^{\alpha}$ and $|x-y|\le t^{1/\alpha}$,  then $t\le c_1 C_0f(x_1)^\alpha$ for some constant $c_1\ge1$. Therefore, putting all the previous lemmas in this section together yields the
following statement.
\begin{proposition}\label{p2-1}
There is a constant $C_0\in (0, 1]$
such that for all $t>0$ and $x,y\in D$ such that  for all $0<t\le
C_0f(y_1)^{\alpha}$,
\begin{align*}
p_D(t,x,y)\simeq &\,
p(t,x,y)\left(\frac{\delta_D(x)^{\alpha/2}\left(f(x_1)^{\alpha/2}\wedge
t^{1/2}\right)}{t}\wedge 1\right)
\left(\frac{\delta_D(y)^{\alpha/2}}{\sqrt{t}}\wedge 1\right)
\simeq p(t,x,y)
\Psi(t,x)\Psi(t,y) ,\end{align*} where $\Psi(t,x)$ is defined by \eqref{e:DPsi}.
\end{proposition}

\smallskip

The next two sections are devoted to estimates of $p_D(t,x,y)$
for the case that $t\ge C_0f(y_1)^\alpha$, where $C_0$ is
the fixed constant given
 in Proposition \ref{p2-1}. For this, we define for any $y\in D$,
\begin{equation}\label{e:lllsss}t_0(y):=\inf\Big\{t>0: e^{-C_*f(y_1)^{-\alpha}t}\le t(1+|y|)^{-(d+\alpha-1)}\Big\},\end{equation}
where $C_*=c_{2.9.2}>0$ is given in \eqref{l3-1-0}.
\begin{remark}\label{remark4.1} As mentioned in the remark below \eqref{e:t0}, $t_0(y)\in (0,\infty)$ is unique and satisfies that
$$e^{-C_*f(y_1)^{-\alpha}t_0(y)}= t_0(y)(1+|y|)^{-(d+\alpha-1)}.$$ In particular, we can check that there is a constant $C_{3.10}>0$ such that for all $y\in D$,
\begin{equation}\label{e:lllsss11}f(y_1)^{\alpha}\le C_{3.10} t_0(y).\end{equation}
Usually it is not easy to obtain the explicit value of $t_0(y)$;
however, we are able to get explicit estimates of $t_0(y)$
under some mild assumption on $f$. For example, if $f(r)\ge
c(1+r)^{- p}$ for some constants $c$ and $p>0$, then
$t_0(y)\simeq f(y_1)^\alpha
\log(2+|y|)$ for all $y\in D$.
\end{remark}

\section{Case II: $C_0f(y_1)^\alpha\le t\le C_1t_0(y)$ for any given constant $C_1>0$.}

Throughout this section, we always let $C_0$ be the constant in Proposition \ref{p2-1}, and $t_0(y)$ be defined by \eqref{e:lllsss} for any $y\in D$.

\begin{lemma}\label{l3-3} {\bf (Lower bound)}\,\,
There exist  constants
$c_{\ac{1}}\in (0,1)$ and $c_{\ac{2}}>0$ such that for all
 $x,y\in D$ and $C_0f(y_1)^{\alpha}\le t$,
\begin{align*}
p_D(t,x,y)\ge& c_{\ac{1}}\left(t^{-d/\alpha}\wedge \frac{t}{|x-y|^{d+\alpha}}\right)
\left(\frac{\delta_D(x)^{\alpha/2}f(x_1)^{\alpha/2}}{t\wedge1}\right)
\left(\frac{\delta_D(y)^{\alpha/2}f(y_1)^{\alpha}}{t\wedge1}\right)e^{-c_{\ac{2}}tf(y_1)^{-\alpha}}.
\end{align*}
\end{lemma}
\begin{proof}
(i)
Case 1: $|x-y|\le t^{1/\alpha}\wedge1$. According to Lemma \ref{l5-1},
 we can find
$\xi_x,\xi_y\in D$ and a constant $\lambda:=\lambda(C_0)\in (0,1)$ small enough
such that $V_x:=B(\xi_x, \lambda f(x_1)) \subset B(\xi_x, 4\lambda f(x_1))\subset D$, $V_y:=B(\xi_y, \lambda
f(y_1)) \subset B(\xi_y, 4\lambda
f(y_1))\subset D$, and
\begin{equation}\label{l3-3-2}
\begin{split}
&\int_{V_x}p_D\left(2^{-2}C_0f(x_1)^\alpha, x,z\right)\,dz\ge c_1
\delta_D(x)^{\alpha/2} f(x_1)^{-\alpha/2},\\
&\int_{V_y}p_D\left(2^{-2}C_0f(y_1)^\alpha, y,z\right)\,dz\ge
c_1\delta_D(y)^{\alpha/2} f(y_1)^{-\alpha/2}.
\end{split}
\end{equation}
 Here, we  used the fact that
$\delta_D(x)\le c_2 f(x_1)$ for all $x\in D$ with some constant
$c_2>0$.

On the other hand, for any $z\in V_x$, $w\in V_y$ and
$t\ge 2^{-1}{C_0f(y_1)^{\alpha}}$, taking
$n:=n(t,y)=[2t/(C_0f(y_1)^\alpha)]+1$ and
$\bar c:=\bar c(t,y)=tf(y_1)^{-\alpha}n^{-1}$,
we have
\begin{align*}
 p_D(t,z,w)
&=\int_{D}\cdots\int_Dp_D(t/n,z,z_1)\cdots p_D(t/n,z_{n-1},w)\,dz_1\cdots dz_{n-1}\\
&\ge \int_{V_y}\cdots
\int_{V_y}
p_D(\bar c f(y_1)^\alpha, z,z_1)\cdots
p_D(\bar c f(y_1)^\alpha, z_{n-1},w)\,dz_1\cdots dz_{n-1},
\end{align*}
where in the inequality above we used the facts that $V_y\subset D$.

The assumption $|x-y|\le t^{1/\alpha}\wedge1$ implies that $f(x_1)\simeq f(y_1)$.
Using this and $C_0f(y_1)^\alpha\le t\wedge1$, we have that,
for all $z\in V_x$ and $u\in V_y$, $|z-u|\le
c_3(t^{1/\alpha}\wedge1)$, $\delta_D(u)\ge \lambda f(y_1)$ and
$\delta_D(z)\ge \lambda f(x_1)\ge c_4\lambda f(y_1)$.
Hence, according to Lemma \ref{L:2.2}, we
obtain that for $z\in V_x$ and $u\in V_y$,
$$p_D(\bar c
f(y_1)^\alpha, z,u)
\ge c_5\left(\frac{f(y_1)^\alpha}{|z-u|^{d+\alpha}}\wedge (f(y_1)^\alpha)^{-d/\alpha}\right)
\ge c_6t^{-(d+\alpha)/\alpha}f(y_1)^{\alpha},$$
where the last inequality is due to the fact that $|z-u|\le c_3(t^{1/\alpha}\wedge 1)$ and $t\ge C_0f(y_1)^\alpha$. We mention that, since $\bar c\in [C_0/4,C_0/2]$ (i.e.,
$\bar c$ may depend on $y$ and $t$ but it is uniformly bounded between $C_0/4$ and $C_0/2$), $c_5>0$ here is independent of $y$ and $t$ due to
the argument in \cite[Proposition 3.3]{CKS1}.
Similarly, we have
$p_D(\bar c
f(y_1)^\alpha, w,u)\ge c_7f(y_1)^{-d}$ for $w,u\in V_y.$
Hence, putting all the  estimates above together yields that for all
$z\in V_x$, $w\in V_y$ and $t\ge 2^{-1}{C_0f(y_1)^{\alpha}}$,
\begin{equation}\label{e:kkll7}\begin{split}
p_D(t,z,w)&\ge \Big(c_7f(y_1)^{-d} |B(\xi_y, \lambda f(y_1))|\Big)^{n-1} c_6t^{-(d+\alpha)/\alpha}f(y_1)^{\alpha}\ge
c_8f(y_1)^{-d}e^{-c_9tf(y_1)^{-\alpha}},
\end{split}\end{equation}
where the last inequality follows from the facts that
$n=[2t/(C_0f(y_1)^\alpha)]+1$ and
$\left(tf(y_1)^{-\alpha}\right)^{-(d+\alpha)/\alpha}$ $\ge c_{10}e^{
-c_{11}t
f(y_1)^{-\alpha}}$ for each
$t\ge 2^{-1}{C_0f(y_1)^{\alpha}}$. Therefore, for all $t\ge
{C_0f(y_1)^{\alpha}}$,
\begin{align*}p_D (t,x,y)
=&\int_D\int_D p_D\left(2^{-2}C_0f(x_1)^\alpha,x,z\right)  p_D\left(t-2^{-2}C_0f(x_1)^\alpha-
2^{-2}C_0f(y_1)^\alpha
,z,w\right)\\
&\qquad\quad\times p_D\left(2^{-2}C_0f(y_1)^\alpha
,w,y\right)\,dz\,dw\\
\ge& \left[\int_{V_x} p_D\left(2^{-2}C_0f(x_1)^\alpha, x,z\right)\,dz \right]
\left[\int_{V_y}p_D\left(2^{-2}C_0f(y_1)^\alpha
,w,y\right)\,dw\right] \\
&\qquad\quad\times \inf_{z\in V_x,w\in V_y} p_D\left(t-2^{-2}C_0f(x_1)^\alpha-
2^{-2}C_0f(y_1)^\alpha
,z,w\right)\\
\ge & c_{12}(\delta_D(x)^{\alpha/2} f(x_1)^{-\alpha/2}) (\delta_D(y)^{\alpha/2} f(y_1)^{-\alpha/2})\\
& \qquad\quad\times \exp\left(-c_9\left(t-
2^{-2}C_0f(y_1)^\alpha
-2^{-2}C_0f(x_1)^\alpha\right)f(y_1)^{-\alpha}\right) f(y_1)^{-d}\\
\ge& c_{12}(\delta_D(x)^{\alpha/2} f(x_1)^{-\alpha/2})(\delta_D(y)^{\alpha/2} f(y_1)^{-\alpha/2})\exp\left(-
c_{9}tf(y_1)^{-\alpha}\right) f(y_1)^{-d}\\
=& c_{12}\left( \frac{\delta_D(x)^{\alpha/2} f(x_1)^{\alpha/2}}{t} \right)\left(\frac{\delta_D(y)^{\alpha/2} f(y_1)^{\alpha/2}}{t}
\right)  \exp\left(-c_{9}tf(y_1)^{-\alpha}\right)\left(tf(y_1)^{-\alpha}\right)^{2+d/\alpha}t^{-d/\alpha}\\
\ge & c_{13}\left( \frac{\delta_D(x)^{\alpha/2}
f(x_1)^{\alpha/2}}{t\wedge1} \right)\!\!\!
\left(\frac{\delta_D(y)^{\alpha/2} f(y_1)^{\alpha/2}}{t\wedge1}
\right) e^{-
c_{14}tf(y_1)^{-\alpha}}
t^{-d/\alpha},  \end{align*}
where the second inequality follows from \eqref{l3-3-2},
\eqref{e:kkll7} and the fact that
$$t-2^{-2}C_0f(x_1)^\alpha-
2^{-2}C_0f(y_1)^\alpha
\ge
{t}/{2}\quad \hbox{ for } t\ge C_0f(y_1)^\alpha,$$ (thanks to
$t\ge
C_0f(y_1)^\alpha\ge C_0f(x_1)^\alpha$),
 and  the last inequality is due to $$\left(\frac{t\wedge 1}{t}\right)^2\left(tf(y_1)^{-\alpha}\right)^{2+d/\alpha} \ge c_{15}e^{-c_{16}tf(y_1)^{-\alpha}}\quad \hbox{ for } t\ge C_0f(y_1)^\alpha.$$

(ii) Case 2: $|x-y|\ge
t^{1/\alpha}\wedge1 $. Let $V_x=B(\xi_x,\lambda f(x_1))$ and
$V_y=B(\xi_y, \lambda f(y_1))$ be those defined
in part (i).
By  Lemma \ref{l5-1}, we have $|x-\xi_x|\le c_{17}\lambda
f(x_1)$ and $|y-\xi_y|\le c_{17}\lambda f(y_1)$. Choosing
$\lambda>0$ small enough if necessary, we find that for every $z\in
B\big(\xi_x,2\lambda f(x_1)\big)$ and $w\in B\big(\xi_y, 2\lambda
f(y_1)\big)$,
\begin{align*}
|z-w|& \ge |x-y|-|x-\xi_x|-|z-\xi_x|-|y-\xi_y|-|w-\xi_y| \ge |x-y|-c_{18}\lambda f(y_1)\ge c_{19}|x-y|
\end{align*}
and, similarly,
\begin{equation}\label{l3-3-4}
|z-w| \le |x-y|+c_{18}\lambda f(y_1)\le c_{20}|x-y|,
\end{equation}
where we have used the fact that $|x-y|\ge
(t^{1/\alpha}\wedge1)\ge C_0^{1/\alpha} f(y_1)$ (because $C_0, f \in (0,1]$).   In particular, $B\big(z,\lambda
f(x_1)\big) \cap B\big(w,\lambda f(y_1)\big)=\emptyset$ for
every $z\in V_x$ and $w\in V_y$. Therefore, for any $z\in V_x$,
$w\in V_y$ and $t>C_0f(y_1)^\alpha/2$,
\begin{align*}&p_D(t,z,w)=\Ee^z[p_D(t/2,X_{t/2}^D,w)]\\
&\ge \Ee^z\Big[p_D(t/2, X_{t/2},w) :0<\tau_{B(z,\lambda f(x_1))}<2^{-2} {C_0f(x_1)^\alpha}, X_{\tau_{B(z, \lambda f(x_1))}}\in B(w, \lambda  f(y_1)/2),\\
&\qquad \quad
 X_s\in B(w,\lambda f(y_1))\text{ for all }s\in [\tau_{B(z,\lambda f(x_1))}, \tau_{B(z,\lambda f(x_1))}+t]\Big]\\
&\ge c_{21} e^{-c_{22}t f(y_1)^{-\alpha}} f(y_1)^{-d}\inf_{u\in B(w,\lambda f(y_1)/2)} \Pp^u (\tau_{B(u,\lambda f(y_1)/2)}>t)\\
&\qquad\quad\times\left(\int_0^{2^{-2}{C_0f(x_1)^\alpha}}\!\!\!\!\int_{B(z, \lambda f(x_1))}\!\!
p_{B(z,\lambda f(x_1))}(s,z,u)\!\!\int_{B(w,\lambda f(y_1)/2)}\!\!\frac{1}{|u-v|^{d+\alpha}}\,dv\,du\,ds\right)\\
&\ge c_{23} e^{-c_{22}t f(y_1)^{-\alpha}} f(y_1)^{-d}e^{-c_{24}t f(y_1)^{-\alpha}}\\
&\qquad\quad \times f(x_1)^\alpha\Pp^z (\tau_{B(z,\lambda f(x_1)/2)}>2^{-2}C_0f(x_1)^\alpha ) \big|B(w,\lambda f(y_1))\big| \frac{1}{|x-y|^{d+\alpha}}\\
&\ge c_{25} e^{-c_{26}t f(y_1)^{-\alpha}}\frac{1}{|x-y|^{d+\alpha}} f(x_1)^\alpha.
\end{align*}
Here
the first inequality is due to L\'evy system \eqref{e2-6},
the second inequality follows from
$$
\inf_{u,w\in B(\xi_y,2\lambda f(y_1))}p_D(t/2,u,w)\ge
c_{21}f(y_1)^{-d} e^{-c_{22}tf(y_1)^{-\alpha}},
$$
which is a direct consequence of the argument for \eqref{e:kkll7} (by choosing $\lambda$ small enough if necessary),
in the third inequality we have used \eqref{l3-3-4} and the estimate
as follows
\begin{equation}\label{l3-3-5}\begin{split}
\Pp^u\big(\tau_{B(u,r)}>t\big)&=
\Pp^u\big(\tau_{B(u,1)}>t/r^\alpha\big) =\int_{B(u,1)}p_{B(u,1)}\big(t/r^\alpha,u,z\big)dz\ge
c_{26}e^{-c_{27}tr^{-\alpha}},\quad t\ge r^\alpha,
\end{split}\end{equation}
which is deduced from the scaling property of symmetric
$\alpha$-stable processes and \eqref{hk-s1}, and the fourth inequality is due to \eqref{l5-1-2}.
Therefore, combining this with \eqref{l3-3-2}, we arrive at that for
all $t\ge {C_0f(y_1)^{\alpha}}$,
\begin{align*}
 p_D(t,x,y)
&\ge\int_{V_x}\int_{V_y} p_D\left(2^{-2}C_0f(x_1)^\alpha,x,z\right)
p_D\left(t-2^{-2}C_0f(x_1)^\alpha-
2^{-2}C_0f(y_1)^\alpha
, z,w\right)\\
&\qquad\quad \times p_D\left(2^{-2}C_0f(x_1)^\alpha,w,y\right)\,dz\,dw\\
&\ge \left(\int_{V_x} p_D\left(2^{-2}C_0f(x_1)^\alpha,x,z\right)\,dz \right)\left(\int_{V_y} p_D\left(2^{-2}C_0f(x_1)^\alpha,w,y\right)\,dw\right) \\
&\qquad\quad\times\inf_{z\in V_x,w\in V_y}
p_D\left(t-2^{-2}C_0f(x_1)^\alpha-
2^{-2}C_0f(y_1)^\alpha
, z,w\right)\\
&\ge c_{28}\left(\delta_D(x)^{\alpha/2} f(x_1)^{-\alpha/2}\right)
\left(  \delta_D(y)^{\alpha/2} f(y_1)^{-\alpha/2}\right)\left(  e^{-c_{29}t f(y_1)^{-\alpha}}\frac{1}{|x-y|^{d+\alpha}} f(x_1)^\alpha\right)\\
&= c_{28}\frac{\delta_D(x)^{\alpha/2} f(x_1)^{\alpha/2}}{t}  \frac{\delta_D(y)^{\alpha/2} f(y_1)^{\alpha/2}}{t}
\frac{t}{|x-y|^{d+\alpha}}e^{-c_{29}tf(y_1)^{-\alpha}}\left(tf(y_1)^{-\alpha}\right)\\
&\ge c_{30} \left( \frac{\delta_D(x)^{\alpha/2}
f(x_1)^{\alpha/2}}{t\wedge1}\right)\left(\frac{\delta_D(y)^{\alpha/2}
f(y_1)^{\alpha/2}}{t\wedge1}\right)
e^{-2c_{29}
tf(y_1)^{-\alpha}}
\frac{t}{|x-y|^{d+\alpha}},
\end{align*}
where the last inequality follows from
the inequality $({t\wedge 1}/{t})^2\left(tf(y_1)^{-\alpha}\right)\ge c_{31}e^{-c_{29}tf(y_1)^{-\alpha}}$ for $ t\ge C_0f(y_1)^\alpha.$
We complete the proof.
\end{proof}

\begin{lemma}\label{l3-2} {\bf (Upper bound)}\,\,
For any $c_0>0$,
there exists  a constant
$c_{\ac{1}}:=c_{\ac{1}}(c_0)>0$ such  that for all $x,y\in D$
and
$C_0f(y_1)^{\alpha}\le t \le
c_0t_0(y)$,
  \begin{align*}
p_D(t,x,y)\le& c_{\ac{1}}\left(t^{-d/\alpha}\wedge \frac{t}{|x-y|^{d+\alpha}}\right)
\left(\frac{\delta_D(x)^{\alpha/2}f(x_1)^{\alpha/2}}{t\wedge1}\right)
\left(\frac{\delta_D(y)^{\alpha/2}f(y_1)^{\alpha}}{t\wedge1}\right)
e^{-\frac{C_*tf(y_1)^{-\alpha}}
{2c_0 \vee 4}
},
\end{align*} where $C_*=c_{2.9.2}$ is given in \eqref{l3-1-0}.
\end{lemma}
\begin{proof}
Without loss of generality we
may assume that $c_0 \ge 2$.

(i) Case 1: $|x-y|\le t^{1/\alpha}$.
Using \eqref{l3-1-0} and considering the cases $C_0f(y_1)^\alpha<t/2 \le t_0(y)$ and $t_0(y) <t/2\le c_0t_0(y)/2$, we know that for any
$y\in D$ and $t>0$ with $C_0f(y_1)^\alpha<t\le c_0t_0(y)$,
\begin{equation}\label{e:op1}
\Pp^y(\tau_D> t/2)\le  c_1\left(\frac{\delta_D(y)^{\alpha/2}
f(y_1)^{\alpha/2}}{t\wedge1}\right) e^{-2^{-1}(C_*/c_0)f(y_1)^{-\alpha}t},
\end{equation}
where we used the fact that for every $t_0(y) <t/2\le c_0t_0(y)/2$,
\begin{equation}\label{l3-2-1}
\begin{split}
\frac{t}{(1+|y|)^{d+\alpha-1}}
&\le \frac{c_0t_0(y)}{(1+|y|)^{d+\alpha-1}} =c_0e^{-C_*t_0(y)f(y_1)^{-\alpha}}\le c_0e^{-(C_*/c_0)f(y_1)^{-\alpha}t}.
\end{split}
\end{equation}
 Let $c_2:=c_2(c_0)=C_*/c_0$.
Then, for any $y,z\in D$ and $t>0$ with $C_0f(y_1)^\alpha<t\le c_0t_0(y)$,
\begin{align*}
p_D(2t/3,z,y)
=&\int_D p_D(t/6,z,u)p_D(t/2,u,y)\,du\le c_3t^{-d/\alpha}\Pp^y(\tau_D>t/2)\\
\le& c_4t^{-d/\alpha}\left(\frac{\delta_D(y)^{\alpha/2}
f(y_1)^{\alpha/2}}{t\wedge1}\right) e^{-c_2f(y_1)^{-\alpha}t/2}.
\end{align*}
Hence, for any $x,y\in D$ and $t>0$ with $C_0f(y_1)^\alpha<t\le c_0t_0(y)$,
\begin{align*}p_D(t,x,y)&= \int_D p_D(t/3,x,z)p_D(2t/3,z,y)\,dz\\
&\le c_4t^{-d/\alpha}\left(\frac{\delta_D(y)^{\alpha/2}
f(y_1)^{\alpha/2}}{t\wedge1}\right)
e^{-c_2f(y_1)^{-\alpha}t/2}\int_D p(t/3,x,z)\,dz\\
&\le c_5t^{-d/\alpha}\left(\frac{\delta_D(y)^{\alpha/2}
f(y_1)^{\alpha/2}}{t\wedge1}
\right)\left(\frac{\delta_D(x)^{\alpha/2}
f(x_1)^{\alpha/2}}{t\wedge1}\right)e^{-c_2f(y_1)^{-\alpha}t/2},
\end{align*}
where in the last inequality we have used \eqref{l3-1-0} and the fact that $C_0f(x_1)^\alpha\le C_0f(y_1)^\alpha\le t$.

(ii) Case 2:
$|x-y|\ge t^{1/\alpha}$.
 Let $V_1=\{z\in D: |z-y|>
|x-y|/2\}$ and $V_2=\{z\in D: |z-y|\le |x-y|/2\}$. Then, it holds that for any $x,y\in D$ and $t>0$, \begin{align*}
p_D(t,x,y)=\int_{V_1}p_D(t/2, x,z)p_D(t/2,z,y)\,dy+\int_{V_2}p_D(t/2, x,z)p_D(t/2,z,y)\,dy=:I_1+I_2.
\end{align*} On the one hand, for any $z\in V_1$,
$|z-y|\ge |x-y|/2\ge t^{1/\alpha}/2$. Then,
by Lemma \ref{l:new},
for
any $z\in V_1$, $y\in D$ and $t>0$ with $|x-y|\ge t^{1/\alpha}$,
\begin{align*}p_D(t/2,z,y)\le &
c_7\left(\frac{\delta_D(y)^{\alpha/2}f(y_1)^{\alpha/2}}{t\wedge1}\right)\frac{ t }{|z-y|^{d+\alpha}}
\le
c_8\left(\frac{\delta_D(y)^{\alpha/2}f(y_1)^{\alpha/2}}{t\wedge1}\right)\frac{
t}{|x-y|^{d+\alpha}}.\end{align*} According to \eqref{l3-1-0}, for all $x,y\in D$ with
$x_1\ge y_1$
and $t>0$ with $C_0f(y_1)^{\alpha}\le t \le
c_0t_0(y)$,
\begin{align*}\Pp^x(\tau_D>t/2)\le &
c_9\left(\frac{\delta_D(x)^{\alpha/2}
f(x_1)^{\alpha/2}}{t\wedge1}\right)\left(e^{-C_*f(x_1)^{-\alpha}t/2}+
t(1+|x|)^{-(d+\alpha-1)}\right)\\
\le&  c_{10}\left(\frac{\delta_D(x)^{\alpha/2}
f(x_1)^{\alpha/2}}{t\wedge1}\right)\left(e^{-C_*f(y_1)^{-\alpha}t/2}+
t(1+|y|)^{-(d+\alpha-1)}\right)\\
\le& c_{11}\left(\frac{\delta_D(x)^{\alpha/2}
f(x_1)^{\alpha/2}}{t\wedge1}\right)
e^{-2^{-1}(C_*/c_0)f(y_1)^{-\alpha}t},\end{align*}
 where
 in the second inequality we used the fact that $1+|x|\ge c_{12}(1+|y|)$ for all $x,y\in D$ with $x_1\ge y_1$,  and the last inequality
follows from \eqref{l3-2-1}.
Hence, for all $x,y\in D$ and $t>0$ with $C_0f(y_1)^{\alpha}\le t \le c_0t_0(y)$,
\begin{align*}I_1&\le c_8\left(\frac{\delta_D(y)^{\alpha/2}f(y_1)^{\alpha/2}}{t\wedge1}\right)\frac{
t}{|x-y|^{d+\alpha}}\Pp^x(\tau_D>t/2)\\
&\le
c_{13}\left(\frac{\delta_D(x)^{\alpha/2}f(x_1)^{\alpha/2}}{t\wedge1}\right)\left(\frac{\delta_D(y)^{\alpha/2}f(y_1)^{\alpha/2}}{t\wedge1}\right)\frac{
t}{|x-y|^{d+\alpha}}e^{-c_{2}f(y_1)^{-\alpha}t/2}.
\end{align*}
On the other hand, for every $z\in V_2$, $|z-x|\ge |x-y|/2\ge
t^{1/\alpha}/2$.
So, according to Lemma \ref{l:new}, we obtain that for every $z\in V_2$,
$$p_D(t/2,x,z)\le
c_{14}\left(\frac{\delta_D(x)^{\alpha/2}f(x_1)^{\alpha/2}}{t\wedge1}\right)\frac{
t}{|x-y|^{d+\alpha}}.$$ This along with \eqref{e:op1} yields that for all $x,y\in D$ and $t>0$ with $C_0f(y_1)^{\alpha}\le t \le
c_0t_0(y)$,
\begin{align*}I_2\le & c_{14}\left(\frac{\delta_D(x)^{\alpha/2}f(x_1)^{\alpha/2}}{t\wedge1}\right)
\frac{t}{|x-y|^{d+\alpha}}\Pp^y\big(\tau_D>t/2\big)\\
\le&
c_{15}\left(\frac{\delta_D(x)^{\alpha/2}f(x_1)^{\alpha/2}}{t\wedge1}\right)\left(\frac{\delta_D(y)^{\alpha/2}f(y_1)^{\alpha/2}}{t\wedge1}\right)
\frac{ t}{|x-y|^{d+\alpha}} e^{-c_{2}t
f(y_1)^{-\alpha}/2}.\end{align*} Therefore, according to all the
estimates above,  for any $x,y\in D$ and $t>0$ with $C_0f(y_1)^{\alpha}\le t \le
c_0t_0(y)$,
\begin{align*}p_D(t,x,y)\le & c_{16}\left(\frac{\delta_D(x)^{\alpha/2}f(x_1)^{\alpha/2}}{t\wedge1}\right) \left(\frac{\delta_D(y)^{\alpha/2}f(y_1)^{\alpha/2}}{t\wedge1}\right)
\frac{ t}{|x-y|^{d+\alpha}} e^{-c_{2 }t
f(y_1)^{-\alpha}/2}.\end{align*}

Now, the required assertion follows from both conclusions above.
\end{proof}
We
summarize  both lemmas above as follows.
\begin{proposition}\label{p3-1} Let $\Psi(t,x)$ be defined by \eqref{e:DPsi}, and $C_*=c_{2.9.2}$ be given  in \eqref{l3-1-0}. Then the following hold.
\begin{itemize}
\item[(i)] There exist  constants
$c_{\ac{1}}, c_{\ac{2}},c_{\ac{3}}>0$
such that for all
 $x,y\in D$
 with  $C_0f(y_1)^{\alpha}\le t $,
\begin{align*}
p_D(t,x,y)\ge &c_{\ac{1}}p(t,x,y)\left(\frac{\delta_D(x)^{\alpha/2}f(x_1)^{\alpha/2}}{t\wedge1}\right)\left(\frac{\delta_D(y)^{\alpha/2}f(y_1)^{\alpha/2}}{t\wedge1}\right)e^{-c_{\ac{2}}tf(y_1)^{-\alpha}}\\
\ge&c_{\ac{3}}p(t,x,y)
\Psi(t,x)\Psi(t,y)e^{-c_{\ac{2}}tf(y_1)^{-\alpha}}.
\end{align*}

\item[(ii)] For any $c_0 \ge 1$, there exist constants $c_{\ac{4}}:=c_{\ac{4}}(c_0), c_{\ac{5}}:=c_{\ac{5}}(c_0)>0$ such that   for all $x,y\in D$
 with $C_0f(y_1)^{\alpha}\le t \le
 c_0t_0(y)=c_0t_0(C_*, y)$,
\begin{align*}
p_D(t,x,y)\le& c_{\ac{4}}p(t,x,y)
\left(\frac{\delta_D(x)^{\alpha/2}f(x_1)^{\alpha/2}}{t\wedge1}\right)\left(\frac{\delta_D(y)^{\alpha/2}f(y_1)^{\alpha/2}}{t\wedge1}\right)e^{-(2c_0 \vee 4)^{-1}C_*tf(y_1)^{-\alpha}}\\
\le& c_{\ac{5}}p(t,x,y)
\Psi(t,x)\Psi(t,y)e^{-(2c_0 \vee 4)^{-1}C_*tf(y_1)^{-\alpha}}.
\end{align*} \end{itemize}
\end{proposition}

\begin{remark} In Proposition \ref{p3-1}, we do not require $t_0(y)$ to be bounded.
Actually, we will
treat all cases including  $\lim_{y\in D,|y|\to\infty} t_0(y)>0$
(which in particular
includes the case that
$\lim_{y\in D,|y|\to\infty} t_0(y)=\infty$)
  in
the next section.
When $\lim_{y\in D,|y|\to\infty} t_0(y)>0$, Proposition \ref{p3-1} has shown the explicit heat kernel estimates
for any finite time.
\end{remark}

\section{Case III: $t\ge C_1t_0(y)$ for some $C_1>0$}\label{section5}
In this section, we will make additional assumptions on the reference function $f$ as in Theorem \ref{Main}:
\begin{itemize}
\item[(i)] There exist constants $c, p>0$ such that $f(s)\ge c(1+s)^{-p}$ for all $s>0$;
\item[(ii)] There is a monotone function $g$ on $(0,\infty)$ such that $g(s)\simeq f(s)^{\alpha}\log(2+s)$.
\end{itemize}
As mentioned in Remark \ref{remark4.1}, under (i), for any $y\in D$, $t_0(y)\simeq
f(y_1)^\alpha
\log(2+|y|)$, where
$t_0(y)=t_0(C_*, y)$ is defined by \eqref{e:lllsss} and $C_*=c_{2.9.2}$ is the constant in \eqref{l3-1-0}.
According to the different monotone property of $g$, we will split this section into two parts.

 \subsection{\bf{Case III-1: $g$ is
non-increasing on $(0,\infty)$ such that $\lim\limits_{s\to\infty} g(s)=0$}}
 In this part, we are concerned with
 the case that
 $g$ is
non-increasing on $(0,\infty)$ and  $\lim\limits_{s\to\infty} g(s)=0$.
 Since $t_0(y)\simeq
f(y_1)^\alpha\log(2+|y|)$, we have
$\lim\limits_{y\in D,|y|\rightarrow \infty}t_0(y)=0$.

For any $t>0$, define
 \begin{align}\label{e:index}
s_0(t)=\inf\{s>0: f(s)^\alpha\le t\}
\vee 2,\quad s_1(t)=
g^{-1}(t)\vee 2,
\end{align}
where
$g^{-1}(t)=\inf\{s\ge 0: g(s)\le t\}$ and we use the convention that
$\inf\emptyset=\infty$.
It is clear that there exists a constant $C_{5.2}\in(0,1]$ such that
  \begin{align}C_{5.2} s_0(t)\le s_1(t) \quad \text{for all }t>0.\label{e:C3} \end{align}

Recall that
$C_0$ is the constant in Proposition \ref{p2-1}, $C_{3.10}$ is the constant given in \eqref{e:lllsss11} and  $\phi$ is the function defined in
\eqref{e1-1}.
\begin{lemma}\label{l4-2}{\bf (Lower bound when $C_1t_0(y)\le t\le C$
for any $C_1\ge C_0C_{3.10}$ and any $C>0$.)}\,\,
Suppose that $g$ is
non-increasing on $(0,\infty)$ such that  $\lim\limits_{s\to\infty} g(s)=0$. Then, for every $c_1 \ge C_0C_{3.10}$ and  $c_2>0$,
there exist positive constants $c_{\ac{1}}, c_{\ac{2}}, c_{\ac{3}}, c_{\ac{4}}$ $($depending on $c_1$ and $c_2)$
 such that
for every $y\in D$  and $c_1t_0(y)\le t\le c_2$,
\begin{align*}
p_D(t,x,y)\ge  &c_{\ac{1}}
\phi(x)\phi(y)
\int_{0}^{c_{\ac{2}}s_1(c_{\ac{3}}t)}
f(s)^{d-1}e^{-c_{\ac{4}}tf(s)^{-\alpha}}
\,ds.
 \end{align*}
 \end{lemma}
\begin{proof}
Fix $c_1 \ge C_0C_{3.10}$
and $c_2>0$.
 Since  $s_1(t_0(y))\asymp |y|\vee 2$,
there exist $c_3,c_4$
 such that
$c_3s_1(c_4t)\le (|y|\vee2)/2$ for all $t\ge c_1
t_0(y)$. 
Recall that we
take $c_1\ge C_0C_{3.10}$, and
assume that $D$ is a $C^{1,1}$-horn-shaped region satisfying $\{x\in D: x_1>2\}=D_f^2$
and $C_0f(y_1)^\alpha \le c_1
t_0(y)$
for all $y \in D$. 
Recall also that we assume that $x_1\ge y_1$.
Then, one can  choose $M\ge 2$ large enough
so that $|x| \ge 2|y|/3$ for every $|y|>M$.

We first consider the case that $|y|\le M$.
Note that there exists $c_0>0$
such that $t_0(y)>c_0$ for all $y\in D$ with $|y| \le M$.
Since
$s_1(c_4
t)\le (2c_3)^{-1}(|y|\vee2)\le (2c_3)^{-1}M$ and
$(c_1
c_0) \vee (C_0f(y_1)^\alpha) \le c_1t_0(y)$,
by Proposition \ref{p3-1}
(i)
 we have that
for every $y\in D$ with $|y| \le M$ and any
$c_1t_0(y)\le t\le c_2$,
\begin{equation}\label{e:refe-1}\begin{split}
 p_D(t,x,y)
&\ge   c_5\delta_D(x)^{\alpha/2}f(x_1)^{\alpha/2}\delta_D(y)^{\alpha/2}f(y_1)^{\alpha/2}
 \left(1 \wedge \frac{1}{|x-y|^{d+\alpha}} \right)e^{-c_6f(M)^{-\alpha}}\\
&\ge   c_7
\phi(x)\phi(y)
\int_0^{s_1(c_4 t)}
f(s)^{d-1}e^{-c_8tf(s)^{-\alpha}}ds.
 \end{split}\end{equation}
 Here in the second inequality
 we used the facts that
$|x-y|\le c_{9}(1+|x|)$ and
 $s_1(c_4
 t)\le (2c_3)^{-1}(|y|\vee2)\le (2c_3)^{-1}M$ for every
 $|y|\le M$ yielding
$$\int_0^{s_1(c_4t)}f(s)^{d-1}e^{-c_{8}tf(s)^{-\alpha}}\,ds\le
\int_0^{(2c_3)^{-1}M}f(s)^{d-1}e^{-c_{8}tf(s)^{-\alpha}}\,ds\le c_{10}.$$

 For the remainder of the proof, we assume that  $|y| > M$.
Recall that
$c_3s_1(c_4t)
\le |y|/2$ for all $t\ge c_1t_0(y)$.
According to Propositions \ref{p2-1} and \ref{p3-1},
for all $x,y,z\in D$ and
 $c_2\ge t\ge c_1t_0(y)$ with
$z_1\le c_3s_1(c_4t)\le |y|/2$
(which implies that
$t\le c_{11}t_0(z)$),
we have that
\begin{align*}
p_D(t,x,z)\ge c_{12}\Psi(t,x)\Psi(t,z)\left(\frac{t}{|x-z|^{d+\alpha}}\wedge t^{-d/\alpha}\right) e^{-c_{13}tf(z_1)^{-\alpha}}
\ge c_{14} \Psi(t,x)\Psi(t,z)\frac{te^{-c_{13}tf(z_1)^{-\alpha}}}{(1+|x|)^{d+\alpha}}
\end{align*}
and $$p_D(t,y,z) \ge c_{15} \Psi(t,y)\Psi(t,z)\frac{t}{(1+|y|)^{d+\alpha}} e^{-c_{16}tf(z_1)^{-\alpha}},
$$ where we used the fact that
$z_1\le |y|/2 \le 3|x|/4$.

Now, we let
$\widetilde D:=\{z:=(z_1,\tilde z)\in D: |\tilde z|\le
2c_{17}f(z_1)\}\subseteq D$ for some
constant $c_{17}>0$ (small enough) such that
$\delta_D(z)\ge c_{17}
f(z_1)$
 for all $z\in \widetilde D$. In particular, for any $z\in \widetilde D$
with $z_1\le c_3s_1(c_4t)$
and $c_1t_0(y)\le t \le c_2$,
$\Psi(t,z)\ge
c_{18}
 {f(z_1)^\alpha}/{t}\ge c_{19}e^{-tf(z_1)^{-\alpha}}$.
Then, combining both estimates above together yields that for all $x,y\in
D$,
$c_1t_0(y)\le t \le c_2$
 and $z\in \widetilde
D$ with $z_1\le c_3s_1(c_4 t)$,
\begin{align*}
 p_D(t,x,z) \ge c_{20} \frac{\Psi(t,x)t}{(1+|x|)^{d+\alpha}} e^{-c_{21}tf(z_1)^{-\alpha}}\quad\text{and}\quad p_D(t,y,z) \ge c_{20}\frac{ \Psi(t,y)t}{(1+|y|)^{d+\alpha}} e^{-c_{21}tf(z_1)^{-\alpha}}.
\end{align*}
Hence,  for all $c_1t_0(y)\le t \le c_2$,
\begin{align*}
p_D(2t,x,y)
&\ge \int_{\{z\in \widetilde D: z_1\le c_3s_1(c_4t)\}} p_D(t,x,z)p_D(t,z,y)\,dz\\
&\ge c_{20}^2\frac{\Psi(t,x)t}{(1+|x|)^{d+\alpha}}\frac{ \Psi(t,y)t}{(1+|y|)^{d+\alpha}} \int_{\{z\in \widetilde D: z_1\le c_3s_1(c_4t)\}\}}  e^{-2c_{21}tf(z_1)^{-\alpha}}\,dz\\
&\ge c_{22}
\phi(x)\phi(y)
 \int_{\{z\in \widetilde D: z_1\le c_3s_1(c_4t)\}}  e^{-2c_{21}tf(z_1)^{-\alpha}}\,dz,
 \end{align*} where  the last inequality follows from the definition of $\Psi(t,x)$
 and the fact that $t\ge C_0f(y_1)^\alpha\ge C_0f(x_1)^\alpha$.

Furthermore, note that for all $c_1t_0(y)\le t \le c_2$,
it holds that
$0<c_{23}:=C_{5.2}c_3s_0(c_2c_4)\le C_{5.2}c_3s_0(c_4t) \le c_3s_1(c_4t)$.
 Thus, by the fact $\{x\in D:x_1>2\}=D_f^2$, we have
\begin{align*}
 &\int_{\{z\in \widetilde D: z_1\le c_3s_1(c_4t)\}}  e^{-2c_{21}tf(z_1)^{-\alpha}}\,dz\\
&\ge c_{24}\bigg[\Big(\int_2^{c_3s_1(c_4t)}f(s)^{d-1}e^{-2c_{21}tf(s)^{-\alpha}}\,ds\Big)
\I_{\{c_3s_1(c_4t)>2\}}+\Big(\int_{\{z\in \widetilde D: z_1\le  c_{23}
 \}}  e^{-2c_{21}tf(z_1)^{-\alpha}}\,dz\Big)
\bigg]\\
&\ge c_{25}\int_0^{c_3s_1(c_4t)}f(s)^{d-1}e^{-2c_{21}tf(s)^{-\alpha}}\,ds,
\end{align*}
where the last inequality follows from
the property that for every $c_1t_0(y)\le t \le c_2$,
\begin{align*}
\int_{\{z\in \widetilde D: z_1\le c_{23}\}}  e^{-2c_{21}tf(z_1)^{-\alpha}}\,dz\ge c_{26}\ge c_{27}
\int_0^2 f(s)^{d-1}e^{-2c_{21}tf(s)^{-\alpha}}\,ds.
\end{align*}
By now we have obtained the desired assertion.
 \end{proof}

 Since $t_0(y)\simeq
f(y_1)^\alpha
\log(2+|y|)$ for all $y\in D$ and the function
$g(s)\simeq f(s)^\alpha \log(2+s)$ is non-increasing on $(0,\infty)$,  for any $y,z\in D$ with $|z|\ge |y|/8$, $t_0(y)\ge
c_0t_0(z)$ holds for some constant $c_0>0$
independent of $y$ and $z$. In
particular, according to \eqref{l3-1-0}, we know that for any $z\in
D$ such that $|z|\ge |y|/8$ and any
 $c_1t_0(y)\le t \le c_2$ (with any fixed $c_1$ and $c_2$),
\begin{equation}\label{e:kkll8}
\begin{split}
\Pp^z(\tau_D>t)&\le c_3\Psi(t,z)\min\Big\{ e^{-c_{2.9.2}f(z_1)^{-\alpha}t}+\frac{t}{(1+|z|)^{d+\alpha-1}}, e^{-c_{2.9.2}t}\Big\}\\
&\le c_3\Psi(t,z)\bigg( e^{-c_{2.9.2}f(z_1)^{-\alpha}t}+\frac{t}{(1+|z|)^{d+\alpha-1}}\bigg) \le c_4\Psi(t,z) \left(\frac{t}{(1+|z|)^{d+\alpha-1}}\right)^{q},
\end{split}
\end{equation} where $q= c_0c_1 \wedge 1 \le 1$
and in the last inequality we used the facts that $c_0c_1t_0(z)\le c_1 t_0(y)\le t \le c_2 (1+|z|)$ and
\begin{align*}e^{-{c_{2.9.2}f(z_1)^{-\alpha}t}}&\le
e^{-c_0 c_1 c_{2.9.2} f(z_1)^{-\alpha}t_0(z)}=\left(\frac{t_0(z)}{(1+|z|)^{d+\alpha-1}}\right)^{c_0 c_1}
 \le c_5
 \left(\frac{t}{(1+|z|)^{d+\alpha-1}}\right)^{c_0 c_1}.\end{align*}

To consider upper bounds of $p^D(t,x,y)$ we will frequently use \eqref{e:kkll8}.

\begin{lemma}\label{l4-1} {\bf(Upper bound when $C_1t_0(y)\le t\le C$ for some $C_1$
and for any $C>C_1$.)}\,\,
Suppose that $g$ is
non-increasing on $(0,\infty)$ such that $\lim\limits_{s\to\infty} g(s)=0$.
Then there exists $c_1>
C_0C_{3.10}$ such that for every $c_2>c_1$, we can find  positive constants
$c_{\ac{1}}$, $c_{\ac{2}}$, $c_{\ac{3}}$ and $c_{\ac{4}}$ $($depending on $c_1$ and $c_2)$ so that
for every $y\in D$  with $c_1t_0(y)\le t\le c_2$,
\begin{align*}
p_D(t,x,y)\le  &c_{\ac{1}}
\phi(x)\phi(y)
\int_{0}^{c_{\ac{2}}s_1(c_{\ac{3}}t)}
f(s)^{d-1}e^{-c_{\ac{4}}tf(s)^{-\alpha}}\,ds.
 \end{align*}
\end{lemma}
\begin{proof}
Recall that
for any $z, y\in D$ with $|z|\ge |y|/8$, $t_0(y)\ge
c_0t_0(z)$ holds for some constant $c_0>0$ independent of $z, y$.
As explained in the proof of Lemma \ref{l4-2}, $C_0f(y_1)^\alpha\le c_1t_0(y)$ for every
$y\in D$ and $c_1>C_0C_{3.10}$, and we can choose $M$ large enough such that $|x| \ge 2|y|/3$ for every
$y\in D$ with $|y|>M$.

Note that there exists $c_3>0$
such that $t_0(y)>c_3$ for $y\in D$ with $|y| \le M$.
Thus, for every $y\in D$ with $|y| \le M$ and
$c_1t_0(y) \le t \le c_2$, it holds that
\begin{align}
\label{e:fde}
1 \wedge \frac{1}{|x-y|^{d+\alpha}} \le \frac{c_4} {(1+|x|)^{d+\alpha}} \le \frac{c_4
(1+M)^{d+\alpha}}
{ (1+|x|)^{d+\alpha}(1+|y|)^{d+\alpha}}\end{align}
and
$(c_1c_3 \vee C_0f(y_1)^\alpha ) \le t\le c_2(1 \wedge c_3^{-1}t_0(y))$.
Thus, by applying   Proposition \ref{p3-1}(ii) and \eqref{e:fde}, we get that
\begin{equation}\label{l4-1-0}
\begin{split}
 p_D(t,x,y)
&\le
c_5{\delta_D(x)^{\alpha/2}f(x_1)^{\alpha/2}} {\delta_D(y)^{\alpha/2}f(y_1)^{\alpha/2}} \left(1 \wedge \frac{1}{|x-y|^{d+\alpha}} \right)e^{-c_6f(M)^{-\alpha}}\\
&\le  c_7
\phi(x)\phi(y)
\int_0^{c_8s_1(c_9t)}f(s)^{d-1}e^{-c_{10}tf(s)^{-\alpha}}ds.
\end{split}
\end{equation}
 Here in the second inequality we have used
 the facts that for $y\in D$ with $|y|\le M$ and $c_1c_3\le c_1t_0(y)\le t\le c_2$,  (by noting that $s_1(t)\ge 2$ for all $t>0$),
\begin{equation}\label{l4-1-1a}
\int_0^{c_8s_1(c_9t)}f(s)^{d-1}e^{-c_{10}tf(s)^{-\alpha}}ds\ge
\int_0^{2c_8}f(s)^{d-1}e^{-c_{10}c_1c_3f(s)^{-\alpha}}ds\ge c_{11}.
\end{equation}

Next, we suppose that $|y| > M$.
It follows from the assumption $f(s)\ge c(1+s)^{-p}$
that, for any $t\ge c_1t_0(y)\ge C_0f(y_1)^\alpha$ and $v,u\in D$,
\begin{equation}\label{l4-1-3}
p_D(t,v,u)\le c_{12}t^{-d/\alpha}\le c_{13}f(y_1)^{-d}\le c_{14}(1+|y|)^{dp}.
\end{equation}

Fix large $N$ such that $(N-1)q_0-dp\ge d+\alpha$ and $(N-1)q\ge1$, where
$q_0:=q(d+\alpha-1)$ and $q>0$ is the  constant in \eqref{e:kkll8}.
Suppose that $z,u\in D$ satisfies $|z|\ge |y|/2$ and $ c_1t_0(y)\le t \le c_2$.
Choose
$M$ larger if necessary
such that ${3|z|}/{4}\ge {3|y|}/{8}\ge {3M}/{8}\ge (3Nc_2)^{1/\alpha}\ge (3Nt)^{1/\alpha}$.
Then,
\begin{equation}\label{l4-1-4}
\begin{split}
p_D(2t,z,u)\le & \begin{cases}
\displaystyle \int_Dp_D(t,z,v)p_D(t,v,u)\,dv, &|u|\ge |z|/4\\
c_{15}\Psi(t,z){t}{|z-u|^{-(d+\alpha)}},&|u|<|z|/4\end{cases}\\
\le&\begin{cases} c_{14}(1+|y|)^{dp}\displaystyle \int_Dp_D(t,z,v)\,dv, &|u|\ge |z|/4\\
c_{15}\Psi(t,z){t}{|z-u|^{-(d+\alpha)}},&|u|<|z|/4\end{cases}\\
\le&c_{16}\begin{cases} (1+|y|)^{dp}\Psi(t,z)\left(\frac{t}{(1+|z|)^{d+\alpha-1}}
\right)^q, &|u|\ge |z|/4\\
\Psi(t,z)t(1+|z|)^{-(d+\alpha)},&|u|<|z|/4\end{cases}\\
\le& c_{17}\Psi(t,z) \left(\frac{
t^q}
{(1+|z|)^{q_0-dp}} \vee \frac{t}{(1+|z|)^{d+\alpha}}\right),
\end{split}
\end{equation}
where the first inequality follows from
Lemma \ref{l:new}
because $|z-u|\ge {3|z|}/{4}\ge (Nt)^{1/\alpha}$ for every $|u|<|z|/4$,
the second
inequality is due to \eqref{l4-1-3}, in the third inequality we
have used \eqref{e:kkll8},
and the last inequality is due to $|z|\ge |y|/2$ and $t\le c_2$.

Furthermore, we can obtain that for any
$z,u \in D$ with $|z|\ge |y|/2$ and any $c_1t_0(y)\le t \le c_2$,
\begin{align*} p_D(3t,z,u)
&\le \begin{cases} \displaystyle \int_Dp_D(2t,z,v)p_D(t,v,u)\,dv, &|u|\ge |z|/4\\
c_{18}\Psi(t,z){t}{|z-u|^{-(d+\alpha)}},&|u|<|z|/4\end{cases}\\
&\le \begin{cases}
 c_{19}\Psi(t,z) \left(\frac{
t^q}{(1+|z|)^{q_0-dp}} \vee \frac{t}{(1+|z|)^{d+\alpha}}\right)
 \displaystyle \int_Dp_D(t,v,u)\,dv, &|u|\ge |z|/4\\
c_{18}\Psi(t,z){t}{|z-u|^{-(d+\alpha)}},&|u|<|z|/4\end{cases}\\
&\le c_{20}\begin{cases}\Psi(t,z)\left(\frac{
t^q}{(1+|z|)^{q_0-dp}}\vee \frac{t}{(1+|z|)^{d+\alpha}}\right)
\left[1 \wedge \left(\frac{t}{(1+|z|)^{d+\alpha-1}}\right)^q \right], &|u|\ge |z|/4\\
\Psi(t,z)t(1+|z|)^{-(d+\alpha)},&|u|<|z|/4\end{cases}\\
&\le c_{21}\Psi(t,z)\left(\frac{
t^{2q}}
{(1+|z|)^{2q_0-dp}} \vee \frac{t}{(1+|z|)^{d+\alpha}}\right),
\end{align*}
where the first inequality is due to
Lemma \ref{l:new},
the second inequality follows from
\eqref{l4-1-4}, and  we have used \eqref{e:kkll8} again
and the fact that $|u|\ge |z|/4\ge |y|/8$
in the third
inequality.

Since $(N-1)q_0-dp\ge d+\alpha$,
$(N-1)q\ge1$ and
 $|z-u|\ge (Nt)^{1/\alpha}$ for every $|u|<|z|/4$ and $|z|\ge |y|/2$,
we can iterate
the argument above $N$ times
to obtain
 that for all
$z,u\in D$ with $|z|\ge |y|/2$ and all $c_1t_0(y)\le t \le c_2$,
$$p_D(Nt,z,u)\le c_{22}t\Psi(t,z)(1+|z|)^{-d-\alpha}.
$$
Combining this with \eqref{l3-1-1},
we further obtain that for any $u,z\in D$ with
$|z|\ge |y|/2$
and $c_1t_0(y)\le t \le c_2$,
\begin{equation}\label{l4-1-5}
\begin{split}
 p_D((N+1)t,u,z)
&=\int_Dp_D(t,u,v)p_D(Nt,v,z)\,dv\\
&\le c_{22} t\Psi(t,z)(1+|z|)^{-d-\alpha} \int_Dp_D(t,u,v)\,dv\\
&\le c_{23} t\Psi(t,z)(1+|z|)^{-d-\alpha}\left(e^{-c_{24}t f(u_1)^{-\alpha}}+t(1+|u|)^{-(d+\alpha-1)}\right).
\end{split}
\end{equation}
Since $|x| \ge 2|y|/3$,
by \eqref{l4-1-5} we arrive at that
for
every $c_1t_0(y)\le t \le c_2$,
\begin{equation}\label{l4-1-6}
\begin{split}
p_D(2(N+1)t,x,y)&=\int_Dp_D((N+1)t,x,u)p_D((N+1)t,y,u)\,du\\
&\le c_{25}t^2\Psi(t,x)\Psi(t,y)(1+|x|)^{-d-\alpha}
(1+|y|)^{-d-\alpha}L(t),
\end{split}
\end{equation}
where
$$L(t):=\int_D K(t,z)^2\, dz \quad \text{ and } \quad K(t,z):=e^{-c_{24}t
f(z_1)^{-\alpha}}+t(1+|z|)^{-(d+\alpha-1)}.$$

Moreover, thanks to the non-increasing property of $g(s)\simeq f(s)^\alpha \log(2+s)$, it is not difficult to verify that for all $c_1t_0(y)\le t \le c_2$
  \begin{align*}
K(t,z)\le c_{26}
\begin{cases}
1&\text{if}\ 0<z_1\le c_{27}s_0(c_{28}t);\\
 e^{-c_{24}tf(z_1)^{-\alpha}}&\text{if}\ c_{27}s_0(c_{28}t)<z_1\le c_{29}s_1(c_{30}t);\\
t(1+|z|)^{-(d+\alpha-1)}&\text{if}\ z_1>c_{29}s_1(c_{30}t).
\end{cases}
\end{align*}
Write
\begin{align*}
L(t)&= \int_{\{z\in D:
z_1\le  c_{27}s_0(c_{28}t)\}}K(t,z)^2 \,dz +
\int_{\{z\in D: c_{27}s_0(c_{28}t)< z_1\le c_{29} s_1(c_{30}t)\}}K(t,z)^2\,dz\\
&\quad +\int_{\{z\in D: z_1>c_{29}s_1(c_{30}t)
\}}
K(t,z)^2 \,dz =:L_1+L_2+L_3.\end{align*}
 Therefore, using the facts that
 $s_0(t)\ge 2$ and $\{x\in D: x_1>2\}=D_f^2$,
 \begin{align*}L_1\le &
 c_{31}\left(\int_{c_{27}}^{c_{27}s_0(c_{28}t)}  f(s)^{d-1}\,ds+1\right)\le
 c_{32}\left(\int_0^{c_{27}s_0(c_{28}t)}f(s)^{d-1}e^{-tf(s)^{-\alpha}}\,ds+1\right),\\
L_2\le & c_{33}\int_{c_{27}s_0(c_{28}t)}^{c_{29}s_1(c_{30}t)}f(s)^{d-1}e^{-2c_{24}tf(s)^{-\alpha}}\,ds,\\
L_3\le &
 c_{34}\int_{\{z\in D: z_1>c_{29}s_1(c_{30}t)\}}(1+|z|)^{-2(d+\alpha-1)}\,dz\le c_{35},
\end{align*}
 where the inequality for $L_1$ follows from the argument of \eqref{l4-1-1a} and the fact that
 $e^{-tf(s)^{-\alpha}}\ge c_{36}$ for every $0\le s\le c_{27}s_0(c_{28}t)$.
 Hence,
 according to the proof of \eqref{l4-1-1a} again,
 $$L(t)\le c_{36}\int_0^{c_{29}s_1(c_{30}t)}f(s)^{d-1}e^{-c_{37}tf(s)^{-\alpha}}\,ds.$$
  This, along with \eqref{l4-1-6}
 (by replacing $2(N+1)t$ with $t$), the definition of $\Psi(t,x)$ and
 \eqref{e:f3}, yields that  for all $2c_{1}(N+1)t_0(y)\le t \le c_2$,
\begin{align*}
 p_D(t,x,y)
&\le   c_{38} \Psi(2^{-1}(N+1)^{-1}t,x)\Psi(2^{-1}(N+1)^{-1}t,y)\frac{t}{(1+|x|)^{d+\alpha}}\frac{t}{(1+|y|)^{d+\alpha}}\\
 &\quad \times \int_{0}^{c_{29} s_1(c_{30}2^{-1}(N+1)^{-1}t)}  f(s)^{d-1} e^{-c_{37}2^{-1}(N+1)^{-1}t f(s)^{-\alpha}}\,ds \\
 & \le  c_{39}
 \phi(x) \phi(y)
 \int_{0}^{c_{32}s_1(c_{40}t)} f(s)^{d-1}e^{-c_{41}tf(s)^{-\alpha}}\,ds,\end{align*}
proving the desired assertion.
 \end{proof}

By Lemmas \ref{l4-2} and \ref{l4-1}, we further have the following statement.

\begin{proposition}\label{l4-3o}
Suppose that $g$ is
non-increasing on $(0,\infty)$ such that $\lim\limits_{s\to\infty} g(s)=0$.
Then there are constants $c_{\ac{i}}>0$ $(i=1,2,\cdots, 8)$
such that for all $x,y\in D$ and $t\ge c_{\ac{1}} t_0(y)$,
\begin{align*}&c_{\ac{2}}
\phi(x)\phi(y)
\max\Big\{\int_{0}^{c_{\ac{3}}s_1(c_{\ac{4}}t)}
f(s)^{d-1}e^{-c_{\ac{5}}tf(s)^{-\alpha}}\,ds, e^{-c_{\ac{6}} t}\Big\}\\
&\le p_D(t,x,y)\le
c_{\ac{7}}
\phi(x)\phi(y)
\max\Big\{\int_{0}^{c_{\ac{8}}s_1(c_{\ac{9}}t)}
 f(s)^{d-1}e^{-c_{\ac{10}}tf(s)^{-\alpha}}\,ds, e^{-c_{\ac{11}} t}\Big\}.\end{align*}
\end{proposition}

\begin{proof}  Since the function $g(s)\simeq f(s)^\alpha\log (1+s)$ is non-increasing with $\lim_{s\to\infty}g(s)=0$,
$s_1(t)=2$
for $t>0$ large enough.
Thus, by Lemmas \ref{l4-2} and \ref{l4-1}, we only need to verify the required assertion for all $t\ge c_0$ with any given $c_0>0$.

According to \cite[Theorem 5]{Kw}, the associated Dirichlet
 semigroup $(P_t^D)_{\ge0}$ is intrinsically ultracontractive when $\lim_{y\in D,|y|\rightarrow
 \infty}t_0(y)=0$.  Hence, it
 follows from \cite[Theorem 4.2.5]{Da} that for all
$t\ge c_0$ and $x,y\in D$,
 $p_D(t,x,y) \simeq  e^{-\lambda_D
t}\phi_1(x)\phi_1(y)$
where $\phi_1(x)$ is the ground state (i.e., the first strictly positive eigenfunction corresponding to the smallest eigenvalue $\lambda_D$ of the Dirichlet fractional Laplacian $(-\Delta) ^\alpha|_D$) of the semigroup
$(P_t^D)_{\ge0}$. On the other hand, by \cite[Theorem 6.1]{CKW}
and its proof, for all $x\in D$,
$\phi_1(x)\simeq  \phi(x)=
{\delta_D(x)^{\alpha/2}f(x_1)^{\alpha/2} }{(1+|x|)^{-d-\alpha}}.$
Putting both estimates together, we can obtain the desired
assertion.
\end{proof}

\subsection{\bf{Case III-2: $g$ is
non-decreasing on $(0,\infty)$.}}
In this part, we are concerned
with
the case that
 $g$ is
non-decreasing on $(0,\infty)$. In particular, $\lim_{s\to \infty} g(s)>0$.
Because of $t_0(y)\simeq
f(y_1)^\alpha\log(2+|y|)$, we have
$\liminf\limits_{y\in D,|y|\rightarrow \infty}t_0(y)>0$.

\begin{lemma}\label{l4-4}{\bf (Lower bound)}\,\,
Suppose that $g$ is
non-decreasing on $(0,\infty)$. Then
there exist constants $c_{{\ac{1}}}>0$ large enough and $c_{\ac{2}}, c_{\ac{3}}>0$
such that for every $y\in D$ with $t\ge
c_{\ac{1}}t_0(y)\ge 1$,
\begin{equation}\label{l4-4-1}
\begin{split}
p_D(t,x,y)\ge  &c_{\ac{2}}e^{-c_{\ac{3}}t}
\phi(x)\phi(y).
\end{split}
\end{equation}
\end{lemma}
\begin{proof}
We choose $c_1>C_0C_{3.10}$ and $M>20$  large enough so that
$c_1t_0(y)\ge 2\vee C_0f(y_1)^\alpha$ and, that if $|y|  >M$ then
$|x| \ge 2|y|/3$.
Note that, for $|y|\le M$ and $t\ge 2$,
$$
t^{-d/\alpha} \wedge \frac{t}{|x-y|^{d+\alpha}} \ge c_2
\frac{t^{-d/\alpha}}{(|x|+1)^{d+\alpha}}
\ge c_2
\frac{t^{-d/\alpha}(M+1)^{d+\alpha}}{(|x|+1)^{d+\alpha}(|y|+1)^{d+\alpha}}.
$$
By Proposition \ref{p3-1}(i),
for every $y\in D$ with $|y|\le M$ and $t\ge c_1t_0(y)$,
\begin{align*}
p_D(t,x,y)&\ge c_3
\frac{t^{-d/\alpha}\delta_D(x)^{\alpha/2}f(x_1)^{\alpha/2}\delta_D(y)^{\alpha/2}f(y_1)^{\alpha/2}}{(|x|+1)^{d+\alpha}(|y|+1)^{d+\alpha}}e^{-c_4tf(M)^{-\alpha}}\ge c_5 \phi(x)\phi(y)e^{-2c_4tf(M)^{-\alpha}}.
\end{align*}
Thus, \eqref{l4-4-1} holds if $|y|\le M$.

Next, we assume that  $|y| > M$. Fix a ball $B(x_0,4\lambda_1)\subset D$ with $x_0\in D$
such that $|x_0|\le 6$ and $\lambda_1>0$.
As shown in the beginning of the proof for Lemma \ref{l3-3}, there
are $\xi_x,\xi_y\in D$ and $\lambda_2>0$ such that
$B(\xi_x,4\lambda_2f(x_1))\subset D$,
$B(\xi_y,4\lambda_2f(y_1))\subset D$, and \eqref{l3-3-2} holds
true with $V_x:=B(\xi_x,\lambda_2f(x_1))$ and
$V_y:=B(\xi_x,\lambda_2f(y_1))$.

On the other hand,
we find that for all $z\in V_x$, $w\in
B(x_0,\lambda_1)$ and $t\ge c_1t_0(y)\ge 1$,
\begin{align*}
& p_D(t,z,w)
=\Ee^z\left[p_D\left({t}/{2},X_{t/2}^D,w\right)\right]\\
&\ge \Ee^z\Big[p_D(t/2, X_{t/2},w) :0<\tau_{B(z,\lambda_2
f(x_1))}<2^{-2}C_0
f(x_1)^\alpha, X_{\tau_{B(z, \lambda_2 f(x_1))}}\in B(w, \lambda_1),\\
&\qquad\quad X_s\in B(w,2\lambda_1)\text{ for all }s\in [\tau_{B(z,\lambda_2 f(x_1))}, \tau_{B(z,\lambda_2 f(x_1))}+t]\Big]\\
&\ge c_{6} e^{-c_{7}t} \left(\int_0^{2^{-2}C_0
f(x_1)^\alpha}\int_{B(z,
\lambda_2 f(x_1))}
p_{B(z,\lambda_2 f(x_1))}(s,z,u)\int_{B(w,\lambda_1)}\frac{1}{|u-v|^{d+\alpha}}\,dv\,du\,ds\right)\\
&\qquad\quad\times\inf_{u\in B(w,\lambda_1)} \Pp^u \left(\tau_{B(u,\lambda_1)}>t\right)\\
&\ge c_{8} e^{-c_{9}t} f(x_1)^\alpha\Pp^z(\tau_{B(z,\lambda_2
f(x_1))}>{C_0f(x_1)^\alpha}/{4})
\frac{1}{(1+|x|)^{d+\alpha}}\ge c_{10} e^{-c_{9}t}\frac{f(x_1)^\alpha}{(1+|x|)^{d+\alpha}},
\end{align*}
where the second inequality follows from
L\'evy system \eqref{e2-6} and the fact that
\begin{align*}
\inf_{w,v\in B(x_0,3\lambda_1)}p_D\big({t}/{2},w,v\big) &\ge
\inf_{w,v\in
B(x_0,3\lambda_1)}p_{B(x_0,4\lambda_1)}\big({t}/{2},w,v\big)\ge
c_6e^{-c_7t},\quad t\ge 1
\end{align*}
thanks to \eqref{hk-s1},
the third inequality is due to
\eqref{l3-3-5} (also by \eqref{hk-s1}) and the fact that $|u-v|\le c_{11}(1+|x|)$ for all $u\in
B(z,\lambda_2f(x_1))$ with $z\in V_x$ and $v\in B(w,\lambda_1)$, and in the last
inequality we  have used \eqref{l5-1-2}.

Combining the estimate above with \eqref{l3-3-2} yields that for all
$w\in B(x_0,\lambda_1)$ and $t\ge 2c_1t_0(y)\ge 1$,
\begin{align*}
p_D(t,x,w)&\ge
\int_{V_x}p_D(2^{-2}C_0
f(x_1)^\alpha,x,z)
p_D(t-2^{-2}C_0
f(x_1)^\alpha,z,w)\,dz\\
&\ge c_{12}\delta_D(x)^{\alpha/2}f(x_1)^{-\alpha/2}e^{-c_{13}t}
\frac{f(x_1)^\alpha}{(1+|x|)^{d+\alpha}}
= c_{12}\phi(x)e^{-c_{13}t}.
\end{align*}
Similarly, for every $w\in B(x_0,\lambda_1)$ and $t\ge 2c_1t_0(y)\ge 1$,
$
p_D(t,y,w)\ge
c_{14} \phi(y)e^{-c_{15}t}.$
Hence, for all $t\ge 2c_1t_0(y)\ge 1$,
\begin{align*}
p_D(t,x,y)&\ge \int_{B(x_0,\lambda_1)}
p_D(t/2,x,w)p_D(t/2,w,y)\,dw
\ge c_{16}e^{-c_{17}t} \phi(x) \phi(y).
\end{align*}
Now we have proved the desired assertion.
\end{proof}

\begin{lemma}\label{l4-3}{\bf (Upper bound)}\,\,
Suppose that
$g$ is
non-decreasing on $(0,\infty)$.
Then,
there exist constants $c_{\ac{1}}, c_{\ac{2}}, c_{\ac{3}}>0$
such that for all $t>0$ and $y\in D$ with $t\ge
c_{\ac{1}}
t_0(y) (\ge 1)$,
\begin{equation}\label{l4-3-1}
p_D(t,x,y)\le  c_{\ac{2}}e^{-c_{\ac{3}}t}
 \phi(x) \phi(y).
\end{equation}
\end{lemma}
\begin{proof} Since $t_0(z)\simeq
f(z_1)^\alpha \log(2+|z|)$ and the function $s\mapsto g(s)\simeq f(s)^\alpha \log(2+s)$ is non-decreasing, we have $t_0(z)\ge
c_0>0$ for all
$z\in D$.
Choose $M_0>20$
and $c_1>C_0C_{3.10}$ large enough so that $|x| \ge 2|y|/3$ for every $|y|>M_0$,
$f(z_1)\le f(y_1)$ for every $|z|\ge 2|y|>2M_0$,
and
$c_1t_0(y)\ge 2\vee C_0f(y_1)^\alpha$.

Since $t_0(y)\ge c_0>0$,
 for  every $y\in D$ with $|y|\le M_0$ and $t\ge c_1t_0(y)$,
$$
t^{-d/\alpha} \wedge \frac{t}{|x-y|^{d+\alpha}}
\simeq  \frac{t}{(t^{1/\alpha}+|x-y|)^{d+\alpha}}
\le \frac{c_2t}{(1+|x|)^{d+\alpha}},
$$ so, by applying   Proposition \ref{p3-1}(ii), we get that  for  every $y\in D$ with $|y|\le M_0$ and $t\ge c_1t_0(y)$,
\begin{align*}
 p_D(t,x,y)
&\le
c_3{\delta_D(x)^{\alpha/2}f(x_1)^{\alpha/2}} {\delta_D(y)^{\alpha/2}f(y_1)^{\alpha/2}} \frac{te^{-c_4 tf(M_0)^{-\alpha}}}{(1+|x|)^{d+\alpha}}  \le  c_5
\phi(x)\phi(y) \
e^{-c_4tf(M_0)^{-\alpha}/2}.
\end{align*}
Thus, we only need to consider the case that $|y|>M_0$ and $t\ge c_1t_0(y)$. For this, we will split the proof into two parts.

(1) For any $t>0$, define
\begin{equation}\label{l4-3-3}
\begin{split}
s_2(t):=&
\sup\{s>0: C_{5.14.1}\log(2+s)\le t\}\vee 2M_0,\\
 s_3(t):
=&\sup\{s>0: C_{5.14.2}\,
g(s)\le t\}\vee s_2(t),
\end{split}
\end{equation}
where
we use the convention that
$\sup \emptyset =0$. It is easy to see that $s_3(t)\ge s_2(t)>0$, and
the constants $C_{5.14.1}$, $C_{5.14.2}$ (both of which are large) are to be determined later.

Again by the
assumption that
the function $s\mapsto g(s)\simeq f(s)^\alpha \log(2+s)$ is non-decreasing on
$(0,\infty)$, $t_0(z)\simeq f(z_1)^\alpha \log(2+|z|)$ and the definition of $s_3(t)$, (by choosing
$c_1$, $C_{5.14.1}$, $C_{5.14.2}$ large enough if necessary,) we can find
a positive constant $c_6$
such that for every $t\ge c_1t_0(y)$,
\begin{align}\label{e:qq}
t\le  c_6 t_0(z) \text{ when } |z|\ge s_3(t)/2, \quad \text{and}
\quad
t\ge33
t_0(z)  \text{ when } |z|\le 8 s_3(t).
\end{align}
This along with Lemma \ref{e:lm2} yields that for all $z\in D$ and $t\ge c_1t_0(y)$,
\begin{equation}\label{l4-3-4}\begin{split}
 \Pp^z(\tau_D>t)
&\le
 c_7 \Psi(t,z)\min\bigg\{e^{- c_{2.9.2}  f(z_1)^{-\alpha}t}+\frac{t}{(1+|z|)^{d+\alpha-1}},
e^{- c_{2.9.2} t}\bigg\}
\\
&\le c_7\Psi(t,z)\times
\begin{cases}
e^{- c_{2.9.2} f(z_1)^{-\alpha}t}+
\frac{c_6t_0(z)}{(1+|z|)^{d+\alpha-1}}, & \ \ |z|\ge s_3(t)\\
e^{- c_{2.9.2}
f(z_1)^{-\alpha}t_0(z)}
+\frac{t}{(1+|z|)^{d+\alpha-1}},  &\ \ s_2(t)/2\le|z|\le4 s_3(t)\\
e^{- c_{2.9.2} t},  & \ \ |z|\le 2s_2(t)
\end{cases}\\
&\le c_8\Psi(t,z)\times
\begin{cases}
e^{- c_{2.9.2}f(z_1)^{-\alpha}t} +c_6e^{- c_{2.9.2}t_0(z)f(z_1)^{-\alpha}}
& \ \ |z|\ge s_3(t)\\
\frac{t_0(z)}
{(1+|z|)^{d+\alpha-1}}
+\frac{t}{(1+|z|)^{d+\alpha-1}},  &  \ \ s_2(t)/2\le |z|\le 8 s_3(t) \\
e^{- c_{2.9.2} t},  & \ \ |z|\le 2s_2(t)
\end{cases}\\
&\le c_9\Psi(t,z)\times
\begin{cases}
e^{-c_{10}f(z_1)^{-\alpha}t}, & \ \ |z|\ge s_3(t) \\
\frac{t}
{(1+|z|)^{d+\alpha-1}}
&  \ \ s_2(t)/2\le |z|\le 8 s_3(t) \\
e^{- c_{2.9.2}t},  & \ \ |z|\le 2s_2(t),
\end{cases}
\end{split}\end{equation}
where in the last inequality we have used
\eqref{e:qq}.
Thus, for all $t\ge c_1t_0(y)$,
$u,z\in D$ with $|z|\ge s_3(t)$ and $N\ge 2$,
\begin{equation}\label{l4-3-4a}
\begin{split}
p_D(Nt,z,u)&=\int_D p_D(t,z,v)p_D((N-1)t,v,u)\,dv\le c_{11}((N-1)t)^{-d/\alpha}\int_D p_D(t,z,v)\,dv\\
&\le c_{12}\Pp^z(\tau_D>t)\le c_{13}\Psi(t,z)e^{-c_{10}f(z_1)^{-\alpha}t}.
\end{split}
\end{equation}

Below, we will further refine the estimate above.
For every $|u|\ge
2|z|$, $|z|\ge s_3(t)\ge2$ and $N\ge 2$, we have
\begin{align*}
p_D((N+1)t,z,u)=\left(\int_{\{v\in D: |v-z|\le |u|/4\}}+\int_{\{v\in D: |v-z|> |u|/4\}} \right)p_D(t,z,v)p_D(Nt,v,u)\,dv=:I_1+I_2.
\end{align*}
On the one hand, for $|z|\ge s_3(t)$ and $|u|\ge 2|z|$,
\begin{align*}
I_1&\le c_{14}\int_{\{v\in D: |v-z|\le |u|/4\}}p_D(t,z,v)\frac{Nt}{|v-u|^{d+\alpha}}\,dv\\
&\le \frac{c_{15}t}{(1+|u|)^{d+\alpha}}\int_D p_D(t,
z,v)\,dv\le
c_{16}\Psi(t,z)\frac{t}{(1+|u|)^{d+\alpha}}e^{-c_{10}f(z_1)^{-\alpha}t},
\end{align*}
where the second inequality follows from the fact that
$|v-u|\ge|u|-|v-z|-|z|\ge |u|/4$ for all $v\in D$ with $|v-z|\le
|u|/4$, and in the last inequality we have used
$\eqref{l4-3-4}$. On the other hand, for
$|z|\ge s_3(t)$ and  $|u| \ge 2|z|$,
\begin{align*}
I_2&\le c_{17}\Psi(t,z)\int_{\{v\in D: |v-z|> |u|/4\}}
\frac{t}{|z-v|^{d+\alpha}}p_D(Nt,v,u)\,dv\le c_{18}\Psi(t,z)\frac{t}{(1+|u|)^{d+\alpha}}e^{-c_{10}f(u_1)^{-\alpha}t},
\end{align*}
where the first inequality follows from
Lemma \ref{l:new}
(since $|z-v|\ge |u|/4\ge |z|/2\ge s_3(t)/2
\ge s_2(t)/2\ge t^{1/\alpha}$ for all $t\ge c_1t_0(y)$ by taking $c_1$ large enough if necessary),
and
 in the second inequality we used
\eqref{l4-3-4}, $\Pp^z(\tau_D>Nt)\le \Pp^z(\tau_D>t)$ and the fact $|u|\ge 2|z|\ge 2s_3(t)$.
Combining with both estimates above, we arrive at that for all
$z,u\in D$ with $|z|\ge s_3(t)$ and $|u|\ge 2|z|$,
\begin{equation}\label{l4-3-5a}
p_D((N+1)t,z,u)\le c_{19}\Psi(t,z)e^{-c_{10}f(z_1)^{-\alpha}t}
\frac{t}{(1+|u|)^{d+\alpha}},
\end{equation} where we used the fact that $f(u_1)\le f(z_1)$ for $|u|\ge 2|z|>2M_0$ due to
the choice of $M_0$.

Meanwhile, for all $z,u\in
D$ with $|z|\ge s_3(t)$ and $|u|\le |z|/2$,  we have
\begin{align*}
p_D((N+1)t,z,u)=\left(\int_{\{v\in D: |v-z|\le |z|/4\}}+\int_{\{v\in D: |v-z|> |z|/4\}}\right)p_D(Nt,z,v)p_D(t,v,u)\,dv=:J_1+J_2.
\end{align*}
Then, replacing \eqref{l4-3-4} by $\Pp^z(\tau_D>t)\le c_{20}\Psi(t,z)e^{-c_{2.9.2}t}$
(due to \eqref{l3-1-0}) and following the arguments above for $I_1$ and $I_2$, we can obtain immediately that for all $z,u\in
D$ with $|z|\ge s_3(t)$ and $|u|\le |z|/2$, and for all $N\ge 2$,
\begin{equation*}
p_D((N+1)t,z,u)\le c_{21}\Psi(t,z)e^{-c_{2.9.2}t}\frac{t}{(1+|z|)^{d+\alpha}}.
\end{equation*}

Therefore, putting all the cases together, we finally get that for
any $N\ge 2$ and $z,u\in D$ with $|z|\ge s_3(t)$,
\begin{equation}\label{l4-3-7}
p_D((N+1)t,z,u)\le c_{22}\Psi(t,z)L_1(z,u,t),
\end{equation}
where
\begin{align*}
L_1(z,u,t)=
\frac{te^{-c_{2.9.2}t}}{(1+|z|)^{d+\alpha}}\I_{\{|u|\le |z|/2\}}+e^{-c_{10}f(z_1)^{-\alpha}t}\I_{\{|z|/2\le |u|\le 2|z|\}}+\frac{te^{-c_{10}f(z_1)^{-\alpha}t}}{(1+|u|)^{d+\alpha}}\I_{\{|u|>
2|z|\}}.
\end{align*}
Note that, for the case $|z|/2\le |u|\le 2|z|$ above, we used \eqref{l4-3-4a} directly.

Similarly, replacing \eqref{l4-3-4} by $\Pp^z(\tau_D>t)\le c_{23}\Psi(t,z)e^{-c_{2.9.2}t}$ and following the argument for  \eqref{l4-3-7},
we  can obtain that for
every
$N\ge2$ and $u,z\in D$ with $|z|\le
4s_2(t)$,
\begin{equation}\label{l4-3-8}
p_D((N+1)t,z,u) \le c_{24}\Psi(t,z)L_3(z,u,t),
\end{equation}
where
\begin{align*}
L_3(z,u,t)=&
\frac{te^{-c_{2.9.2}t}}{(1+|z|)^{d+\alpha}}\I_{\{|u|\le |z|/2\}}+e^{-c_{2.9.2}t}\I_{\{|z|/2\le |u|\le 2|z|\}}+e^{-c_{2.9.2}t}\frac{t}{(1+|u|)^{d+\alpha}}\I_{\{|u|> 2|z|\}}.
\end{align*}
In particular,
by choosing $C_{5.14.1}$ large enough
so that $e^{-{c_{2.9.2}t}/{2}}\le c_{24}{t}{(1+|z|)^{-(d+\alpha)}}$ for
every $|z|\le 4s_2(t)$,
it holds that
\begin{equation}\label{e:notess}L_3(z,u,t)\le c_{25} e^{-c_{2.9.2}t/2}
\frac{t}{(1+|z|)^{d+\alpha}}.
\end{equation}

Next, let  $u,z\in D$ with $8s_3(t)\ge |z|\ge   s_2(t)/2$.
Then, by \eqref{l4-3-4},
\begin{equation*}
\Pp^z(\tau_D>t)\le c_9\Psi(t,z)\frac{t}{(1+|z|)^{d+\alpha-1}},
\quad s_2(t)/2\le |z|\le 8s_3(t).
\end{equation*}
Hence, for every $t\ge c_1t_0(y)$ and $u,z\in D$ with $8s_3(t)\ge |z|\ge   s_2(t)/2$,
\begin{equation}\label{l4-3-5}
\begin{split}
p_D(2t,z,u)\le & \begin{cases}
\displaystyle \int_Dp_D(t,z,v)p_D(t,v,u)\,dv \le c_{27}\Pp^z(\tau_D>t) &|u|\ge |z|/4\\
c_{26}\Psi(t,z){t}{|z-u|^{-(d+\alpha)}} \le c_{26}\Psi(t,z){t}{|z-u|^{-(d+\alpha)}},&|u|<|z|/4\end{cases}\\
\le&\begin{cases} c_{28}\Psi(t,z)
{t}{(1+|z|)^{-(d+\alpha-1)}}
, &|u|\ge |z|/4\\
c_{27.5}\Psi(t,z)t(1+|z|)^{-(d+\alpha)},&|u|<|z|/4\end{cases}\\
\le& c_{28}\Psi(t,z)\frac{t}{(1+|z|)^{d+\alpha-1}},
\end{split}
\end{equation}
where the first inequality is due to Lemma \ref{l:new} (since $|z-u|\ge {3|z|}/{4}\ge {3s_2(t)}/{8}
\ge t^{1/\alpha}$ by choosing $c_1$ large enough if necessary), in the second inequality we have used that
$p_D(t,u,v)\le p(t,u,v)\le c_{29}t^{-d/\alpha}\le c_{29}$ for every
$t\ge c_1t_0(y)$($ \ge 1$).

Now,
applying \eqref{l4-3-5}
and following
the same iteration arguments for  \eqref{l4-1-5},
we can find an integer $M\ge 3$ such that for all $u,z\in D$ with $8s_3(t)\ge |z|\ge   s_2(t)/2$,
\begin{equation}\label{l4-3-6}
p_D((M-1)
t,z,u)\le c_{30}\Psi(t,z)\frac{t}{(1+|z|)^{d+\alpha}}.
\end{equation}
Then, according
to  \eqref{l4-3-6} and \eqref{l4-3-4},
for every $t\ge c_1t_0(y)$ and every $u,z\in D$ with
$s_2(t)/2\le |z|\le 4s_3(t)$ and
$s_2(t)/2\le |u|\le 8s_3(t)$,
\begin{equation}\label{l4-3-6a}
\begin{split}
p_D(Mt
,z,u)&=\int_{D}p_D((M-1)t,z,v)p_D(t
,v,u)\,dv\\
&\le c_{30}\frac{\Psi(t,z)t}{(1+|z|)^{d+\alpha}}\int_D p_D(t,v,u)\,dv\le c_{31}\frac{\Psi(t,z)t}{(1+|z|)^{d+\alpha}}\frac{t}{(1+|u|)^{d+\alpha-1}}.
\end{split}
\end{equation}

Meanwhile, following the same arguments for \eqref{l4-3-5a}, (in particular,
applying
$$\Pp^z(\tau_D>t)\le c_{9}\frac{\Psi(t,z)t}{(1+|z|)^{d+\alpha-1}}$$ in the estimate of $I_1$ for every
$s_2(t)/2\le |z|\le 4s_3(t)$, and $$\Pp^u(\tau_D>t)\le c_8\left(e^{-c_{2.9.2}f(u_1)^{-\alpha}t}+\frac{t_0(u)}{(1+|u|)^{d+\alpha-1}}\right)$$
in the estimate of $I_2$ for every $|u|\ge 8s_3(t)$,
which are due to the last and the third inequalities in \eqref{l4-3-4} respectively,)
we can obtain that for every $u,z\in D$ with
$s_2(t)/2\le |z|\le 4s_3(t)$ and $|u|\ge 8s_3(t)\ge 2|z|$,
\begin{align*}
p_D(Mt,z,u)&\le \frac{c_{32}\Psi(t,z)t}{(1+|u|)^{d+\alpha}}\left(\frac{t}{(1+|z|)^{d+\alpha-1}}+
\frac{t_0(u)}{(1+|u|)^{d+\alpha-1}}+e^{-c_{2.9.2}f(u_1)^{-\alpha}t}\right)\\
&\le \frac{c_{33}\Psi(t,z)t}{(1+|u|)^{d+\alpha}}\left(\frac{t}{(1+|z|)^{d+\alpha-1}}+
\frac{\log(2+|u|)}{(1+|u|)^{d+\alpha-1}}+e^{-c_{2.9.2}f(z_1)^{-\alpha}t_0(z)}\right)\\
&= \frac{c_{33}\Psi(t,z)t}{(1+|u|)^{d+\alpha}}\left(\frac{t+t_0(z)}{(1+|z|)^{d+\alpha-1}}+
\frac{\log(2+|u|)}{(1+|u|)^{d+\alpha-1}}\right)\le \frac{c_{34}\Psi(t,z)t}{(1+|z|)^{d+\alpha}}
\frac{t+\log(2+|u|)}{(1+|u|)^{d+\alpha-1}}.
\end{align*}
Here in the second inequality we have used the facts that $t_0(u)\le c_{35}\log(2+|u|)$,
$f(u_1)\le f(z_1)$ (which is due to $|u|\ge 2|z|\ge 2s_2(t)\ge 2M_0$) and $t\ge t_0(z)$ (which is due to
\eqref{e:qq}), and the last inequality follows from $|u|\ge 2|z|$.

Following the same arguments above for \eqref{l4-3-6a}, and using \eqref{l4-3-6} as well as \eqref{l4-3-4}, we can obtain that
for every $u,z\in D$ with
$s_2(t)/2\le |z|\le 4s_3(t)$
and $|u|\le s_2(t)$,
\begin{align*}
p_D(Mt,z,u)&\le \frac{c_{36}\Psi(t,z)t}{(1+|z|)^{d+\alpha}}e^{-c_{2.9.2}t}.
\end{align*}

Combining all above estimates together, we know that there exists $M\ge 3$ such that
for all $u,z\in D$ with $ s_2(t)/2\le |z|\le 4s_3(t)$  and $t\ge c_1t_0(y)$,
\begin{equation}\label{l4-3-11}
p_D(Mt,z,u)\le c_{37}\Psi(t,z)L_2(z,u,t),
\end{equation}
where
$$
L_2(z,u,t) = e^{-c_{2.9.2}t}\frac{t}{(1+|z|)^{d+\alpha}}\I_{\{|u|\le
s_2(t)\}}+ \frac{t}{(1+|z|)^{d+\alpha}}
\frac{t+\log(2+|u|)}{(1+|u|)^{d+\alpha-1}}\I_{\{|u|\ge s_2(t)\}}.
$$

(2)
According to
\eqref{e:qq},
(by taking  $c_1$
large enough
if necessary), we have $|y|\le s_3(t)$ when $t\ge c_1t_0(y)$.
Now, we will prove the desired upper bounds of $p_D(t,x,y)$ for
$|y|>M_0$ and $t\ge c_1t_0(y)$.
We first note that, since $t\ge c_1t_0(y)\ge 2\vee C_0f(y_1)^\alpha\ge 2\vee C_0f(x_1)^\alpha$, we have
$$
\frac{\Psi(t,x)}{(1+|x|)^{d+\alpha}} \simeq \phi(x) \quad \text{and}\quad
\frac{\Psi(t,y)}{(1+|y|)^{d+\alpha}} \simeq \phi(y).
$$
We consider the following five cases separately.

(i) Case 1:
$s_2(t)\le |y|\le s_3(t)$ and $s_2(t)/2\le |x|\le
4s_3(t)$
(since $|x|\ge {2|y|}/{3}$ for all $|y|>M_0$).
In this case, letting  $C_{5.14.1}>4/c_{2.9.2}$, we get from
\eqref{l4-3-11} that
\begin{align*}
& p_D(2Mt,x,y)
=\int_D p_D(Mt,x,u)p_D(Mt,u,y)\,du\le c_{38}\Psi(t,x)\Psi(t,y)\int_D L_2(x,u,t)L_2(y,u,t)\,du\\
&= c_{38}\Psi(t,x)\Psi(t,y)\frac{t}{(1+|x|)^{d+\alpha}}
\frac{t}{(1+|y|)^{d+\alpha}}\\
&\quad \times \left(e^{-2c_{2.9.2}t} \int_{ \{u \in D: |u| \le s_2(t)\}}  du+
\int_{\{u \in D:|u| \ge s_2(t)\}} \frac{(t+\log(2+|u|))^2}{(1+|u|)^{2(d+\alpha-1)}}du \right)\\
&\le c_{39}
\phi(x)\phi(y) t^2 \left(e^{-2c_{2.9.2}t}s_2(t)+(t+\log s_2(t))^2s_2(t)^{-2d-2\alpha+3}\right)\le c_{40}
\phi(x)\phi(y)
e^{-c_{41}t},
\end{align*}
where in the last inequality we used the facts that  $s_2(t)
\le
e^{t/C_{5.14.1}}$
for large $t$
and $C_{5.14.1}>4/c_{2.9.2}$.

(ii) Case 2: $s_2(t)\le |y|\le s_3(t)$ and $|x|\ge 4s_3(t)$.
In this case,
we write
\begin{align*}
p_D(2Mt,x,y)=&
\left(\int_{\{u\in D: |u|<|x|/2\}}
+\int_{\{u\in D:  |x|/2\le |u|\le 2|x|\}}+
\int_{\{u\in D: |u|> 2|x|\}} \right)
\\
&\quad\times
p_D(Mt,x,u)p_D(Mt,u,y)\,du
=:H_1+H_2+H_3.
\end{align*}
By  \eqref{l4-3-7} and \eqref{l4-3-11}, for all $t\ge c_1t_0(y)$,
\begin{align*}
H_1&\le c_{42}\Psi(t,x)\Psi(t,y)\int_{\{u\in D: |u|<|x|/2\}}
L_1(x,u,t)L_2(y,u,t)\,du\\
&\le c_{43}\Psi(t,x)\Psi(t,y) e^{-c_{2.9.2}t}\frac{t}{(1+|x|)^{d+\alpha}}
 \left( \left(
\int_{\{u\in D:  |u| \le s_2(t)\}}+\int_{\{u\in D: |u| \ge s_2(t)\}}\right)
L_2(y,u,t)\,du \right)\\
&\le c_{44}
\frac{\Psi(t,x)t}{(1+|x|)^{d+\alpha}}\frac{\Psi(t,y)t}{(1+|y|)^{d+\alpha}}
e^{-c_{2.9.2}t}\left(e^{-c_{2.9.2}t}s_2(t)+
\int_{s_2(t)}^{\infty}\frac{t+\log(2+s)}{(1+|s|)^{d+\alpha-1}}ds\right)\\
&\le c_{45}
\phi(x)\phi(y)t^2
e^{-c_{2.9.2}t}\left(e^{-c_{2.9.2}t}s_2(t)+
(t+\log s_2(t))s_2(t)^{-d-\alpha+2}\,\right)\le c_{46}
\phi(x)\phi(y)e^{-c_{47}t},
\end{align*}
where in the last inequality we have used the fact that $s_2(t)\le
e^{t/C_{5.14.1}}$ for large $t$
and
we have chosen
$C_{5.14.1} > 4/c_{2.9.2}$.
On the other hand,
\begin{align*}
&H_2\le c_{48}\Psi(t,y)\int_{\{u\in D:  |x|/2\le |u|\le 2|x|\}}
p_D(Mt,x,u)\frac{t}{|u-y|^{d+\alpha}}\,du\le c_{49}\Psi(t,y)\frac{t}{(1+|x|)^{d+\alpha}}\Pp^x(\tau_D> t)\\
&\le c_{50}\Psi(t,x)\Psi(t,y)\frac{t}{(1+|x|)^{d+\alpha}}e^{-c_{51}f(x_1)^{-\alpha}t}\le c_{52}
\phi(x)\phi(y)
e^{-c_{53}t},
\end{align*}
where the first inequality follows from Lemma \ref{l:new} and the fact that
$|y|\le s_3(t)\le |x|/4$ and so
$|u-y|\ge |u|-|y|\ge |x|/4 \ge s_3(t)
\ge t^{1/\alpha}$
for all $u\in D$ with $|x|/2\le |u|\le 2|x|$ (by taking $c_1$
large
enough if necessary), in the second inequality we have used again the fact $|u-y|\ge |x|/4$,
in the third inequality we have applied \eqref{l4-3-4}, and
in the last inequality we have used the facts that $f(x_1)\le f(y_1)$, and
 for every
$y\in D$ and $t\ge c_1t_0(y)$ with large enough $c_1>0$,
\begin{equation}\label{l4-3-9}
e^{-c_{51}f(x_1)^{-\alpha}t/2}\le e^{-c_1c_{51}f(y_1)^{-\alpha}t_0(y)/2}=
\left(\frac{t_0(y)
}{(1+|y|)^{d+\alpha-1}}\right)^{
c_1c_{51}/(2c_{2.9.2})
}
\le \frac{(c_1^{-1}t)^{
c_1c_{51}/(2c_{2.9.2})
}}{(1+|y|)^{d+\alpha}}.
\end{equation}
Furthermore, applying \eqref{l4-3-7} and \eqref{l4-3-11}, we can
easily verify
\begin{align*}
H_3&\le c_{54}\Psi(t,x)\Psi(t,y)\int_{\{u\in D: |u|\ge 2|x|\}}L_2(y,u,t)L_1(x,u,t)\,du\\
&\le c_{55}\frac{\Psi(t,x)}{(1+|x|)^{d+\alpha}}
\frac{\Psi(t,y)}{(1+|y|)^{d+\alpha}} t^2e^{-c_{56}f(x_1)^{-\alpha}t}\int_{\{u\in D: |u|\ge 2|x| \ge 8 s_2(t)\}}\frac{(t+\log (2+|u|))}{(1+|u|)^{d+\alpha-1}}du\\
&\le c_{57}
\phi(x)\phi(y)
e^{-c_{58}t}\int_{2|x|}^{\infty}\frac{1+\log(2+s)}{(1+|s|)^{d+\alpha-1}}
\,ds\le c_{59}
\phi(x)\phi(y)
e^{-c_{58}t},
\end{align*} where in the third inequality we used the facts that
$t^2e^{-c_{56}f(x_1)^{-\alpha}t}
\le c_{60} e^{-c_{58}t}.$
Therefore, according to all estimates for $H_1$, $H_2$ and $H_3$, we
can  obtain the desired conclusion in this case.

(iii) Case 3: $|y|\le s_2(t)$ and $s_2(t)/2\le |x|\le 4s_3(t)$.
According to
\eqref{l4-3-8}, \eqref{e:notess} and \eqref{l4-3-11}, we have
\begin{align*}
p_D(2Mt,x,y)&=\int_D p_D(Mt,x,u)p_D(Mt,u,y)\,du\\
&\le c_{61}\Psi(t,y)e^{-c_{2.9.2}t/2}\Psi(t,x)\frac{t}{(1+|y|)^{d+\alpha}}
\int_D L_2(x,u,t)\,du\le c_{62}
\phi(x)\phi(y)
e^{-c_{63}t},
\end{align*}
where in the last inequality we used the fact
that
$\int_D L_2(x,u,t)\,du\le c_{64}{\Psi(t,x)}{(1+|x|)^{-d-\alpha}}e^{-c_{65}t}$
(that has been verified
in the proof of cases (i) and (ii) above).

(iv) Case 4: $|y|\le s_2(t)$ and $|x|\ge 4s_3(t)$. Define $H_1$, $H_2$ and  $H_3$ as those in case (ii). According to
\eqref{l4-3-7}, we arrive at
\begin{align*}
H_1&\le c_{66}\Psi(t,x)\int_{\{u\in D: |u|\le |x|/2\}}L_1(x,u,t)p_D(Mt,u,y)
\,du\le c_{67}\frac{\Psi(t,x)}{(1+|x|)^{d+\alpha}}te^{-c_{2.9.2}t}\Pp^y(\tau_D>t)\\
&\le c_{68}
\phi(x)
\Psi(t,y) te^{-2c_{2.9.2}t}
\le
c_{69}
\phi(x)\phi(y)
t e^{-c_{2.9.2}t} \le
c_{70}
\phi(x)\phi(y)
e^{-c_{2.9.2}t/2},
\end{align*}
where the third inequality is due to
\eqref{l3-1-0}, and in the fouth inequality we have used
the fact that given $\frac{d+\alpha}{C_{5.14.1}}\le c_{2.9.2}$ (by choosing
$C_{5.14.1}$ large enough if necessary), it holds
\begin{equation}\label{l4-3-12}
\frac{1}{(1+|y|)^{d+\alpha}}\ge c_{71}s_2(t)^{-d-\alpha}\ge c_{71}e^{-\frac{d+\alpha}{C_{5.14.1}}t}\ge c_{71}e^{
-c_{2.9.2}t} \quad \text{for every }|y|\le s_2(t).
\end{equation}

Following
the arguments in case (ii)
and using \eqref{l4-3-8}, \eqref{l4-3-12} instead of \eqref{l4-3-11},
we also can obtain the desired estimates
for $H_2$ and $H_3$.

(v) Case 5: $|y|\le s_2(t)$ and $|x|\le s_2(t)$. According to
\eqref{e:notess} and \eqref{l3-1-0}, we arrive at
\begin{align*}
&p_D(2Mt,x,y)=\int_D p_D(Mt,x,u)p_D(Mt,u,y)\,du\\
&\le c_{72}\Psi(t,x)e^{-c_{2.9.2}t/2}\frac{t}{(1+|x|)^{d+\alpha}}\Pp^y\left(\tau_D>t\right)\le c_{73}
\phi(x)\Psi(t,y)e^{-c_{2.9.2}t}\le c_{74}\phi(x)\phi(y)e^{-c_{75}t},
\end{align*}
where the last step follows from \eqref{l4-3-12}.

Therefore, by all the conclusions above and the definition of $\Psi(t,x)$, we complete the proof.
\end{proof}

Putting Lemmas \ref{l4-4} and \ref{l4-3} together, we obtain
\begin{proposition}\label{p5-6}
Suppose that $g$ is
non-decreasing on $(0,\infty)$. Then,
there exists a constant $c_{\ac{1}}>0$ large enough
such that for all $x,y\in D$ and $t\ge
c_{\ac{1}}t_0(y)\ge 1$,
\begin{align*}
c_{\ac{2}}
\phi(x)\phi(y)
e^{-c_{\ac{3}}t}
\le p_D(t,x,y)\le
c_{\ac{4}}
\phi(x)\phi(y)
e^{-c_{\ac{5}} t},
\end{align*} where $c_{\ac{i}}$ $(i=2,\cdots, 5)$ are independent of
$t$, $x$ and $y$.
\end{proposition}

\section{Further remarks for Theorem \ref{Main}}

Theorem \ref{Main} immediately follows from
Proposition \ref{p2-1}, Proposition \ref{p3-1}, Proposition \ref{l4-3o} and Proposition \ref{p5-6}
in the previous three sections.

Below, we present one more example to further illustrate Theorem \ref{Main}.

\begin{example} Let $f(s)=(1+s)^{-\theta}$ with $\theta>0$ for all $s\in[0,\infty)$.  For any $x,y\in D$, set $t_1(x,y)=(1+(|x|\wedge |y|))^{-\theta \alpha}$ and $t_2(x,y)= (1+(|x|\wedge |y|))^{-\theta \alpha}\log(2+(|x|\wedge |y|))$. Then
there exist positive constants $c_{6.1.1}$, $c_{6.1.2}$ and $c_{6.1.3}$ such that for all $x,y\in D$,
\begin{align*}
&p_D(t,x,y)\asymp  \\
 &\begin{cases}\displaystyle p(t,x,y)\left( \frac{\delta_D(x)^{\alpha/2}}{\sqrt t} \wedge 1  \right)
\left( \frac{\delta_D(y)^{\alpha/2}}{\sqrt t} \wedge 1  \right)\\
\qquad\qquad\qquad\qquad \qquad\qquad\qquad\qquad\qquad\qquad\qquad\qquad \text{for all }0<t\le c_{6.1.1}t_1(x,y); \\[2pt]
\displaystyle p(t,x,y) \frac{\delta_D(x)^{\alpha/2}(1+|x|)^{-\theta\alpha/2}}{t}\frac{\delta_D(y)^{\alpha/2}(1+|y|)^{-\theta\alpha/2}}{t} \exp(-t(1+(|x|\wedge|y|))^{\theta\alpha})&\\
\qquad\qquad\qquad\qquad \qquad\qquad\qquad\qquad\qquad\qquad\qquad\qquad \text{for all } c_{6.1.1}t_1(x,y)< t\le c_{6.1.2}t_2(x,y);\\[2pt]
\displaystyle  \frac{\delta_D(x)^{\alpha/2}(1+|x|)^{-\theta\alpha/2}}{(1+|x|)^{d+\alpha}}\frac{\delta_D(y)^{\alpha/2}
(1+|y|)^{-\theta\alpha/2}}{(1+|y|)^{d+\alpha}}F(t)
,& \\
\qquad\qquad\qquad\qquad \qquad\qquad\qquad\qquad\qquad\qquad\qquad\qquad \text{for all } c_{6.1.2}t_2(x,y)< t\le c_{6.1.3};\\[2pt]
\displaystyle
\frac{\delta_D(x)^{\alpha/2}(1+|x|)^{-\theta\alpha/2}}{(1+|x|)^{d+\alpha}}\frac{\delta_D(y)^{\alpha/2}
(1+|y|)^{-\theta\alpha/2}}{(1+|y|)^{d+\alpha}}
\exp(-t), \\
\qquad\qquad \qquad\qquad\qquad\qquad\qquad\qquad\qquad\qquad\qquad\qquad \text{for all } t> 1,
\end{cases}
\end{align*}where
$F(t)=
(1 \vee t^{-\frac{1+\theta(1-d)}{\theta \alpha}})\I_{ \{\theta \not=\frac{1}{d-1}\}}+
  \log (1+t^{-1})\I_{\{\theta=\frac{1}{d-1}\}}.$
\end{example}
\begin{proof}
For the reference $f$ given in the example, the associated Dirichlet semigroup $(P_t^D)_{t\ge0}$ is intrinsically ultracontractive. We note that for
$s_0(t)$ and $s_1(t)$ defined in the proof of Example \ref{exam},
  $s_0(t)\simeq
 t^{-1/(\theta\alpha)}$ {and} $s_1(t)\simeq
 t^{-1/(\theta\alpha)}\log^{1/(\theta\alpha)}(1+t^{-1}).$
  Hence,
$$
  \int_0^{c_1s_0(t)}f(s)^{d-1}\,ds \simeq
  F(t) \quad
 \text{ and }\quad
 \int_{c_1s_0(t)}^{c_2s_1(t)} f(s)^{d-1}e^{-c_3tf(s)^{-\alpha}}\,ds
 \le c_4F(t).$$ Then, the assertion follows from Theorem \ref{Main}.\end{proof}

Finally, we present one additional remark on the reference function $f$ in Theorem \ref{Main}.

\begin{remark} In the proof of Theorem \ref{Main}(2)(i), the condition $f(s)\ge c(1+s)^{-p}$
is only required to derive upper bounds of $p_D(t,x,y)$ when $c_2 (t_0(x)\vee t_0(y))\le t \le c_3$ involved in the estimate \eqref{e:hk3}.
Indeed, by carefully tracking  the proofs in Section
5.1,
without
the conditions $f(s)\ge c(1+s)^{-p}$ and $\lim_{s\to \infty}f(s)^\alpha\log(2+s)=0$,
one can still obtain  two sided bounds for $p_D(t,x,y)$ in this special time-space region
with the assumption  $\lim_{y\in D, |y|\to \infty}t_0(y)=0$.
For example, if $f(s)=\exp(-c_0(1+s)^\kappa)$ for
 some $c_0>0$ and $\kappa>0$, then for any $x,y\in D$ and $c_2 (t_0(x)\vee t_0(y))\le t \le c_3$,
 \begin{align*}
 p_D(t,x,y)&\asymp \phi(x)\phi(y)\int_{0}^{s_1(t)}
f(s)^{d-1}e^{-tf(s)^{-\alpha}}\,ds \asymp \delta_D(x)^{\alpha/2}f(x_1)^{\alpha/2}
\delta_D(y)^{\alpha/2}f(y_1)^{\alpha/2}.
 \end{align*}
In particular, the term  $(1+|x|)^{-d-\alpha}
 (1+|y|)^{-d-\alpha}$ arising from $\phi(x)\phi(y)$ in \eqref{e:hk3} and respecting the spatial decay disappears in this case, since it is
 absorbed into
 the boundary decay term $f(x_1)^{\alpha/2}f(y_1)^{\alpha/2}$.
  \end{remark}
\ \

\noindent \textbf{Acknowledgements.}
 We thank the anonymous referee for his/her helpful comments.
The research of Xin Chen is supported by the
National Natural Science Foundation of China (Nos.\ 11501361 and 11871338).\
The research of Panki Kim is supported by
 the National Research Foundation of Korea (NRF) grant funded by the Korea government (MSIP)
(No.\ 2016R1E1A1A01941893).\
 The research of Jian Wang is supported by the National
Natural Science Foundation of China (Nos.\ 11831014 and 12071076), the Program for Probability and Statistics: Theory and Application (No.\ IRTL1704) and the Program for Innovative Research Team in Science and Technology in Fujian Province University (IRTSTFJ).

\ \

\end{document}